\documentclass[leqno,12pt]{amsart}
\usepackage{amsfonts,amsmath,amssymb,enumerate,graphicx,psfrag,soul,xcolor}
\usepackage[makeroom,samesize]{cancel}

\newtheorem{theorem}[equation]{Theorem}
\newtheorem{corollary}[equation]{Corollary}
\newtheorem{lemma}[equation]{Lemma}
\newtheorem{proposition}[equation]{Proposition}
\newtheorem{remark}[equation]{Remark}
\newtheorem{definition}[equation]{Definition}

\numberwithin{equation}{section}

\usepackage{hyperref}
\hypersetup{
    colorlinks=true,
    linkcolor= blue,
    citecolor =cyan,
    urlcolor = teal,
}

\begin{document}

\setstcolor{red}

\title[Harmonic functions and the Hardy space $H^1$ in the Dunkl setting]
{Harmonic functions, \\
conjugate harmonic functions \\
and the Hardy space $H^1$ \\
in the rational Dunkl setting}

\author[J.-Ph. Anker, J. Dziuba\'nski, and A. Hejna]{
Jean-Philippe Anker,
Jacek Dziuba\'nski,
Agnieszka Hejna
}

\address{J.-Ph. Anker,
Institut Denis Poisson (UMR 7013),
Universit\'e d'Orl\'eans, Universit\'e de Tours \& CNRS,
B.P. 6759, 45067 Orl\'eans cedex 2, France}
\email{anker@univ-orleans.fr}

\address{J. Dziuba\'nski and A. Hejna, Uniwersytet Wroc\l awski,
Instytut Matematyczny,
Pl. Grunwaldzki 2/4,
50-384 Wroc\l aw,
Poland}
\email{jdziuban@math.uni.wroc.pl}
\email{hejna@math.uni.wroc.pl}



\thanks{
The first and second named authors thank their institutions for reciprocal invitations, where part of the present work was carried out.
Second and third authors supported by the National Science Centre, Poland (Narodowe Centrum Nauki), Grant 2017/25/B/ST1/00599.
}


\begin{abstract}
In this work we extend the theory of the classical Hardy space $H^1$ to the rational Dunkl setting. Specifically, let $\Delta$ be the Dunkl Laplacian on a Euclidean space $\mathbb{R}^N$. On the half-space $\mathbb{R}_+\times\mathbb{R}^N$, we consider systems of conjugate $(\partial_t^2+\Delta_{\mathbf{x}})$-harmonic functions satisfying an appropriate uniform $L^1$ condition. We  prove that the boundary values of such harmonic functions, which constitute the real Hardy space $H^1$, can be characterized in several different ways, namely by means of atoms, Riesz transforms, maximal functions or Littlewood-Paley square functions.
\end{abstract}

\maketitle

\section{Introduction}

Real Hardy spaces on $\mathbb{R}^N$ have their origin in studying holomorphic functions of one variable in the upper half-plane $\mathbb{R}^2_+=\{ z=x+iy\in\mathbb C: y>0\}$. The theorem of Burkholder, Gundy, and Silverstein \cite{BGS} asserts that a real-valued harmonic function $u$ on $\mathbb{R}^2_+$ is the real part of a holomorphic function $F(z)=u(z)+iv(z)$ satisfying the $L^p$ condition
$$\sup_{y>0} \int_{\mathbb{R}} |F(x+iy)|^p\, dx <\infty, \ \ \ 0<p<\infty,$$
if and only if the nontangential maximal function $u^*(x)=\sup_{|x-x'|<y} |u(x'+iy)|$ belongs to $L^p(\mathbb{R})$.
Here $0<p<\infty$.
The $N$-dimensional theory was then developed in Stein and Weiss \cite{SW} and Fefferman and Stein \cite{FS},  where the notion of holomorphy was replaced by conjugate harmonic functions. To be more precise, a system of $C^2$ functions
$$
\boldsymbol{u}(x_0,x_1,\dots,x_N)
=(u_0(x_0,x_1,\dots,x_N),u_1(x_0,x_1,\dots,x_N),\dots,
u_N(x_0,x_1,\dots,x_N)),
$$
where $x_0>0$,  satisfies the generalized Cauchy-Riemann equations if
\begin{equation}\label{classicalCR}
\frac{\partial u_j}{\partial x_i}=\frac{\partial u_i}{\partial x_j}
\quad{\forall\;0\le i\neq j\le N\quad\text{and}\quad}
\sum_{j=0}^N\frac{\partial u_j}{\partial x_j}=0.
\end{equation}
One says that $\boldsymbol u$ has the $L^p$ property if
\begin{equation}\label{Hp}
\sup_{x_0>0} \int_{\mathbb{R}^N} |\boldsymbol u(x_0,x_1,\dots ,x_N)|^p\, dx_1 \dots dx_N<\infty.
\end{equation}
As in the case $N=1$, if $1\!-\!\frac1N<p<\infty$ and $u_0(x_0,x_1,\dots ,x_N)$ is a harmonic function, there is a system $\boldsymbol u=(u_0,u_1,\dots,u_N)$ of $C^2$ functions satisfying \eqref{classicalCR} and \eqref{Hp} if and only if
$$
u_0^*(\mathbf{x})=\sup_{\| \mathbf{x}-\mathbf{x}'\| <x_0}|u_0(x_0,\mathbf{x}^{\prime})|
$$
belongs to $L^p(\mathbb{R}^N)$.
 Here $\mathbf{x}=(x_1,\dots,x_N)\in\mathbb{R}^N$
and similarly $\mathbf{x}'=(x_1',\dots,x_N')$.
Then $u_0$ has a limit $f_0$ in the sense of distributions, as $x_0\searrow0$, and  $u_0$ is the Poisson integral of $f_0$.
It turns out that the set of all distributions obtained in this way, which form the so-called real Hardy space $H^p(\mathbb{R}^N)$, can be equivalently characterized in terms of real analysis (see \cite{FS}), namely by means of various maximal functions, square functions or Riesz transforms. An other important contribution to this theory lies in the atomic decomposition introduced by Coifman \cite{Coifman} and extended to spaces of homogeneous type by Coifman and Weiss \cite{CW}.

The goal of this paper is to study harmonic functions, conjugate harmonic functions, and related Hardy space $H^1$ for the Dunkl Laplacian $\Delta$ (see Section \ref{preliminaries}). We shall prove that these objects have properties  analogous  to the classical ones. In particular the related real Hardy space $H^1_{\Delta}$, which can be defined as the set of boundary values  of $(\partial_t^2+\Delta_{\mathbf x})$-harmonic functions satisfying a relevant $L^1$ property, can be characterized by appropriate maximal functions,  square functions, Riesz transforms or atomic decompositions. Apart from the square function characterization, this was achieved previously in \cite{ABDH} and \cite{Dz} in the one-dimensional case, as well as in the product case.

\section{Statement of the results}\label{preliminaries}

In this section we first collect basic facts concerning Dunkl operators, the Dunkl Laplacian, and the corresponding heat and Poisson semigroups. For details we refer the reader to \cite{Dunkl}, 
\cite{Roesler3} and \cite{Roesler-Voit}. Next we state our main results.

In the Euclidean space $\mathbb{R}^N$, equipped with a scalar product $\langle\mathbf{x},\mathbf{y}\rangle$,
the reflection $\sigma_\alpha$ with respect to the hyperplane $\alpha^\perp$ orthogonal to a nonzero vector $\alpha$ is given by
$$
\sigma_\alpha (\mathbf{x})
=\mathbf{x}-2\frac{\langle \mathbf{x},\alpha\rangle}{\|\alpha\|^2}\alpha.
$$
A finite set $R\subset\mathbb{R}^N\setminus\{0\}$ is called a {\it root system} if $\sigma_\alpha (R )=R$ for every $\alpha\in R$. We shall consider normalized reduced root systems, that is, $\|\alpha\|^2=2$  for every $\alpha\in R$. The finite group $G$ generated by the reflections $\sigma_\alpha$ is called the {\it Weyl group} ({\it reflection group}) of the root system. We shall denote by $\mathcal{O}(\mathbf{x})$, resp.~$\mathcal{O}(B)$ the $G$-orbit of a point $\mathbf{x}\in\mathbb{R}^N$, resp.~a subset $B\subset\mathbb{R}^N$. A \textit{multiplicity function\/} is a $G$-invariant function $k:R\to\mathbb C$, which will be fixed and $\ge0$ throughout this paper.

Given a root system $R$ and a multiplicity function $k$, the \textit{Dunkl operators} $T_\xi$ are  the following deformations of directional derivatives $\partial_\xi$ by difference operators\,:
\begin{align*}
T_\xi f(\mathbf{x})
&{=\partial_\xi f(\mathbf{x})+\sum_{\alpha\in R}\frac{k(\alpha)}2\langle\alpha,\xi\rangle\frac{f(\mathbf{x})\!-\!f(\sigma_\alpha(\mathbf{x}))}{\langle\alpha,\mathbf{x}\rangle}}\\
&=\partial_\xi f(\mathbf{x})+\hspace{-1mm}\sum_{\alpha\in R^+}\hspace{-1mm}k(\alpha)\langle\alpha,\xi\rangle\frac{f(\mathbf{x})\!-\!f(\sigma_\alpha(\mathbf{x}) )}{\langle\alpha,\mathbf{x}\rangle}.
\end{align*}
Here $R^+$ is any fixed positive subsystem of $R$. The Dunkl operators $T_\xi$, which were introduced in \cite{Dunkl}, commute pairwise and are skew-symmetric with respect to the $G$-invariant measure $dw(\mathbf{x})=w (\mathbf{x})\,d\mathbf{x}$, where
\begin{equation*}
w (\mathbf{x})
=\prod_{\alpha\in R}|\langle \alpha, \mathbf{x}\rangle|^{k(\alpha)}
{=\hspace{-1mm}\prod_{\alpha\in R^+}\hspace{-1mm}|\langle \alpha, \mathbf{x}\rangle|^{2k(\alpha)}}.
\end{equation*}
{Set $T_j=T_{e_j}$, where $\{e_1,\dots,e_N\}$ is the canonical basis of $\mathbb{R}^N$}. The \textit{Dunkl Laplacian} associated with $R$  and $k$ is the differential-difference operator ${\Delta}=\sum_{j=1}^nT_{j}^2$, which  acts on $C^2$ functions by
$$
{\Delta}f(\mathbf{x})
{=\Delta_{\text{eucl}}f(\mathbf{x})
+\sum_{\alpha\in R}k(\alpha)\delta_\alpha f(\mathbf{x})}
=\Delta_{\text{eucl}}f(\mathbf{x})
+2\sum_{\alpha\in R^+}k(\alpha)\delta_\alpha f(\mathbf{x}),
$$
{where}
$$
\delta_\alpha f(\mathbf{x})
=\frac{\partial_\alpha f(\mathbf{x})}{\langle\alpha,\mathbf{x}\rangle}-\frac{f(\mathbf{x})-f(\sigma_\alpha(\mathbf{x}))}{\langle \alpha,\mathbf{x}\rangle^2}.
$$
The operator $\Delta$ is essentially self-adjoint on $L^2(dw)$
 (see for instance \cite[Theorem\;3.1]{AH})
and generates the heat semigroup
\begin{equation}\label{heat_semigroup}
{H_tf(\mathbf{x})=}e^{t{\Delta}}f(\mathbf{x})=\int_{\mathbb{R}^N}
h_t(\mathbf{x},\mathbf{y})f(\mathbf{y})\,dw (\mathbf{y}){.}
\end{equation}
Here the heat kernel $h_t(\mathbf{x},\mathbf{y})$ is a $C^\infty$ function {in} all variables $t>0$, $\mathbf{x}\in\mathbb{R}^N$, $\mathbf{y}\in\mathbb{R}^N$, which satisfies
$$
h_t(\mathbf{x},\mathbf{y})=h_t(\mathbf{y},\mathbf{x})
{>0\quad\text{and}\quad}
\int_{\mathbb{R}^N} h_t(\mathbf{x},\mathbf{y})\,dw(\mathbf{y})=1.
$$
{Notice that} \eqref{heat_semigroup} defines a strongly continuous semigroup of linear contractions on $L^p({dw})$, {for every} $1\leq p<\infty$.

The  Poisson semigroup $P_t=e^{-t\sqrt{-{\Delta}}}$ is given by the subordination formula
\begin{equation}\label{subordination}
P_tf(\mathbf{x})= \pi^{-1/2}\int_0^{\infty} e^{-u}  \exp\Big(\frac{t^2}{4u} {\Delta}\Big)f(\mathbf{x}) \frac{du}{\sqrt{u}}
\end{equation}
and solves the boundary value problem{
$$\begin{cases}
\,(\partial_{t}^2+\Delta_{\mathbf{x}})\,u(t,\mathbf{x})=0\\
\;u(0,\mathbf{x})=f(\mathbf{x})
\end{cases}$$
}in the upper half{-}space $\mathbb{R}_+^{{1+N}}=(0,\infty)\times\mathbb{R}^N\subset\mathbb{R}^{1+N}$. Let  $e_0=(1,0,{\dots},0)$, $e_1=(0,1,{\dots},0)$,{\dots}, $e_N=(0,0,{\dots},1)$ be the canonical basis in $\mathbb{R}^{1+N}$. In order to unify our notation we shall denote the variable $t$ by $x_0$ and set $T_0=\partial_{e_0}$.

Our goal is to study real harmonic functions of the operator
\begin{equation}\label{operatorL}
\mathcal L=T_0^2+{\Delta}=\sum_{j=0}^N T_{j}^2.
\end{equation}
The operator $\mathcal L$ {is} the Dunkl {Laplacian} associated with the root system $R$, considered as a subset of $\mathbb{R}^{1+N}$ under the {embedding} $R\subset\mathbb{R}^N\hookrightarrow\mathbb{R}\times\mathbb{R}^N$.

We say that a system{
$$
\mathbf{u}=(u_0,u_1,\dots,u_N),
\text{ \,where \;}
u_j=u_j(x_0,\underbrace{x_1,\dots,x_N}_{\mathbf{x}})
\;\;\forall\;0\le j\le N,
$$
}of $C^1$ {real} functions on $\mathbb{R}_+^{{1+N}}$
satisfies the generalized Cauchy-Riemann equations  if\begin{equation}\label{C-R}\begin{cases}
\,{T_iu_j=T_ju_i\quad\forall\;0\le i\ne j\le N},\\
\,\sum_{j=0}^N T_j u_j=0.
\end{cases}\end{equation}
{In this case each component $u_j$ is $\mathcal{L}$-harmonic, i.e., $\mathcal{L}u_j=0$.}

We say that a system $\mathbf{u}$ of $C^2$ real $\mathcal{L}$-harmonic functions {on $\mathbb{R}_+^{{1+N}}$} belongs to the Hardy space {$\mathcal{H}^{1}$} if it satisfies both \eqref{C-R}  and the $L^1$  condition
\begin{equation*}
 \|\mathbf{u}\|_{\mathcal{H}^{1}}=\sup_{x_0>0}\big\||\mathbf{u}(x_0,\cdot)|\big\|_{L^{1}({dw})}
 =\sup_{x_0>0}\int_{\mathbb{R}^N}|\mathbf{u}(x_0,\mathbf{x})|\,{dw}(\mathbf{x})<\infty,
\end{equation*}
where $|\mathbf{u}(x_0,\mathbf{x})|=\Big(\sum_{j=0}^n |u_j(x_0,\mathbf{x})|^2\Big)^{1\slash 2}$.

We are now {ready} to state our first main result.

\begin{theorem}\label{main1}
Let $u_0$ be {a} $\mathcal L$-harmonic function in the upper half{-}space $\mathbb{R}_+^{1+N}$. Then there are $\mathcal L$-harmonic functions $u_j$ $(j=1,{\dots},N)$ such that $\mathbf{u}=(u_0,u_1,{\dots},u_N)$ belongs to {$\mathcal{H}^1$} if and only if the nontangential maximal function
\begin{equation}\label{star}
u_{{0}}^*(\mathbf{x})=\sup\nolimits_{\,{\|}\mathbf{x}'-\mathbf{x}{\|}<x_0} |u_0(x_0,\mathbf{x}')|
\end{equation}
belongs to $L^1({dw})$.
In this case, the norms $\|u_0^{*}\|_{L^1({dw})}$
and $\|\mathbf{u}\|_{\mathcal{H}^1}$ are moreover equivalent.
\end{theorem}

{If $\mathbf u\in\mathcal H^1$, we shall prove that} the limit  $f(\mathbf{x})=\lim_{x_0\to 0} u_0(x_0, \mathbf{x})$ exists in {$L^1(dw)$} and $u_0(x_0,\mathbf{x})=P_{x_0}f(\mathbf{x})$. {This leads to consider the so-called real Hardy space}
$$
H^1_{{\Delta}}
=\{f(\mathbf{x})=\lim_{x_0\to 0}u_0(x_0,\mathbf{x})\,{|}\,
(u_0,u_1,{\dots},u_N)\in{\mathcal{H}^1}\},
$$
{equipped with the norm}
$$ \| f\|_{H^1_{{\Delta}}}
=\|(u_0,u_1,{\dots},u_N)\|_{\mathcal{H}^1}.
$$
Let us denote by
$$
\mathcal{M}_Pf(\mathbf{x})=\sup\nolimits_{\,{\|}\mathbf{x}-\mathbf{x}'{\|}<t}\,\bigl|{P_t}f(\mathbf{x}')\bigr|
$$
the nontangential maximal function associated with the {Poisson semigroup $P_t=e^{-t\sqrt{-\Delta}}$}. According to Theorem \ref{main1}, {$H^1_{{\Delta}}$
coincides with the following subspace of $L^1(dw)$\,:}
$$
H^1_{{\rm max},P}=\{f\in L^1({dw})\,{|}\,\|\mathcal{M}_Pf\|_{L^1({dw})}<\infty\}.
$$
{Moreover, the norms $\|f\|_{H^1_\Delta}$ and $\|\mathcal{M}_Pf\|_{L^1(dw)}$ are equivalent.}
\smallskip

Our task is to prove {other} characterizations of $H^1_{{\Delta}}$ by means of real analysis.
\smallskip

\noindent
\textbf{A. Characterization by the heat maximal function.}
Let
$$
\mathcal{M}_{H}f(\mathbf{x})=
\sup\nolimits_{\,{\|}\mathbf{x}-\mathbf{x}'{\|}^2<t}|{H_t}f(\mathbf{x}')|
$$
be the nontangential maximal function associated with the  heat semigroup {$H_t=e^{t\Delta}$} and {set}
$$
H^1_{{\rm max},H}=\{f\in L^1({dw})\,{|}\,\|\mathcal{M}_{H}f\|_{L^1({dw})}<\infty\}.
$$

\begin{theorem}\label{main2}
{The spaces $H^1_{\Delta}$ and $H^1_{{\rm max},H}$ coincide and the corresponding norms $\|f\|_{H^1_{\Delta}}$ and $\|\mathcal{M}_{H}f\|_{L^1(dw)}$ are equivalent.}
\end{theorem}

\noindent
\textbf{B. Characterization by square functions.}
For every $1\leq p\leq \infty$, the operators
\,$Q_t=t\sqrt{-{\Delta}}e^{-t\sqrt{-{\Delta}}}$
are uniformly bounded on $L^p({dw})$
(this is a consequence of the estimates \eqref{Gauss}, \eqref{DtDxDyPoisson} and \eqref{Poisson_low_up}).
{Consider} the square function
\begin{equation}\label{square}
Sf(\mathbf{x})=\left(\iint_{{\|}\mathbf{x}-\mathbf{y}{\|}<t}|Q_tf(\mathbf{y})|^2
\frac{dt\,{dw}(\mathbf{y})}{t\,{w}(B(\mathbf{x},t))}\right)^{1\slash 2}
\end{equation}
{and the space}
$$
H^1_{\rm square}=\{ f\in L^1({dw})\,{|}\,\| Sf\|_{L^1({dw})}<\infty\}.
$$
\begin{theorem}\label{main5}
{The spaces $H^1_{\Delta}$ and $H^1_{\rm square}$ coincide
and the corresponding norms $\| f\|_{H^1_{\Delta}}$ and $\|Sf\|_{L^1(dw)}$ are equivalent.}
\end{theorem}

\noindent
\begin{remark}
\normalfont
The square function characterization of $H^1_{{\Delta}}$ is {also} valid {for} $Q_t=t^2{\Delta}\,e^{\,t^2{\Delta}}$.
\end{remark}

\noindent
\textbf{C. Characterization by Riesz transforms.}
{T}he Riesz transforms{, which are defined} in the Dunkl setting by
$$
R_jf=T_j(-{\Delta})^{-1\slash 2}f
$$
(see Section \ref{SectionRiesz}), {are} bounded operators on $L^p({dw})${, for every} $1<p<\infty$ (cf. \cite{AS}).  {In the limit case $p=1$, they turn out to be} bounded {operators} from $H^1_{\Delta}$ {into $H^1_{\Delta}\subset L^1({dw})$. This leads to consider the space}
$$
H^1_{\rm Riesz}=\{f\in L^1({dw})\,{|}\,\| R_j f\|_{L^1(w)}<\infty\;{\forall\;1\le j\le N}\}.
$$

\begin{theorem}\label{main4}
{The spaces $H^1_{\Delta}$ and $H^1_{\rm Riesz}$ coincide and the corresponding norms $\|f\|_{H^1_{\Delta}}$ and
$$
\|f\|_{H^1_{\rm Riesz}}:=\|f\|_{L^1({dw})}+\sum\nolimits_{j=1}^N\|R_jf\|_{L^1({dw})}.
$$
are equivalent.}
\end{theorem}

\noindent
\textbf{D. Characterization by atomic decompositions.}
 Let us define atoms in the spirit of \cite{HMMLY}.
Given a Euclidean ball $B$ in $\mathbb{R}^N$,
we shall denote its radius by $r_B$ and its $G$-orbit by $\mathcal{O}(B)$.
For any positive integer $M$,
let $\mathcal{D}({\Delta}^M)$ be the domain of ${\Delta}^M$
as an (unbounded) operator on $L^2({dw})$.

\begin{definition}\label{def-atom}\normalfont
Let $1<q\leq\infty$ and {let} $M$ be a positive integer.
A function {$a\in L^2(dw)$} is said to be  a $(1,q,M)$-\textit{atom} if
{there exist} $ b\in\mathcal{D}({\Delta}^M)$ and a ball {$B$} such that
\begin{itemize}
\item
$a={\Delta}^Mb$\,,
\vspace{.5mm}
\item
$\text{\rm supp}\,{(}{\Delta}^\ell b{)}\subset\mathcal{O}(B)$ \,{\,$\forall\;0\le\ell\le M$,}
\item
$\| (r_{{B}}^2{\Delta})^{\ell}b\|_{L^q({dw})} \leq r^{2M}{w}(B)^{\frac{1}{q}-1}$
{\,$\forall\;0\le\ell\le M$.}
\end{itemize}
\end{definition}

\begin{definition}\normalfont
A function $f$ belongs to $H^1_{(1,q,M)}$ if there are  $\lambda_j\in\mathbb C$ and $(1,q,M)$-atoms $a_j$ such that
\begin{equation}\label{AtomicRepresentation}
f=\sum\nolimits_{j}\lambda_j\,a_j\,.
\end{equation}
{In this case,} set
$$
\|f\|_{H^1_{(1,q,M)}}=\inf\,\Bigl\{\,\sum\nolimits_j|\lambda_j|\,\Bigr\}\,,
$$
where the infimum is taken over all representations {\eqref{AtomicRepresentation}}.
\end{definition}

\begin{theorem}\label{main3} The spaces $H^1_{{\Delta}}$ and $H^1_{(1,q,M)}$ coincide and the corresponding norms are equivalent.
\end{theorem}

{In the one-dimensional case and in the product case considered in \cite{ABDH} and \cite{Dz}, the Dunkl kernel can be expressed explicitly in terms of classical special functions (Bessel functions or the confluent hypergeometric function).
Thus its behavior is fully understood, and consequently all kernels involved in the definitions above. In the general case considered in the present paper, no such information is available. Therefore an essential part of our work consists in deriving estimates of the Dunkl kernel, of the heat kernel, of the Poisson kernel, and of their derivatives (see the end of Section \ref{SectionDunkl}, Section \ref{SectionHeat} and Section \ref{SectionPoisson}).} These estimates, which are in a spirit of analysis on spaces of homogeneous type, allow us to build {up} the theory of the Hardy space $H^1$ in the Dunkl setting.

{Let us briefly describe the organization of the proofs of the results.
Clearly, $H^1_{(1,q_1,M)}\subset H^1_{(1,q_2,M)}$ for $1<q_2\leq q_1\leq \infty$.
The proof $(u_0,u_1,{\dots},u_N)\in\mathcal{H}^1_k$ implies $u_0^*\in L^1({dw})$, which is actually the inclusion $H^1_{{\Delta}}\subset H^1_{{\rm max},P}$,  is presented in Section \ref{HarmFunt}, see Proposition \ref{PropMain1Part1}. The proof is based on $\mathcal L$-subharmonicity of certain function constructed from $\mathbf{u}$ (see Section \ref{SectionSub}). The converse to Proposition \ref{PropMain1Part1} is proved at the very end of  Section \ref{Atomic}. Inclusions: $H^1_{{\Delta}}\subset H^1_{{\rm Riesz}}\subset H^1_{{\Delta}}$ are shown in Section \ref{SectionRiesz}.
Further, $H^1_{(1,q,M)} \subset H^1_{\rm Riesz}$ for $M$ large is proved in Section \ref{Sect9}.
The proofs of  $H^1_{{\rm max},{H}}\subset H^1_{{\rm max, P}} \subset H^1_{(1,\infty, M)}$ for every $M\geq 1$ are presented in Section \ref{Atomic}.
Inclusion: $ H^1_{(1,q,M)}\subset H^1_{{\rm max},{H}}$ for every $M\geq 1$ is proved in Section \ref{SectionMax}.
Finally $H^1_{(1,2,M)}\subset H^1_{\rm square}\subset H^1_{(1,2,M)}$ are  established in Section \ref{SectionSquare}.
}

\section{{Dunkl kernel, }Dunkl transform and Dunkl translations}\label{SectionDunkl}

The purpose of this section is to collect some facts {about} the Dunkl kernel, the Dunkl transform  and  Dunkl translations.  General references are \cite{Dunkl}, \cite{Roesler3}, \cite{Roesler-Voit}. At the end of this section we shall derive estimates for the Dunkl translations of radial functions. These estimates will be used later to obtain bounds for the heat {kernel} and {for} the Poisson kernel, as well as for their derivatives, and furthermore upper and lower  bounds for the Dunkl kernel.

We {begin with} some notation.
{Given a root} system $R$ in $\mathbb{R}^N$ {and} a multiplicity function $k \ge0${,} let
\begin{equation}\label{gamma}
\gamma=\sum\nolimits_{\alpha\in R^+}\!k(\alpha)
\quad{\text{and}}\quad\mathbf{N}=N\!+2\gamma.
\end{equation}
The number $\mathbf{N}$ is {called} the homogeneous dimension, {because of the scaling property}
$$ {w}(B(t\mathbf{x}, tr))=t^{\mathbf{N}}{w}(B(\mathbf{x},r)).
$$
Observe that {(\footnote{\,The symbol $\sim$ between two positive expressions means that their ratio is bounded between two positive constants.})}
$$
w(B(\mathbf{x},r))\sim r^{N}\prod_{\alpha\in R}(\,|\langle\alpha,\mathbf{x}\rangle|+r\,)^{k(\alpha)}{.}
$$
{Thus the measure} {$w$} is doubling, that is, there is a constant $C>0$ such that
$$
w(B(\mathbf{x},2r))\leq C\,w(B(\mathbf{x},r)).
$$
Moreover, {there exists a constant $C\ge1$ such that,
for every $\mathbf{x}\in\mathbb{R}^N$ and for every $R\ge r>0$,}
\begin{equation}\label{growth}
C^{-1}\Big(\frac{R}{r}\Big)^{N}\leq\frac{{w}(B(\mathbf{x},R))}{{w}(B(\mathbf{x},r))}\leq C \Big(\frac{R}{r}\Big)^{\mathbf N}.
\end{equation}
{Set}
$$
V(\mathbf{x},\mathbf{y},t)=\max{\bigl\{}w(B(\mathbf{x},t)), w(B(\mathbf{y}, t)){\bigr\}}.
$$
{Finally, l}et $d(\mathbf{x},\mathbf{y})=\min_{\,\sigma\in G}{\|}\mathbf{x}-{\sigma(}\mathbf{y}{)}{\|}$ denote the distance {between two $G$-}orbits $\mathcal O(\mathbf{x})$ and $\mathcal O(\mathbf{y})$. Obviously, $\{{\mathbf{y}\!\in\!\mathbb{R}^N\,|\,d(\mathbf{y},\mathbf{x})}<r\}=\mathcal{O}(B({\mathbf{x}},r))$ and
$$
{w}(B({\mathbf{x}},r))\leq{w}(\mathcal{O}(B({\mathbf{x}},r)))\leq |G|\,{w}(B({\mathbf{x}},r)).
$$

{\bf Dunkl kernel.} For {fixed} ${\mathbf{x}}\in\mathbb{R}^N${,} the \textit{Dunkl kernel} ${\mathbf{y}\longmapsto}{E}(\mathbf{x},\mathbf{y})$ is {the} unique solution {to} the system
$$\begin{cases}
\,T_\xi f=\langle\xi,{\mathbf{x}}\rangle\,f\quad\forall\;\xi\in\mathbb{R}^N,\\
\;f(0)=1.
\end{cases}$$
The following integral formula was {obtained} by R\"osler \cite{Roesle99}\,{:}
\begin{equation}\label{Rintegral}
{E}(\mathbf{x},\mathbf{y})=\int_{\mathbb{R}^N} e^{\langle\eta, \mathbf{y}\rangle} d\mu_{\mathbf{x}}(\eta),
\end{equation}
where ${\mu_{\mathbf{x}}}$ is a probability measure supported {in the convex hull} $\operatorname{conv}\mathcal{O}(\mathbf{x})$ {of the $G$-orbit of $\mathbf{x}$}. The function ${E}(\mathbf{x},\mathbf{y})$, which generalizes the exponential function $e^{\langle\mathbf{x},\mathbf{y}\rangle}${, extends holomorphically to} $\mathbb{C}^N\times\mathbb{C}^N$ {and satisfies the following properties\,:}
\begin{itemize}
\item[$\bullet$]
$\,{E}(0,\mathbf{y})=1\quad{\forall}\;\mathbf{y}\in\mathbb{C}^N$,
\item[$\bullet$]
$\,{E}(\mathbf{x},\mathbf{y})={E}(\mathbf{y},\mathbf{x})\quad{\forall}\;\mathbf{x},\mathbf{y}\in \mathbb{C}^N${,}
\item[$\bullet$]
$\,{E}(\lambda\mathbf{x},\mathbf{y})={E}(\mathbf{x},\lambda\mathbf{y})\quad{\forall}\;\lambda\in\mathbb{C},\;{\forall}\;\mathbf{x},\mathbf{y}\in \mathbb{C}^N${,}
\item[$\bullet$]
$\,{E}(\sigma{(}\mathbf{x}{)},\sigma{(}\mathbf{y}{)})={E}(\mathbf{x},\mathbf{y})\quad{\forall}\;\sigma\in G,\;{\forall}\;\mathbf{x},\mathbf{y}\in \mathbb{C}^N${,}
\item[$\bullet$]
$\,\overline{{E}(\mathbf{x},\mathbf{y})}={E}(\bar{\mathbf{x}},\bar{\mathbf{y}})\quad{\forall}\;\mathbf{x},\mathbf{y}\in\mathbb{C}^N${,}
\item[$\bullet$]
$\,{E}(\mathbf{x},\mathbf{y}){>0}\quad{\forall}\;\mathbf{x},\mathbf{y}\in\mathbb{R}^N${,}
\item[$\bullet$]
$\,|{E}( i\mathbf{x},\mathbf{y})|\leq1\quad{\forall}\;\mathbf{x},\mathbf{y}\in\mathbb{R}^N${,}
\item[$\bullet$]
$\,|\partial_{\mathbf{y}}^\alpha{E}(\mathbf{x},\mathbf{y})|\leq{\|}\mathbf{x}{\|}^{|\alpha|}\max_{\,\sigma\in G}e^{\,{\operatorname{Re}}\,\langle \sigma (\mathbf{x}),\mathbf{y}\rangle}\quad{\forall\;\alpha\in\mathbb{N}^N\,(\footnotemark),\;\forall}\;\mathbf{x}\in\mathbb{R}^N,\;{\forall}\;\mathbf{y}\in\mathbb{C}^N$.
\end{itemize}
\smallskip

\footnotetext{
\,Here and subsequently, $\mathbb{N}$ denotes the set of nonnegative integers.}

{\bf Dunkl transform.} The {\it Dunkl transform}  is defined on $L^1({dw})$ {by}
$$
\mathcal{F} f(\xi)=c_k^{-1}\int_{\mathbb{R}^N}f(\mathbf{x})E(\mathbf{x},-i\xi)\, {dw}(\mathbf{x}),
$$
{where}
$$
c_k=\int_{\mathbb{R}^N}e^{-\frac{{\|}\mathbf{x}{\|}^2}2}\,{dw}(\mathbf{x}){>0}\,.
$$
\smallskip

\noindent
{The following properties hold for} the Dunkl transform (see \cite{dJ}, \cite{Roesler-Voit}):
\begin{itemize}
\item[{$\bullet$}]
The Dunkl transform is a topological automorphisms of the Schwartz {space} $\mathcal{S}(\mathbb{R}^N)$.
\item[{$\bullet$}]
$\,${(\textit{Inversion formula\/})} For every {$f\in\mathcal{S}(\mathbb{R}^N)$ and actually for every} $f\in L^1({dw})$ such that $\mathcal{F}f\in L^1({dw})${,} we have
$$
f(\mathbf{x})=(\mathcal{F})^2f(-\mathbf{x})
\qquad{\forall\;\mathbf{x}\in\mathbb{R}^N}.
$$
\item[{$\bullet$}]
$\,${(\textit{Plancherel Theorem\/})}
The Dunkl transform exten{d}s to an isometric automorphism of $L^2({dw})$.
\item[{$\bullet$}]
The Dunkl transform of a radial function is {again} a radial function.
\item[{$\bullet$}]
{(\textit{Scaling\/})}
{For $\lambda\in\mathbb{R}^*$, we have
$$
\mathcal{F}(f_\lambda)(\xi)=\mathcal{F}f(\lambda\xi),
$$
where $f_\lambda(\mathbf{x})=|\lambda|^{-\mathbf{N}}f(\lambda^{-1}\mathbf{x})$.}
\item[{$\bullet$}]
{Via the Dunkl transform, the Dunkl operator $T_\eta$ corresponds to the multiplication by $\pm i\,\langle\eta,\cdot\,\rangle$. Specifically,
$$\begin{cases}
\,\mathcal{F}(T_\eta f)=i\,\langle\eta,\cdot\,\rangle\,\mathcal{F}f,\\
\,T_\eta(\mathcal{F}f)=-i\,\mathcal{F}(\langle\eta,\cdot\,\rangle f).
\end{cases}$$}
In particular, $\mathcal{F}({\Delta}f)(\xi)=-{\|}\xi{\|}^2\mathcal{F}f(\xi)$.
\end{itemize}
\smallskip

{\bf Dunkl translation{s} and Dunkl convolution.}
The {\it Dunkl translation\/} $\tau_{\mathbf{x}}f$ of a function $f\in\mathcal{S}(\mathbb{R}^N)$ by $\mathbf{x}\in\mathbb{R}^N$ is defined by
\begin{equation}\label{translation}
\tau_{\mathbf{x}} f(\mathbf{y})=c_k^{-1} \int_{\mathbb{R}^N}{E}(i\xi,\mathbf{x})\,{E}(i\xi,\mathbf{y})\,\mathcal{F}f(\xi)\,{dw}(\xi).
\end{equation}
{Notice the following properties of Dunkl translations\,:
\begin{itemize}
\item[$\bullet$]
each translation $\tau_{\mathbf{x}}$ is a continuous linear map of $\mathcal{S}(\mathbb{R}^N)$ into itself, which extends to a contraction on $L^2({dw})$,
\item[$\bullet$]
(\textit{Identity\/})
$\tau_0=I$,
\item[$\bullet$]
(\textit{Symmetry\/})
$\tau_{\mathbf{x}}f(\mathbf{y})=\tau_{\mathbf{y}}f(\mathbf{x})\quad\forall\;\mathbf{x},\mathbf{y}\in\mathbb{R}^N,\;\forall\;f\in\mathcal{S}(\mathbb{R}^N)$,
\item[$\bullet$]
(\textit{Scaling\/})
$\tau_{\mathbf{x}}(f_\lambda)=(\tau_{\lambda^{-1}\mathbf{x}}f)_\lambda\quad\forall\;\lambda>0\,,\;\forall\;\mathbf{x}
\in\mathbb{R}^N,\;\forall\;f\in\mathcal{S}(\mathbb{R}^N)$,
\item[$\bullet$]
(\textit{Commutativity\/})
the Dunkl translations $\tau_{\mathbf{x}}$ and the Dunkl operators $T_\xi$ all commute,
\item[$\bullet$]
(\textit{Skew--symmetry\/})
\begin{equation*}\quad
\int_{\mathbb{R}^N}\!\tau_{\mathbf{x}}f(\mathbf{y})\,g(\mathbf{y})\,dw(\mathbf{y})=\int_{\mathbb{R}^N}f(\mathbf{y})\,\tau_{-\mathbf{x}}g(\mathbf{y})\,dw(\mathbf{y})
\quad\forall\;\mathbf{x}\in\mathbb{R}^N,\;\forall\;f,g\in\mathcal{S}(\mathbb{R}^N).
\end{equation*}
\end{itemize}
The latter formula allows us to define the Dunkl translations $\tau_{\mathbf{x}}f$ in the distributional sense for $f\in L^p({dw})$ with $1\leq p\leq \infty$. In particular,
\begin{equation*}\quad
\int_{\mathbb{R}^N}\!\tau_{\mathbf{x}}f(\mathbf{y})\,dw(\mathbf{y})=\int_{\mathbb{R}^N}f(\mathbf{y})\,dw(\mathbf{y})
\quad\forall\;\mathbf{x}\in\mathbb{R}^N,\;\forall\;f\in\mathcal{S}(\mathbb{R}^N).
\end{equation*}
Finally, notice that $\tau_{\mathbf{x}}f$ is given by \eqref{translation},
if $f\in L^1({dw})$ and $\mathcal{F}f\in L^1({dw})$.}
\smallskip

{The \textit{Dunkl convolution\/} of two reasonable functions (for instance Schwartz functions) is defined by
$$
(f*g)(\mathbf{x})=c_k\,\mathcal{F}^{-1}[(\mathcal{F}f)(\mathcal{F}g)](\mathbf{x})=\int_{\mathbb{R}^N}(\mathcal{F}f)(\xi)\,(\mathcal{F}g)(\xi)\,E(\mathbf{x},i\xi)\,dw(\xi)\quad\forall\;\mathbf{x}\in\mathbb{R}^N
$$
or, equivalently, by}
$$
{(}f{*}g{)}(\mathbf{x})=\int_{\mathbb{R}^N}f(\mathbf{y})\,\tau_{\mathbf{x}}g(-\mathbf{y})\,{dw}(\mathbf{y})\qquad{\forall\;\mathbf{x}\in\mathbb{R}^N}.
$$

\textbf{Dunkl translations of radial functions.}
{The following specific formula was obtained by R\"osler \cite{Roesler2003}
for the Dunkl translations of (reasonable) radial functions $f({\mathbf{x}})=\tilde{f}({\|\mathbf{x}\|})$\,:}
\begin{equation}\label{translation-radial}
\tau_{\mathbf{x}}f(-\mathbf{y})=\int_{\mathbb{R}^N}{(\tilde{f}\circ A)}(\mathbf{x},\mathbf{y},\eta)\,d\mu_{\mathbf{x}}(\eta){\qquad\forall\;\mathbf{x},\mathbf{y}\in\mathbb{R}^N.}
\end{equation}
{Here}
\begin{equation*}
A(\mathbf{x},\mathbf{y},\eta)=\sqrt{{\|}\mathbf{x}{\|}^2+{\|}\mathbf{y}{\|}^2-2\langle \mathbf{y},\eta\rangle}=\sqrt{{\|}\mathbf{x}{\|}^2-{\|}\eta{\|}^2+{\|}\mathbf{y}-\eta{\|}^2}
\end{equation*}
and $\mu_{\mathbf{x}}$ is {the} probability measure {occurring} {in} \eqref{Rintegral}{, which is supported in} $\operatorname{conv}\mathcal{O}(\mathbf{x})$.

In the remaining part of this section{,} we shall derive estimates for the Dunkl translations of {certain} radial functions. {Let us begin with} the following  {elementary estimates} (see, e.g., \cite{AS}){, which hold for $\mathbf{x},\mathbf{y}\in\mathbb{R}^N$ and $\eta\in\operatorname{conv}\mathcal{O}(\mathbf{x})$\,:
\begin{equation}\label{A1}
A(\mathbf{x},\mathbf{y},\eta)\ge d(\mathbf{x},\mathbf{y})
\end{equation}
and}
\begin{equation}\label{A2}\begin{cases}
\,{\|}\nabla_{\mathbf{y}}\{A(\mathbf{x},\mathbf{y},\eta)^2\}{\|}\le{2}\,A(\mathbf{x},\mathbf{y},\eta),\\
\,|\hspace{.25mm}\partial^\beta_{\mathbf{y}}\{A(\mathbf{x},\mathbf{y},\eta)^2\}|\le{2}
&{\text{if \,}|\beta|=2\hspace{.25mm},}\\
\,\partial^\beta_{\mathbf{y}}\{A(\mathbf{x},\mathbf{y},\eta)^2\}=0
&{\text{if \,}|\beta|>2\hspace{.25mm}.}
\end{cases}\end{equation}
{Hence
\begin{equation}\label{A3}
{\|}\nabla_{\mathbf{y}}A(\mathbf{x},\mathbf{y},\eta){\|}\le{1}
\end{equation}
and, more generally,
\begin{equation*}
|\partial^\beta_{\mathbf{y}}(\theta\circ A)(\mathbf{x},\mathbf{y},\eta)|\le C_\beta\,A(\mathbf{x},\mathbf{y},\eta)^{m-|\beta|}\qquad\forall\;\beta\in\mathbb{N}^N,
\end{equation*}
if \,$\theta\in C^\infty(\mathbb{R}\!\smallsetminus\!\{0\})$ is a homogeneous symbol of order $m\in\mathbb{R}$, i.e.,
\begin{equation*}
|{\bigl(\tfrac d{dx}\bigr)^{\hspace{-.25mm}\beta}}\theta(x)\bigr|\le C_\beta\,|x|^{m-\beta}\qquad\forall\;x\in\mathbb{R}\!\smallsetminus\!\{0\}\,,\;\forall\;\beta\in\mathbb{N}\,.
\end{equation*}
Similarly,
\begin{equation*}
|\partial^\beta_{\mathbf{y}}(\tilde\theta\circ A)(\mathbf{x},\mathbf{y},\eta)|\le C_\beta\,{\bigl\{1\!+\!A(\mathbf{x},\mathbf{y},\eta)\bigr\}^{m-|\beta|}}\qquad\forall\;\beta\in\mathbb{N}^N,
\end{equation*}
if \,$\tilde\theta\in C^\infty(\mathbb{R})$ is an {even} inhomogeneous symbol of order $m\in\mathbb{R}$, i.e.,
\begin{equation*}
\bigl|{\bigl(\tfrac d{dx}\bigr)^{\hspace{-.25mm}\beta}}\tilde\theta(x)\bigr|\le C_\beta\,{(1\!+\!|x|)^{m-\beta}}\qquad\forall\;x\in\mathbb{R}\,,\;\forall\;\beta\in\mathbb{N}\,.
\end{equation*}}

{Consider the radial function
$$
q(\mathbf{x})={c_M}\,(1\!+\!{\|}\mathbf{x}{\|}^2)^{-M/2}
$$
on $\mathbb{R}^N$, where $M\!>\hspace{-.5mm}\mathbf{N}$ {and $c_M\!>\hspace{-.5mm}0$ is a normalizing constant such that $\int_{\mathbb{R}^N}\hspace{-.5mm}q(\mathbf{x})\hspace{.5mm}dw(\mathbf{x})\hspace{-.5mm}=\!1$}. Notice that $\tilde{q}(x)={c_M}\,(1\!+\!x^2)^{-M/2}$ is an {even} inhomogeneous symbol of order $-M$. The following estimate holds for the translates $q_t(\mathbf{x},\mathbf{y})=\tau_{\mathbf{x}}q_t(-\mathbf{y})$ of $q_t(\mathbf{x})=t^{-\mathbf{N}}q(t^{-1}\mathbf{x})$.}

\begin{proposition}\label{translation1}
There exists {a constant $C>0$ $($depending on $M)$} such that
\begin{equation*}
0\le q_{t}(\mathbf{x},\mathbf{y})\le{C}\,V(\mathbf{x},\mathbf{y},{t})^{-1}
\qquad{\forall\;t>0,\;\forall\;\mathbf{x},\mathbf{y}\in\mathbb{R}^N}.
\end{equation*}
\end{proposition}

\begin{proof}
{By scaling we can reduce to $t=1$.}
Fix $\mathbf{x},\mathbf{y}\in\mathbb{R}^N$. {By continuity, the function $\mathbf{y}'\longmapsto q_1(\mathbf{x},\mathbf{y}')$ reaches a maximum $K=q_1(\mathbf{x},\mathbf{y}_0)\ge0$ on the ball}
$$
\bar{B}=\overline{B(\mathbf{y},1)}=\{\mathbf{y}'\in\mathbb{R}^N\,{|}\,{\|}{\mathbf{y}'\!-\hspace{-.5mm}\mathbf{y}}{\|}\leq 1\}.
$$
For {every} $\mathbf{y}'\in\bar{B}$, we have
\begin{equation*}\begin{split}
{0\le}q_1(\mathbf{x},\mathbf{y}_0)-q_1(\mathbf{x},\mathbf{y}')
&= \int_{\mathbb{R}^N}{\bigl\{}{(\tilde{q}\circ A)}(\mathbf{x},\mathbf{y}_0,\eta)-{(\tilde{q}\circ A)}(\mathbf{x},\mathbf{y}',\eta){\bigl\}}\,d\mu_{\mathbf{x}}(\eta) \\
& = \int_{\mathbb{R}^N}\int_0^1{\frac{\partial}{\partial s}}\,{(\tilde{q}\circ A)}(\mathbf{x},{\underbrace{\mathbf{y}'+s(\mathbf{y}_0-\mathbf{y}')}_{\mathbf{y}_{\hspace{-.25mm}s}}},\eta)\,
ds\,d\mu_{\mathbf{x}}(\eta) \\
&\le{\|}\mathbf{y}_0-\mathbf{y}'{\|}\int_{\mathbb{R}^N}\int_0^1|{(\tilde{q}^{\hspace{.25mm}\prime}\circ A)}(\mathbf{x},\mathbf{y}_{\hspace{-.25mm}s},\eta)|\,ds\,d\mu_{\mathbf{x}}(\eta)\\
&\le M\,{\|}\mathbf{y}_0-\mathbf{y}'{\|}\int_{\mathbb{R}^N}\int_0^1{(\tilde{q}\circ A)}(\mathbf{x},\mathbf{y}_{\hspace{-.25mm}s},\eta)\,ds\,d\mu_{\mathbf{x}}(\eta)\\
&=M\,{\|}\mathbf{y}_0-\mathbf{y}'{\|}\int_0^1q_1(\mathbf{x},\mathbf{y}_{\hspace{-.25mm}s})\,ds\\
&\le M\,{\|}\mathbf{y}_0-\mathbf{y}'{\|}\,K\,.
\end{split}\end{equation*}
{Here we have used \eqref{A3} and the elementary estimate}
\begin{equation*}
{|\tilde{q}^{\hspace{.25mm}\prime}(x)|\le M\,\tilde{q}(x)\qquad\forall\;x\hspace{-.5mm}\in\hspace{-.5mm}\mathbb{R}\hspace{.5mm}.}
\end{equation*}
{Hence}
$$
q_1(\mathbf{x},\mathbf{y}')\ge q_1(\mathbf{x},\mathbf{y}_0)-|q_1(\mathbf{x},\mathbf{y}_0)-q_1(\mathbf{x},\mathbf{y}')|\ge K-\frac K2=\frac K2\,,
$$
{if $\mathbf{y}'\in\bar{B}\cap B(\mathbf{y}_0, r)$ with $r={\frac1{2M}}$. Moreover, as} ${w}(\bar{B}\cap B(\mathbf{y}_0, r))\sim {w}(\bar{B})$, {we have}
\begin{equation*}\begin{split}
1&=\int_{\mathbb{R}^N}q_1(\mathbf{x},\mathbf{y}'){dw}(\mathbf{y}')\ge\int_{\bar{B}\cap B(\mathbf{y}_0, r)}q_1(\mathbf{x},\mathbf{y}')\,{dw}(\mathbf{y}')\\
&\ge\frac K2{w}(\bar{B}\cap B(\mathbf{y}_0,r))\ge\frac K{C}\,{w}(\bar{B})\,.
\end{split}\end{equation*}
Thus
$$
0\leq q_1(\mathbf{x},\mathbf{y})\leq K\leq C\,{w}(B(\mathbf{y},1))^{-1}.
$$
{We conclude by using the symmetry $q_1(\mathbf{x},\mathbf{y})=q_1(\mathbf{y},\mathbf{x})$\,.}
\end{proof}

{Consider next a radial function $f$ satisfying {(\footnote{\,As usual, the symbol $\lesssim$ means that there exists a constant $C\!>\hspace{-.5mm}0$ such that $|f(\mathbf{x})|\hspace{-.5mm}\le\hspace{-.5mm}C\hspace{.25mm}(1\!+\hspace{-.5mm}\|\mathbf{x}\|)^{-\hspace{-.25mm}M\hspace{-.25mm}-\kappa}$.})}
\begin{equation*}
|f(\mathbf{x})|\lesssim(1\hspace{-.5mm}+\hspace{-.25mm}\|\mathbf{x}\|)^{-M-\kappa}
\qquad\forall\;\mathbf{x}\in\mathbb{R}^N
\end{equation*}
with $M\!>\hspace{-.5mm}\mathbf{N}$ and $\kappa\hspace{-.5mm}\ge\hspace{-.5mm}0$\hspace{.25mm}. Then the following estimate holds for the translates $f_t(\mathbf{x},\mathbf{y})=\tau_{\mathbf{x}}f_t(-\mathbf{y})$ of
$f_t(\mathbf{x})=t^{-\mathbf{N}}f(t^{-1}\mathbf{x})$.}

\begin{corollary}\label{translation2}
{There exists a constant \,$C\!>\!0$ such that}
\begin{equation*}
|{f_t(\mathbf{x},\mathbf{y})}|\le C\,V(\mathbf{x},\mathbf{y},t)^{-1}\Bigl(1\!+\hspace{-.5mm}\frac{d(\mathbf{x},\mathbf{y})}{t}\Bigr)^{\!-{\kappa}}\qquad{\forall\;t>0,\;\forall\;\mathbf{x},\mathbf{y}\in\mathbb{R}^N}.
\end{equation*}
\end{corollary}

\begin{proof}
{By scaling we can reduce to $t=1$. By using \eqref{translation-radial}, \eqref{A1} and Proposition \ref{translation1}, we get}
\begin{equation*}\begin{split}
|{f_1(\mathbf{x},\mathbf{y})}|
&{\lesssim}\int_{\mathbb{R}^N}\bigl(1\hspace{-.5mm}+\hspace{-.5mm}A(\mathbf{x},\mathbf{y},\eta)\bigr)^{\!-{M}}\bigl(1\hspace{-.5mm}+\hspace{-.5mm}A(\mathbf{x},\mathbf{y},\eta)\bigr)^{\!-{\kappa}}d\mu_{\mathbf{x}}(\eta)\\
&{\le}\int_{\mathbb{R}^N}\bigl(1\hspace{-.5mm}+\hspace{-.5mm}A(\mathbf{x},\mathbf{y},\eta)^2\bigr)^{\!-{M/2}}\bigl(1\hspace{-.5mm}+\hspace{-.25mm}d(\mathbf{x},\mathbf{y})\bigr)^{\!-{\kappa}}d\mu_{\mathbf{x}}(\eta)\\
&\le C\,V(\mathbf{x},\mathbf{y},1)^{-1}\bigl(1\hspace{-.5mm}+\hspace{-.25mm}d(\mathbf{x},\mathbf{y})\bigr)^{\!-{\kappa}}.
\end{split}\end{equation*}
\end{proof}

{
Notice that the space of radial Schwartz functions $f$ on $\mathbb{R}^N$ identifies with the space of even Schwartz functions $\tilde{f}$ on $\mathbb{R}$, which is equipped with the norms
\begin{equation}\label{seminorm}
\|\tilde{f}\|_{\mathcal{S}_m}=\max_{0\le j\le m}\sup_{x\in\mathbb{R}}\,(1\!+\hspace{-.5mm}|x|)^m\,\Bigl|\Bigl(\frac{d}{dx}\Bigr)^{\hspace{-.5mm}j}\hspace{-.5mm}\tilde{f}(x)\Bigr|\qquad\forall\;m\in\mathbb{N}\,.
\end{equation}}

\begin{proposition}\label{Psi_st}
{For every $\kappa\ge0$, there exist $C\ge0$ and $m\in\mathbb{N}$ such that, for all even Schwartz functions $\tilde{\psi}^{\{1\}},\tilde{\psi}^{\{2\}}$ and for all even nonnegative integers $\ell_1,\ell_2$, the convolution kernel
$$
\Psi_{s,t}(\mathbf{x},\mathbf{y})=c_k^{-1}\int_{\mathbb{R}^N}(s\hspace{.25mm}\|\xi\|)^{\ell_1}\hspace{.25mm}\psi^{\{1\}}(s\hspace{.25mm}\|\xi\|)\hspace{.5mm}(t\hspace{.25mm}\|\xi\|)^{\ell_2}\hspace{.25mm}\psi^{\{2\}}(t\hspace{.25mm}\|\xi\|)\hspace{.5mm}E(\mathbf{x},i\xi)\hspace{.5mm}E(-\mathbf{y},i\xi)\hspace{.5mm}dw(\xi)
$$
}satisfies
\begin{align*}
|\Psi_{s,t}(\mathbf{x},\mathbf{y})|&\le{C\,\|\psi^{\{1\}}\|_{\mathcal{S}_{m+\ell_1+\ell_2}}\hspace{.5mm}\|\psi^{\{2\}}\|_{\mathcal{S}_{m+\ell_1+\ell_2}}}\\
&{\times\min\hspace{.5mm}\Bigl\{\hspace{-.25mm}\bigl(\tfrac st\bigr)^{\ell_1}\hspace{-.5mm},\hspace{-.25mm}\bigl(\tfrac ts\bigr)^{\ell_2}\Bigr\}}\,V(\mathbf{x},\mathbf{y}, s+t)^{-1}\Big(1+\frac{d(\mathbf{x},\mathbf{y})}{s+t}\Big)^{\!-{\kappa}}{,}
\end{align*}
{for every $s,t>0$ and for every $\mathbf{x},\mathbf{y}\in\mathbb{R}^N$.}
\end{proposition}

\begin{proof}
{By continuity of the inverse Dunkl transform in the Schwartz setting, there exists a positive even integer $m$ and a constant $C>0$ such that
$$
\sup\nolimits_{\hspace{.5mm}\mathbf{z}\in\mathbb{R}^N}(1\!+\hspace{-.5mm}\|\mathbf{z}\|)^{M+\kappa}\,|\hspace{.25mm}\mathcal{F}^{-1}\hspace{-.5mm}f(\mathbf{z})|\le C\,\|\tilde{f}\|_{\mathcal{S}_m},
$$
for every even function $\tilde{f}\in C^m(\mathbb{R})$ with $\|\tilde{f}\|_{\mathcal{S}_m}<\infty$\hspace{.25mm}. Consider first the case $0<s\le t=1$}\hspace{.25mm}. Then
$$
\|\hspace{.25mm}(s\xi)^{\ell_1}\hspace{.5mm}\tilde{\psi}^{\{1\}}(s\xi)\hspace{.5mm}\xi^{\ell_2}\hspace{.5mm}\tilde{\psi}^{\{2\}}(\xi)\hspace{.25mm}\|_{\mathcal{S}_m}\le C\,{\|\psi^{\{1\}}\|_{\mathcal{S}_m}\hspace{.5mm}\|\psi^{\{2\}}\|_{\mathcal{S}_{m+\ell_1+\ell_2}}}\,s^{\ell_1}.
$$
{According to} Corollary \ref{translation2}, {we deduce that}
\begin{equation*}\begin{split}
|\Psi_{s,1}(\mathbf{x},\mathbf{y})|
&\leq C\,{\mathcal{N}}\,s^{\ell_1}\,V(\mathbf{x},\mathbf{y}, 1)^{-1}\hspace{.5mm}\bigl(1\!+\hspace{-.25mm}{d(\mathbf{x},\mathbf{y})}\bigr)^{\!-\kappa}\\
 &\leq C\,{\mathcal{N}}\,s^{\ell_1}\,V(\mathbf{x},\mathbf{y}, {s\hspace{-.5mm}+\!1})^{-1}\hspace{.5mm}\Bigl(1\!+\hspace{-.5mm}\frac{d(\mathbf{x},\mathbf{y})}{s\hspace{-.5mm}+\!1}\Bigr)^{\!-\kappa}{,}
\end{split}\end{equation*}
{where \,$\mathcal{N}\!=\hspace{-.25mm}\|\psi^{\{1\}}\|_{\mathcal{S}_{m+\ell_1+\ell_2}}\hspace{.25mm}\|\psi^{\{2\}}\|_{\mathcal{S}_{m+\ell_1+\ell_2}}$. In the case $s=1\geq t>0$, we have similarly}
\begin{equation*}
|\Psi_{1,t}(\mathbf{x},\mathbf{y})|
\leq C\,{\mathcal{N}}\,t^{\ell_2}\,V(\mathbf{x},\mathbf{y}, 1\!+\hspace{-.25mm}t)^{-1}\hspace{.5mm}\Bigl(1\!+\hspace{-.5mm}\frac{d(\mathbf{x},\mathbf{y})}{1\!+\hspace{-.25mm}t}\Bigr)^{\!-\kappa}.
\end{equation*}
{The general case is obtained by scaling.}
\end{proof}

\section{Heat kernel and Dunkl kernel}\label{SectionHeat}

{Via the Dunkl transform, the heat semigroup $H_t=e^{\hspace{.25mm}t\hspace{.25mm}\Delta}$ is given by}
$$
{H_t}f(\mathbf{x})=\mathcal{F}^{-1}\bigl(e^{-t{\|}\xi{\|}^2}\mathcal{F}f(\xi)\bigr)(\mathbf{x}).
$$
{Alternately (see, e.g., \cite{Roesler-Voit})
\begin{equation*}
H_tf(\mathbf{x})=f*h_t(\mathbf{x})=\int_{\mathbb{R}^N}
h_t(\mathbf{x},\mathbf{y})\,{f(\mathbf{y})}\,dw(\mathbf{y}),
\end{equation*}
where the heat kernel $h_t(\mathbf{x},\mathbf{y})$ is a smooth positive radial convolution kernel. Specifically, for every $t>0$ and for every $\mathbf{x},\mathbf{y}\in\mathbb{R}^N$,
\begin{equation}\label{Expression1HeatKernel}
h_t(\mathbf{x},\mathbf{y})
=c_k^{-1}\,(2t)^{-\mathbf{N}/2}\,e^{-\frac{\|\mathbf{x}\|^2+\hspace{.25mm}\|\mathbf{y}\|^2}{4t}}E\Bigl(\frac{\mathbf{x}}{\sqrt{2t}},\frac{\mathbf{y}}{\sqrt{2t}}\Bigr)
=\tau_{\mathbf{x}}h_t(-\mathbf{y}),
\end{equation}
where
\begin{equation*}
h_t(\mathbf{x})=\tilde{h}_t(\|\mathbf{x}\|)
=c_k^{-1}\,(2t)^{-\mathbf{N}/2}\,e^{-\frac{{\|}\mathbf{x}{\|}^2}{4t}}.
\end{equation*}
In particular,}
\begin{gather}
{h_t(\mathbf{x},\mathbf{y})=h_t(\mathbf{y},\mathbf{x})>0,}
\nonumber\\
{\int_{\mathbb{R}^N}h_t(\mathbf{x},\mathbf{y})\,dw(\mathbf{y})=1,}
\nonumber\\
h_t(\mathbf{x},\mathbf{y})\leq c_k^{-1}\,(2t)^{-\mathbf{N}/2}\,e^{-\frac{d(\mathbf{x},\mathbf{y})^2}{4t}}.
\label{estimateRosler}
\end{gather}
\vspace{-2mm}

\textbf{Upper heat kernel estimates.} We {prove now Gaussian bounds} for the heat kernel {and its derivatives, in the spirit of} spaces of homogeneous type{, except that} the metric ${\|}\mathbf{x}-\mathbf{y}{\|}$ is replaced by the orbit distance $d(\mathbf{x},\mathbf{y})$. {In comparison with \eqref{estimateRosler}, the main difference lies in the factor $t^{\mathbf{N}\slash 2}$, which is replaced by the volume of appropriate balls.}

\begin{theorem}\label{theoremGauss}
{{\rm(a) Time derivatives\,:}
for any nonnegative integer $m$,} there are constants \,$C,c>0$ such that
\begin{equation}\label{Gauss}
\left|{\partial_t^m}\hspace{.5mm}h_t(\mathbf{x},\mathbf{y})\right|\leq C\,{t^{-m}}\,V(\mathbf{x},\mathbf{y},\!\sqrt{t\,})^{-1}\,e^{-\hspace{.25mm}c\hspace{.5mm}d(\mathbf{x},\mathbf{y})^2\slash t},
\end{equation}
{for every \,$t\hspace{-.5mm}>\hspace{-.5mm}0$ and for every \,$\mathbf{x},\mathbf{y}\!\in\hspace{-.5mm}\mathbb{R}^N$.}
\par\noindent
{{\rm(b) H\"older  bounds\,:}
for any nonnegative integer $m$,} there are constants \,$C,c>0$ such that
\begin{equation}\label{Holder}
\left|{\partial_t^{{m}}}h_t(\mathbf{x},\mathbf{y})-{\partial_t^{{m}}}h_t(\mathbf{x},\mathbf{y}')\right|\leq C\,t^{-{m}}\,\Bigl(\frac{{\|}\mathbf{y}\!-\!\mathbf{y}'{\|}}{\sqrt{t\,}}\Bigr)\,V(\mathbf{x},\mathbf{y},\!\sqrt{t\,})^{-1}\,e^{-\hspace{.25mm}c\hspace{.5mm}d(\mathbf{x},\mathbf{y})^2\slash t},
\end{equation}
{for every \,$t\hspace{-.5mm}>\hspace{-.5mm}0$ and for every \,$\mathbf{x},\mathbf{y},\mathbf{y}'\hspace{-1mm}\in\hspace{-.5mm}\mathbb{R}^N$ such that \,${\|}\mathbf{y}\!-\!\mathbf{y}'{\|}\!<\!\sqrt{t\,}$}.
\par\noindent
{{\rm(c) Dunkl derivative\,:}
for any \,$\xi\hspace{-.5mm}\in\hspace{-.5mm}\mathbb{R}^N$ and for any nonnegative integer $m$,} there are constants \,$C,c\hspace{-.5mm}>\hspace{-.5mm}0$ such that
\begin{equation}\label{TxiDtHeat}
\Bigl|\hspace{.5mm}T_{{\xi},\mathbf{x}}\,{\partial_t^m}\hspace{.5mm}h_t(\mathbf{x},\mathbf{y})\Bigr|\leq C\,t^{-m-1\slash 2}\,V(\mathbf{x},\mathbf{y},\!\sqrt{t\,})^{-1}\,e^{-\hspace{.25mm}c\hspace{.5mm}d(\mathbf{x},\mathbf{y})^2\slash t}\,{,}
\end{equation}
{for all \,$t\hspace{-.5mm}>\hspace{-.5mm}0$ and \,$\mathbf{x},\mathbf{y}\!\in\hspace{-.5mm}\mathbb{R}^N$}.
\par\noindent
{{\rm(d) Mixed derivatives\,:}
for any nonnegative integer \,$m$ and for any multi-indices \,$\alpha,\beta$, there are constants \,$C,c\hspace{-.5mm}>\hspace{-.5mm}0$ such that, for every \,$t>0$ and for every \,$\mathbf{x},\mathbf{y}\in\mathbb{R}^N$,
\begin{equation}\label{DtDxDyHeat}
\bigl|\hspace{.25mm}\partial_t^m\partial_{\mathbf{x}}^{\alpha}\partial_{\mathbf{y}}^{\beta}h_t(\mathbf{x},\mathbf{y})\bigr|\le C\,t^{-m-\frac{|\alpha|}2-\frac{|\beta|}2}\,V(\mathbf{x},\mathbf{y},\!\sqrt{t\,})^{-1}\,e^{-\hspace{.25mm}c\hspace{.5mm}d(\mathbf{x},\mathbf{y})^2\slash t},
\end{equation}
for every \,$t\hspace{-.5mm}>\hspace{-.5mm}0$ and for every \,$\mathbf{x},\mathbf{y}\!\in\hspace{-.5mm}\mathbb{R}^N$.}
\end{theorem}

\begin{proof}
{The proof relies on the expression
\begin{equation}\label{Expression2HeatKernel}
h_t(\mathbf{x},\mathbf{y})=\!\int_{\mathbb{R}^N}\!
\tilde{h}_t\bigl(A(\mathbf{x},\mathbf{y},\eta)\hspace{-.25mm}\bigr)\hspace{.5mm}d\mu_{\mathbf{x}}(\eta)
\end{equation}
and on the properties of \hspace{.5mm}$A(\mathbf{x},\mathbf{y},\eta)$\hspace{.25mm}.
\par\noindent
(a) Consider first the case \,$m\hspace{-.5mm}=\hspace{-.5mm}0$\hspace{.25mm}. By scaling we can reduce to \,$t\hspace{-.5mm}=\!1$. On the one hand, we use \eqref{A1} to estimate}
\begin{equation*}\begin{split}
c_k\,2^{\,\mathbf{N}\slash 2}\,h_1(\mathbf{x},\mathbf{y})
&=\int_{\mathbb{R}^N}e^{-A(\mathbf{x},\mathbf{y},\eta)^2\slash 8}\,e^{-A(\mathbf{x},\mathbf{y},\eta)^2\slash 8}\,d\mu_{\mathbf{x}}(\eta)\\
&{\le e^{-d(\mathbf{x},\mathbf{y})^2\slash 8}}\int_{\mathbb{R}^N}e^{-A(\mathbf{x},\mathbf{y},\eta)^2\slash 8}\,d\mu_{\mathbf{x}}(\eta)\,.
\end{split}\end{equation*}
{On the other hand, it follows from Proposition \ref{translation1} and Corollary \ref{translation2} that}
\begin{equation*}
\int_{\mathbb{R}^N}e^{\,-\,c\,A(\mathbf{x},\mathbf{y},\eta)^2}d\mu_{\mathbf{x}}(\eta)\lesssim V(\mathbf{x},\mathbf{y},{1})^{-1}\,{,}
\end{equation*}
{for any fixed \hspace{.5mm}$c\hspace{-.5mm}>\hspace{-.5mm}0$\hspace{.5mm}. Hence
$$
h_1(\mathbf{x},\mathbf{y})\lesssim V(\mathbf{x},\mathbf{y},1)^{-1}\,e^{-\,d(\mathbf{x},\mathbf{y})^2\slash 8}\,.
$$
Consider next the case \,$m\hspace{-.5mm}>\hspace{-.5mm}0$\hspace{.25mm}. Observe that \hspace{.5mm}$\partial_t^{\hspace{.25mm}m}\tilde{h}_t(x)$ is equal to \hspace{.5mm}$t^{-m}\hspace{.5mm}e^{-\hspace{.25mm}x^2\hspace{-.25mm}/4\hspace{.25mm}t}$ times a polynomial in $\frac{x^2}t$, hence}
\begin{equation}\label{eq41}
\bigl|\hspace{.25mm}{\partial_t^m}\hspace{.25mm}\tilde{h}_t(x)\bigr|\leq C_m\,t^{-m}\,\tilde{h}_{2t}(x)\,.
\end{equation}
{By differentiating \eqref{Expression2HeatKernel} and by using \eqref{eq41}, we deduce that
\begin{equation*}
\bigl|\hspace{.25mm}{\partial_t^m}\hspace{.25mm}h_t(\mathbf{x},\mathbf{y})\bigr|\leq C_m\,t^{-m}\,h_{2t}(\mathbf{x},\mathbf{y})\,.
\end{equation*}
We conclude by using the case \,$m\hspace{-.5mm}=\hspace{-.5mm}0$\hspace{.25mm}.
\par\noindent
(b) Observe now that \hspace{.5mm}$\tilde{\mathfrak{h}}_t(x)\hspace{-.5mm}=\hspace{-.25mm}\partial_x\hspace{.25mm}\partial_t^m\tilde{h}_t(x)$ is equal to \hspace{.5mm}$\frac x{t^{m+1}}\hspace{.5mm}e^{-\hspace{.25mm}x^2\hspace{-.25mm}/4\hspace{.25mm}t}$ times a polynomial in $\frac{x^2}t$, hence}
\begin{equation}\label{dervmix}
\bigl|\hspace{.25mm}{\tilde{\mathfrak{h}}_t}(x)\bigr|\leq C_m\,t^{-m-1\slash 2}\,\tilde{h}_{2t}(x)\,.
\end{equation}
{By differentiating \eqref{Expression2HeatKernel} and by using \eqref{A3} and \eqref{Gauss}, we estimate}
\begin{equation*}\begin{split}
|{\partial_t^mh_t}(\mathbf{x},\mathbf{y})-{\partial_t^mh_t}(\mathbf{x},\mathbf{y}')|
&=\Bigl|\int_{\mathbb{R}^N}\bigl\{{\partial_t^m\tilde{h}_t}(A(\mathbf{x},\mathbf{y},\eta))-{\partial_t^m\tilde{h}_t}(A(\mathbf{x},\mathbf{y}',\eta))\bigr\}d\mu_{\mathbf{x}}(\eta)\Bigr|\\
& =\Bigl|\int_{\mathbb{R}^N}\int_0^1{\frac{\partial}{\partial s}\partial_t^m\tilde{h}_t}(A(\mathbf{x},{\underbrace{\mathbf{y}'\hspace{-1mm}+\hspace{-.5mm}s(\mathbf{y}\!-\hspace{-.5mm}\mathbf{y}')}_{\mathbf{y}_{\hspace{-.25mm}s}}},\eta))\,ds\,d\mu_{\mathbf{x}}(\eta)\Bigr|\\
&\le{\|}\mathbf{y}\!-\hspace{-.5mm}\mathbf{y}'{\|}\int_0^1\!
 \int_{\mathbb{R}^N}\,\bigl|\hspace{.25mm}{\tilde{\mathfrak{h}}_t}(A(\mathbf{x},\mathbf{y}_{\hspace{-.25mm}s},\eta))
 \bigr| \,d\mu_{\mathbf{x}}(\eta) \,ds\\
&\le{C_m\,t^{-m}}\,\frac{{\|}\mathbf{y}\!-\hspace{-.5mm}\mathbf{y}'{\|}}{\sqrt{t\,}}
 \int_0^1h_{2t}(\mathbf{x},\mathbf{y}_{\hspace{-.25mm}s})\,ds\\
&\le{C_m^{\hspace{.25mm}\prime}\,t^{-m}}\,\frac{{\|}\mathbf{y}\!-\hspace{-.5mm}\mathbf{y}'{\|}}{\sqrt{t\,}}\int_0^1{V(\mathbf{x},\mathbf{y}_{\hspace{-.25mm}s},\!\sqrt{2t\,})\,e^{-\hspace{.25mm}c\hspace{.25mm}\frac{d(\mathbf{x},\hspace{.25mm}\mathbf{y}_{\hspace{-.25mm}s})^2}{2\hspace{.25mm}t}}}\,ds\,.
\end{split}\end{equation*}
{In order to conclude, notice that
\begin{equation}\label{comparison1}
V(\mathbf{x},\mathbf{y}_{\hspace{-.25mm}s},\!\sqrt{2\hspace{.25mm}t\,})\sim V(\mathbf{x},\mathbf{y},\!\sqrt{t\,})
\end{equation}
under the assumption \hspace{.25mm}${\|}\mathbf{y}\!-\hspace{-.5mm}\mathbf{y}'{\|}\hspace{-.5mm}<\!\sqrt{t\,}$} {and let us show that, for every \hspace{.25mm}$c\hspace{-.5mm}>\hspace{-.5mm}0$\hspace{.25mm}, there exists \hspace{.25mm}$C\hspace{-.5mm}\ge\!1$ \hspace{.25mm}such that
\begin{equation}\label{comparison2}
C^{-1}\hspace{.5mm}e^{-\frac32\hspace{.25mm}c\hspace{.5mm}\frac{d(\mathbf{x},\hspace{.25mm}\mathbf{y})^2}t}\hspace{-.25mm}\le e^{-\hspace{.25mm}c\hspace{.5mm}\frac{d(\mathbf{x},\hspace{.25mm}\mathbf{y}_{\hspace{-.25mm}s})^2}t}\hspace{-.25mm}\le C\,e^{-\frac12\hspace{.25mm}c\hspace{.5mm}\frac{d(\mathbf{x},\hspace{.25mm}\mathbf{y})^2}t}.
\end{equation}
As long as $d(\mathbf{x},\mathbf{y})=\text{O}(\sqrt{t\,})$, all expressions in \eqref{comparison2} are indeed comparable to $1$. On the other hand, if $d(\mathbf{x},\mathbf{y})\ge\sqrt{32\hspace{.5mm}t\,}$, then
\begin{equation*}\begin{split}
|d(\mathbf{x},\mathbf{y})^2\!-d(\mathbf{x},\mathbf{y}_{\hspace{-.25mm}s})^2|
&=|d(\mathbf{x},\mathbf{y})-d(\mathbf{x},\mathbf{y}_{\hspace{-.25mm}s})|\,
\{d(\mathbf{x},\mathbf{y})+d(\mathbf{x},\mathbf{y}_{\hspace{-.25mm}s})\}\\
&\le\|\mathbf{y}\hspace{-.5mm}-\hspace{-.25mm}\mathbf{y}_{\hspace{-.25mm}s}\|\,
\{2\,d(\mathbf{x},\mathbf{y})+\|\mathbf{y}\hspace{-.5mm}-\hspace{-.25mm}\mathbf{y}_{\hspace{-.25mm}s}\|\}
\le\sqrt{2\hspace{.5mm}t\,}\{2\,d(\mathbf{x},\mathbf{y})+\sqrt{2\hspace{.5mm}t\hspace{.5mm}}\}\\
&\le\sqrt{8\hspace{.5mm}t\,}d(\mathbf{x},\mathbf{y})+2\hspace{.5mm}t
\le\frac12\,d(\mathbf{x},\mathbf{y})^2\!+2\hspace{.5mm}t\,.
\end{split}\end{equation*}
Hence
\begin{equation*}
\frac12\,d(\mathbf{x},\mathbf{y})^2/t-2\le d(\mathbf{x},\mathbf{y}_{\hspace{-.25mm}s})^2/t\le\frac32\,d(\mathbf{x},\mathbf{y})^2/t+2\,.
\end{equation*}}
\par\noindent
{(c) By symmetry, we can replace $T_{\xi,\mathbf{x}}$ by $T_{\xi,\mathbf{y}}$. Consider first the contribution of the directional derivative in $T_{\xi,\mathbf{y}}$. By differentiating \eqref{Expression2HeatKernel} and by using \eqref{dervmix} and \eqref{Gauss}, we estimate as above}
\begin{equation*}\begin{split}
|\partial_{{\xi},\mathbf{y}}{\partial_t^m}h_t(\mathbf{x},\mathbf{y})|
&\leq{\|\xi\|}\int_{\mathbb{R}^N}|{\tilde{\mathfrak{h}}}_t(A(\mathbf{x},\mathbf{y},\eta))|\,d\mu_{\mathbf{x}}(\eta)\\
&\leq{C}\,t^{-m-1\slash 2}\,h_{2t}(\mathbf{x},\mathbf{y})\\
&\leq C\,t^{-m-1\slash 2}\,V(\mathbf{x},\mathbf{y},\!\sqrt{t\,})^{-1}\,e^{-\hspace{.25mm}c\hspace{.5mm}d(\mathbf{x},\mathbf{y})^2\slash t}.
\end{split}\end{equation*}
{Consider next the contributions}
\begin{equation}\label{eq44}
 \frac{{\partial_t^m}h_t(\mathbf{x},\mathbf{y})-{\partial_t^m}h_t(\mathbf{x},
 \sigma_\alpha{(}\mathbf{y}{)})}{\langle \alpha, \mathbf{y}\rangle}
\end{equation}
{of the difference operators in $T_{{\xi},\mathbf{y}}$}.
If \,$|\langle \alpha,\mathbf{y}\rangle|\!>\hspace{-.5mm}\sqrt{t{/2}\hspace{.5mm}}$, {we use \eqref{Gauss} and estimate separately each term in \eqref{eq44}}.
If $|\langle \alpha,\mathbf{y}\rangle|\le\sqrt{t{/2}\hspace{.5mm}}$,
{we estimate again}
\begin{equation*}\begin{split}
\Bigl|\frac{{\partial_t^m}h_t(\mathbf{x},\mathbf{y})-{\partial_t^m}h_t(\mathbf{x},\sigma_\alpha{(}\mathbf{y}{)})}{\langle \alpha, \mathbf{y}\rangle}\Bigr|
&\le\sqrt{2\hspace{.5mm}}\int_{\mathbb{R}^N}\int_0^1|{\tilde{\mathfrak{h}}}_t(A(\mathbf{x},\mathbf{y}_{\hspace{-.25mm}s},\eta))|ds\, d\mu_{\mathbf{x}}(\eta)\\
&\le{C}\,t^{-m-1\slash 2}\int_0^1h_{2t}(\mathbf{x},\mathbf{y}_{\hspace{-.25mm}s})\,ds\\
&\le{C}\,t^{-m-1\slash 2}\int_0^1V(\mathbf{x},\mathbf{y}_{\hspace{-.25mm}s},\!\sqrt{{2}\hspace{.25mm}t\,})^{-1}\,e^{-\hspace{.25mm}c\hspace{.5mm}\frac{d(\mathbf{x},\hspace{.25mm}\mathbf{y}_{\hspace{-.25mm}s})^2}{{2}\hspace{.25mm}t}}\,ds\\
&{\le C\,t^{-m-1\slash 2}\,V(\mathbf{x},\mathbf{y},\!\sqrt{t\,})^{-1}\,e^{-\hspace{.25mm}c\hspace{.25mm}\frac{d(\mathbf{x},\hspace{.25mm}\mathbf{y})^2}t}.}
\end{split}\end{equation*}
{In the last step we have used \eqref{comparison1} and \eqref{comparison2}, which hold as \hspace{.5mm}$\|\mathbf{y}_{\hspace{-.25mm}s}\!-\hspace{-.5mm}\mathbf{y}\|\!\le\!\sqrt{t\,}$.}
\par\noindent
{(d) This time, we use \eqref{A2} to estimate
\begin{equation}\label{eq45}
\bigl|\hspace{.25mm}\partial_{\mathbf{y}}^\beta\partial_t^m\tilde{h}_t\bigl(A(\mathbf{x},\mathbf{y},\eta)\hspace{-.25mm}\bigr)\bigr|\le C_{m,\beta}\,t^{-m-\frac{|\beta|}2}\,\tilde{h}_{2t}\bigl(A(\mathbf{x},\mathbf{y},\eta)\hspace{-.25mm}\bigr)\,.
\end{equation}
Firstly, by differentiating \eqref{Expression2HeatKernel} and by using \eqref{eq45},
we obtain
\begin{equation}\label{eq43}
\bigl|\hspace{.25mm}\partial_t^m\partial_{\mathbf{y}}^\beta\hspace{.25mm}h_t(\mathbf{x},\mathbf{y})\bigr|\le C_{m,\beta}\,t^{-m-\frac{|\beta|}2}\,h_{2t}(\mathbf{x},\mathbf{y})\,.
\end{equation}
Secondly, by differentiating
\begin{equation*}
h_t(\mathbf{x},\mathbf{y})\,=\int_{\mathbb{R}^N}\!h_{t/2}(\mathbf{x},\mathbf{z})\,h_{t/2}(\mathbf{z},\mathbf{y})\,dw(\mathbf{z})\,,
\end{equation*}
by using \eqref{eq43} and by symmetry, we get
$$
\bigl|\hspace{.25mm}\partial_t^m\partial_{\mathbf{x}}^\alpha\partial_{\mathbf{y}}^\beta\hspace{.25mm}h_t(\mathbf{x},\mathbf{y})\bigr|\le C_{m,\alpha,\beta}\,t^{-m-\frac{|\alpha|}2-\frac{|\beta|}2}\,h_{2t}(\mathbf{x},\mathbf{y})\,.
$$
We conclude by using \eqref{Gauss}.}
\end{proof}
\smallskip

\textbf{Lower heat kernel estimates.}
We begin with an auxiliary result.

\begin{lemma}\label{auxiliarylemma}
Let $\tilde{f}$ be a smooth bump function on $\mathbb{R}$ such that $0\le\tilde{f}\le 1$, $\tilde f(x)=1$ if $|x|\le\frac12$ and $\tilde f(x)=0$ if $|x|\ge1$. Set as usual
\begin{equation*}
f(\mathbf{x})=\tilde{f}(\|\mathbf{x}\|)
\quad\text{and}\quad
f(\mathbf{x},\mathbf{y})=\tau_{\mathbf{x}}f(-\mathbf{y}).
\end{equation*}
Then $0\leq f(\mathbf{x},\mathbf{y})\leq 1$ and $f(\mathbf{x},\mathbf{y})=0$ if $d(\mathbf{x},\mathbf{y})\ge1$. Moreover there exists a positive constant $c_1$ such that
\begin{equation}\label{lowerestimatef}
{\sup_{\mathbf{y}\in\mathcal{O}(B(\mathbf{x},1)}}f(\mathbf{x},\mathbf{y})
\ge\frac{c_1}{w(B(\mathbf{x},1))}
\end{equation}
for every $\mathbf{x}\in\mathbb{R}^N$.
\end{lemma}

\begin{proof}
All claims follow from \eqref{translation-radial} and \eqref{A1}.
Let us prove the last one.
On the one hand, by translation invariance,
\begin{equation*}
\int_{\mathbb{R}^N}f(\mathbf{x},\mathbf{y})\,dw(\mathbf{y})
=\int_{\mathbb{R}^N}f(\mathbf{y})\,d w(\mathbf{y})
\ge w(B(0,1\slash 2)).
\end{equation*}
On the other hand,
\begin{equation*}
\int_{\mathbb{R}^N}f(\mathbf{x},\mathbf{y})\,dw(\mathbf{y})
=\int_{\mathcal{O}(B(\mathbf{x},1))}f(\mathbf{x},\mathbf{y})\,d w(\mathbf{y})
\le|G|\,w(B(\mathbf{x},1))\sup_{\mathbf{y}\in \mathcal{O}( B(\mathbf{x},1))}f(\mathbf{x},\mathbf{y}).
\end{equation*}
This proves \eqref{lowerestimatef} with $c_1=\frac{w(B(0,1\slash 2))}{|G|}$.
\end{proof}

\begin{proposition}\label{LowerBoundNearDiagonal}
There exist positive constants $c_2$ and $\varepsilon$ such that
\begin{equation*}
h_t(\mathbf{x},\mathbf{y})\geq\frac{c_2}{w(B(\mathbf{x},\!\sqrt{t\,}))}
\end{equation*}
for every \,$t>0$ and \,$\mathbf{x},\mathbf{y}\in\mathbb{R}^N$ satisfying $\|\mathbf{x}-\mathbf{y}\|\le\varepsilon\sqrt{t\,}$.
\end{proposition}

\begin{proof}
 By scaling it suffices to prove the proposition for $t=2$.
According to Lemma \ref{auxiliarylemma}, applied to $\tilde{h}_1\gtrsim\tilde{f}$ {(\footnote{\,As usual, the symbol $\gtrsim$ means that there exists a constant $C\!>\hspace{-.5mm}0$ such that $\tilde{h}_1\!\ge\hspace{-.5mm}C\tilde{f}$.})},
there exists $c_3>0$ and, for every $\mathbf{x}\in\mathbb{R}^N$,
there exists $\mathbf{y}(\mathbf{x})\in\mathcal{O}(B(\mathbf{x},1))$
such that
\begin{equation*}
h_1(\mathbf{x},\mathbf{y}(\mathbf{x}))\geq c_3\,w(B(\mathbf{x},1))^{-1}.
\end{equation*}
This estimate holds true around $\mathbf{y}(\mathbf{x})$, according to \eqref{Holder},
Specifically, there exists $0<\varepsilon<1$ (independent of $\mathbf{x}$)
such that
\begin{equation*}
h_1(\mathbf{x},\mathbf{y})\geq\tfrac{c_3}2\,w(B(\mathbf{x},1))^{-1}
\qquad\forall\;\mathbf{y}\in B(\mathbf{y}(\mathbf{x}),\varepsilon).
\end{equation*}
By using the semigroup property and the symmetry of the heat kernel,
we deduce that
\begin{equation*}\begin{split}
h_2(\mathbf{x},\mathbf{x})
&=\int h_1(\mathbf{x},\mathbf{y})\,h_1(\mathbf{y},\mathbf{x})\,
dw(\mathbf{y})\\
&\geq\int_{B(\mathbf{y}(\mathbf{x}),\varepsilon)}
h_1(\mathbf{x},\mathbf{y})^2\,dw(\mathbf{y})\\
&\geq w(B(\mathbf{y}(\mathbf{x}),\varepsilon)\,
(\tfrac{c_3}2)^2\,w(B(\mathbf{x},1))^{-2}.
\end{split}\end{equation*}
By using the fact that the balls $B(\mathbf{y}(\mathbf{x}),\varepsilon)$,
$B(\mathbf{x},1)$, $B(\mathbf{x},\sqrt{2})$ have comparable volumes
and by using again \eqref{Holder}, we conclude that
\begin{equation*}
h_2(\mathbf{x},\mathbf{y})\geq c_4\,w(B(\mathbf{x},\sqrt{2}))^{-1}
\end{equation*}
for all $\mathbf{x},\mathbf{y}\in\mathbb{R}^N$ sufficiently close.
\end{proof}

A standard argument, which we include for the reader's convenience,
allows us to deduce from such a near on diagonal estimate
the following global lower Gaussian bound.

\begin{theorem}
There exist positive constants $C$ and $c$ such that
\begin{equation}\label{gaussian_lower}
h_t(\mathbf{x},\mathbf{y})\geq\frac{C}{\min\hspace{.5mm}\{\hspace{.25mm}w(B(\mathbf{x},\!\sqrt{t\,})),w(B(\mathbf{y},\!\sqrt{t\,}))\}}\,e^{-\hspace{.25mm}c\hspace{.5mm}{\|}\mathbf{x}-\mathbf{y}{\|}^2\!\slash t}
\end{equation}
for every \,$t>0$ and for every \,$\mathbf{x},\mathbf{y}\in\mathbb{R}^N$.
\end{theorem}

\begin{proof}
{We resume the notation of Proposition \ref{LowerBoundNearDiagonal}.} Assume that ${\|}\mathbf{x}-\mathbf{y}{\|}^2/t\geq1$ and set $n=\lceil4{\|}\mathbf{x}-\mathbf{y}{\|}^2/(\varepsilon^2t)\rceil \geq4$.
{L}et $\mathbf{x}_i=\mathbf{x}+i(\mathbf{y}-\mathbf{x})/n$ ($i=0,\dots,n$), so that $\mathbf{x}_0=\mathbf{x}$, $\mathbf{x}_{n}=\mathbf{y}$, and ${\|}\mathbf{x}_{i+1}-\mathbf{x}_i{\|}={\|}\mathbf{x}-\mathbf{y}{\|}/n$. {Consider the balls} $B_i=B(\mathbf{x}_i,\frac{\varepsilon}4\sqrt{t/n})$ and observe that
$$
{\|}\mathbf{y}_{i+1}-\mathbf{y}_i{\|}
\leq{\|}\mathbf{y}_i-\mathbf{x}_i{\|}
+{\|}\mathbf{x}_i-\mathbf{x}_{i+1}{\|}
+{\|}\mathbf{x}_{i+1}-\mathbf{y}_{i+1}{\|}
<\frac{\varepsilon}4\sqrt{\frac tn}
+\frac{\varepsilon}2\sqrt{\frac tn}
+\frac{\varepsilon}4\sqrt{\frac tn}
=\varepsilon\sqrt{\frac tn}
$$
{if} $\mathbf{y}_i\in B_i$ and $\mathbf{y}_{i+1}\in B_{i+1}$.
By using the semigroup property, Proposition \ref{LowerBoundNearDiagonal} and the {behavior of the ball volume, we estimate}
\begin{equation*}\begin{split}
h_t (\mathbf{x},\mathbf{y})
&=\int_{\mathbb{R}^N}\dots\int_{\mathbb{R}^N}h_{t/n}(\mathbf{x},\mathbf{y}_1)h_{t/n}(\mathbf{y}_1,\mathbf{y}_2)\dots h_{t/n}(\mathbf{y}_{n-1},\mathbf{y})\, dw(\mathbf{y}_1)\dots dw(\mathbf{y}_{n-1})\\
&\geq c_2^{n-1}\int_{B_1}\dots\int_{B_{n-1}}w(B(\mathbf{x},\sqrt{t/n}))^{-1}\dots w(B(\mathbf{y}_{n-1},\sqrt{t/n}))^{-1}dw (\mathbf{y}_1)\dots dw(\mathbf{y}_{n-1})\\
&\geq c_3^{n-1}w(B(\mathbf{x},\sqrt{t/n}))^{-1}\frac{w(B_1)\dots w(B_{n-1})}{w(B(\mathbf{x}_1,\sqrt{t/n}))\dots w(B(\mathbf{x}_{n-1},\sqrt{t/n}))}\\
&\geq c_5^{n-1}w(B(\mathbf{x},\!\sqrt{t\,}))^{-1}
=c_5^{-1}w(B(\mathbf{x},\!\sqrt{t\,}))^{-1} e^{-n\ln c_5^{-1}}
\geq C\,w(B(\mathbf{x},\!\sqrt{t\,}))^{-1}e^{-c\frac{\|\mathbf{x}-\mathbf{y}\|^2}{t}}.
\end{split}\end{equation*}
We conclude by symmetry.
\end{proof}

{By combining \eqref{Gauss} and \eqref{gaussian_lower}, we obtain in particular the following near on diagonal estimates. Notice that the ball volumes $w(B(\mathbf{x},\!\sqrt{t\,}))$ and $w(B(\mathbf{y},\!\sqrt{t\,}))$ are comparable under the assumptions below.

\begin{corollary}\label{EstimateNearDiagonal}
For every \,$c>0$, there exists \,$C>0$ such that
\begin{equation*}
\frac{C^{-1}}{w(B(\mathbf{x},\!\sqrt{t\,}))}\leq h_t(\mathbf{x},\mathbf{y})\leq\frac{C}{w(B(\mathbf{x},\!\sqrt{t\,}))}
\end{equation*}
for every \,$t\hspace{-.5mm}>\hspace{-.5mm}0$ and \,$\mathbf{x},\mathbf{y}\!\in\hspace{-.5mm}\mathbb{R}^N$ such that \,$\|\mathbf{x}\hspace{-.5mm}-\hspace{-.5mm}\mathbf{y}\|\hspace{-.5mm}\le\hspace{-.5mm}c\hspace{.5mm}\sqrt{t\,}$.
\end{corollary}}

{\bf Estimates {of} the Dunkl kernel.}
{According to \eqref{Expression1HeatKernel},
the heat kernel estimates \eqref{Gauss} and \eqref{gaussian_lower}
imply the following results,
which partially improve upon known estimates for the Dunkl kernel.
Notice that $\mathbf{x}$ can be replaced by $\mathbf{y}$ in the ball volumes below.

\begin{corollary}
There are constants $c\ge1$ and $C\ge1$ such that
\begin{equation*}
\frac{C^{-1}}{w(B(\mathbf{x},1))}\,
e^{\frac{\|\mathbf{x}\|^2+\|\mathbf{y}\|^2}2}
e^{-c\,\|\mathbf{x}-\mathbf{y}\|^2}
\leq E(\mathbf{x},\mathbf{y})
\leq\frac{C}{w(B(\mathbf{x},1))}\,
e^{\frac{\|\mathbf{x}\|^2+\|\mathbf{y}\|^2}2}
e^{-c^{-1}d(\mathbf{x},\mathbf{y})^2}
\end{equation*}
for all $\mathbf{x},\mathbf{y}\in\mathbb{R}^N$.
In particular,
\\
$\bullet$ {for every $\varepsilon>0$,}
there exist $C\ge1$ such that
\begin{equation*}
\frac{C^{-1}}{w(B(\mathbf{x},1))}\,e^{\frac{\| \mathbf{x}\|^2+\| \mathbf{y}\|^2}2}
\leq E(\mathbf{x},\mathbf{y})\leq
\frac{C}{w(B(\mathbf{x},1))}\,e^{\frac{\|\mathbf{x}\|^2+\|\mathbf{y}\|^2}2}
\end{equation*}
for all $\mathbf{x},\mathbf{y}\in\mathbb{R}^N$
satisfying $\|\mathbf{x}-\mathbf{y}\|<\varepsilon\,;$
\\
$\bullet$ there exist $c>0$ and $C>0$ such that
$$
E(\lambda\mathbf{x},\mathbf{y})
\geq\frac{C}{w(B(\sqrt{\lambda\,}\mathbf{x},1))}\,
e^{\hspace{.25mm}\lambda\hspace{.25mm}(1-\hspace{.25mm}c\hspace{.5mm}\|\mathbf{x}-\mathbf{y}\|^2)}
$$
for all $\lambda{\ge1}$
and for all $\mathbf{x},\mathbf{y}\in\mathbb{R}^N$
with $\|\mathbf{x}\|=\|\mathbf{y}\|=1$.
\end{corollary}}

\section{Poisson kernel in the Dunkl setting}\label{SectionPoisson}

{The Poisson semigroup $P_t=e^{-\hspace{.25mm}t\sqrt{\hspace{-.5mm}-\Delta}}$ is subordinated to the heat semigroup $H_t=e^{\hspace{.25mm}t\hspace{.25mm}\Delta}$ by \eqref{subordination} and correspondingly for their integral kernels}
\begin{equation}\label{subord}
{p}_t(\mathbf{x},\mathbf{y})=\pi^{-1/2}\int_0^\infty\!e^{-u}\,h_{\frac{t^2}{4u}}(\mathbf{x},\mathbf{y})\,\frac{du}{\sqrt{u}}\,.
\end{equation}
{This subordination formula enables us to transfer properties of the heat kernel $h_t(\mathbf{x},\mathbf{y})$ to the Poisson kernel $p_t(\mathbf{x},\mathbf{y})$. For instance,}
\begin{gather}
{p}_t(\mathbf{x},\mathbf{y}) ={p}_t(\mathbf{y},\mathbf{x}){>0},\nonumber\\
\int_{\mathbb{R}^N} {p}_t(\mathbf{x},\mathbf{y})\, {dw}(\mathbf{y}) =1{,}\nonumber\\
{p}_t(\mathbf{x},\mathbf{y})=\tau_{\mathbf{x}} {p}_t(-\mathbf{y}),\label{Poisson3}
\end{gather}
where
\begin{equation}\label{Poisson4}
{p}_t(\mathbf{x})
={\tilde{p}_t(\|\mathbf{x}\|)
=c_k'\,t\hspace{.5mm}\bigl(t^2\!+\|\mathbf{x}\|^2\bigr)^{\!-\frac{\mathbf{N}+1}2}}
\end{equation}
{and
\begin{equation*}
c_k'=\frac{2^{\hspace{.25mm}\mathbf{N}/2}\,\Gamma(\frac{\mathbf{N}+1}2)}{\sqrt{\pi}\,c_k}>0\,.
\end{equation*}

The following global bounds hold for the Poisson kernel and its derivatives.}

\begin{proposition}\label{Poiss_new}
{\rm (a) Upper and lower bounds\,:}
there is a constant $C{\ge\!1}$ such that
\begin{equation}\label{Poisson_low_up}
\frac{C^{-1}}{V(\mathbf{x},\mathbf{y},t+{\|}\mathbf{x}-\mathbf{y}{\|})}\,\frac{t}{t+{\|}\mathbf{x}-\mathbf{y}{\|}}
\leq {p}_t(\mathbf{x},\mathbf{y})\leq\frac{C}{V(\mathbf{x},\mathbf{y},t+d(\mathbf{x},\mathbf{y}))}\,\frac{t}{t+d(\mathbf{x},\mathbf{y})}
\end{equation}
{for every \,$t>0$ and for every \,$\mathbf{x},\mathbf{y}\in\mathbb{R}^N$.}
\par\noindent
{\rm (b) Dunkl gradient\,:}
for every \,$\xi\in\mathbb{R}^N$, there is a constant $C>0$ such that
\begin{equation}\label{TxiPoisson}
\bigl|T_{\xi,\mathbf{y}} {p}_t(\mathbf{x},\mathbf{y})\bigr|\leq
\frac{C}{V(\mathbf{x}, \mathbf{y},t+d(\mathbf{x},\mathbf{y}))}\,\frac{1}{t+d(\mathbf{x},\mathbf{y})}
\end{equation}
{for all \,$t>0$ and \,$\mathbf{x},\mathbf{y}\in\mathbb{R}^N$.}
\par\noindent
(c) Mixed derivatives\,:
for any nonnegative integer \,$m$ and for any multi-index \,$\beta$, there is a constant \,$C\hspace{-.5mm}\ge\hspace{-.5mm}0$ such that, for every \,$t>0$ and for every \,$\mathbf{x},\mathbf{y}\in\mathbb{R}^N$,
\begin{equation}\label{DtDyPoisson}
\bigl|\hspace{.25mm}\partial_t^m\partial_{\mathbf{y}}^{\beta}\hspace{.25mm}p_t(\mathbf{x},\mathbf{y})\bigr|\le C\,p_t(\mathbf{x},\mathbf{y})\hspace{.25mm}\bigl(\hspace{.25mm}t\hspace{-.25mm}+d(\mathbf{x},\mathbf{y})\bigr)^{\hspace{-.5mm}-m-|\beta|}\times\begin{cases}
\,1&\text{if \,}m\hspace{-.25mm}=\hspace{-.25mm}0\hspace{.25mm},\\
\,1+\frac{d(\mathbf{x},\mathbf{y})}t&\text{if \,}m\hspace{-.5mm}>\hspace{-.5mm}0\hspace{.25mm}.\\
\end{cases}\end{equation}
Moreover, for any nonnegative integer \,$m$ and for any multi-indices \,$\beta,\beta'$, there is a constant \,$C\hspace{-.5mm}\ge\hspace{-.5mm}0$ such that, for every \,$t>0$ and for every \,$\mathbf{x},\mathbf{y}\in\mathbb{R}^N$,
\begin{equation}\label{DtDxDyPoisson}
\bigl|\hspace{.25mm}\partial_t^m\partial_{\mathbf{x}}^{\beta}\partial_{\mathbf{y}}^{\beta'}p_t(\mathbf{x},\mathbf{y})\bigr|\le C\,t^{-m-|\beta|-|\beta'|}\,p_t(\mathbf{x},\mathbf{y})\,.
\end{equation}
\end{proposition}
{Notice that,} by symmetry{,} \eqref{TxiPoisson} holds {also with $T_{\xi,\mathbf{x}}$ instead of $T_{\xi,\mathbf{y}}$}.

\begin{proof}
 (a) The Poisson kernel bounds \eqref{Poisson_low_up} are obtained by inserting the heat kernel bounds \eqref{Gauss} and \eqref{gaussian_lower} in the subordination formula \eqref{subord}. For a detailed proof we refer the reader
 to \cite[Proposition 6]{DzPr}.

\noindent
(b) The Dunkl gradient estimate \eqref{TxiPoisson} is deduced  similarly from \eqref{TxiDtHeat}.

\noindent
{(c) The estimate \eqref{DtDyPoisson} is proved directly. As $(t,x)\longmapsto(t^2\!+\hspace{-.5mm}x^2)^{-{(\mathbf{N}+1)}/2}$ is a homogeneous symbol of order $-\hspace{.25mm}\mathbf{N}\!-\!1$ on $\mathbb{R}^2$, we have
\begin{equation}\label{Poisson5}\begin{cases}
\,|\hspace{.25mm}\partial_x^\beta\hspace{.5mm}\tilde{p}_t(x)|\le C_{\beta}\hspace{.5mm}(t\hspace{-.5mm}+\!|x|)^{-\beta}\hspace{.5mm}\tilde{p}_t(x)\\
\,|\hspace{.25mm}\partial_t^m\hspace{.25mm}\partial_x^\beta\hspace{.5mm}\tilde{p}_t(x)|\le C_{m,\beta}\hspace{.5mm}t^{-1}(t\hspace{-.5mm}+\!|x|)^{1-m-\beta}\hspace{.5mm}\tilde{p}_t(x)
\end{cases}
\quad\forall\;t\hspace{-.5mm}>\hspace{-.5mm}0\hspace{.25mm},\,\forall\;x\!\in\hspace{-.5mm}\mathbb{R}\hspace{.25mm}.
\end{equation}
for every positive integer $m$ and for every nonnegative integer $\beta$. By using \eqref{translation-radial}, \eqref{A1}, \eqref{Poisson3}, \eqref{Poisson4} and
\eqref{Poisson5}, we estimate}
\begin{equation*}\begin{split}
\bigl|\hspace{.25mm}{\partial_{\mathbf{y}}^\beta\hspace{.25mm}p}_t(\mathbf{x},\mathbf{y})\bigr|
&{\leq}\int_{\mathbb{R}^N}{\big|}\hspace{.5mm}{\partial_{\mathbf{y}}^\beta}\hspace{.5mm}\tilde{p}_t(A(\mathbf{x},\mathbf{y},\eta)){\big|}\,d\mu_{\mathbf{x}}(\eta)\\
&\leq{C_\beta}\!\int_{\mathbb{R}^N}{\bigl(\hspace{.25mm}t\hspace{-.25mm}+\hspace{-.5mm}A(\mathbf{x},\mathbf{y},\eta)\hspace{-.25mm}\bigr)^{\!-|\beta|}}\,\tilde{p}_t(A(\mathbf{x},\mathbf{y},\eta))\,d\mu_{\mathbf{x}}(\eta)\\
&\leq{C_\beta\,\bigl(\hspace{.25mm}t\hspace{-.25mm}+\hspace{-.25mm}d(\mathbf{x},\mathbf{y}\hspace{-.25mm}\bigr)^{\!-|\beta|}\,p}_t(\mathbf{x},\mathbf{y})
 \end{split}\end{equation*}
{and similarly
\begin{equation*}
\bigl|\hspace{.25mm}\partial_t^m\partial_{\mathbf{y}}^\beta\hspace{.25mm}p_t(\mathbf{x},\mathbf{y})\bigr|
\leq C_{m,\beta}\,t^{-1}\bigl(\hspace{.25mm}t\hspace{-.25mm}+\hspace{-.25mm}d(\mathbf{x},\mathbf{y}\hspace{-.25mm}\bigr)^{\!1-m-|\beta|}\,p_t(\mathbf{x},\mathbf{y})
\end{equation*}
for every positive integer $m$\hspace{.25mm}.}
{Finally \eqref{DtDxDyPoisson} is deduced from \eqref{DtDyPoisson} by using the semigroup property. More precisely, by differentiating
$$
p_t(\mathbf{x},\mathbf{y})\,=\int_{\mathbb{R}^N}p_{t/2}(\mathbf{x},\mathbf{z})\,p_{t/2}(\mathbf{z},\mathbf{y})\,dw(\mathbf{z})\,,
$$
by using \eqref{DtDyPoisson} and by symmetry, we obtain
$$
\bigl|\hspace{.25mm}\partial_t^m\partial_{\mathbf{x}}^{\beta}\partial_{\mathbf{y}}^{\beta'}p_t(\mathbf{x},\mathbf{y})\bigr|\lesssim t^{-m-|\beta|-|\beta'|}\int_{\mathbb{R}^N}p_{t/2}(\mathbf{x},\mathbf{z})\,p_{t/2}(\mathbf{z},\mathbf{y})\,dw(\mathbf{z})=t^{-m-|\beta|-|\beta'|}\,p_t(\mathbf{x},\mathbf{y})\,.
$$}
\end{proof}

{Notice the following straightforward consequence of the upper bound in \eqref{Poisson_low_up}\,:}
\begin{equation}\label{PHL}
\mathcal{M}_Pf(\mathbf{x}){\lesssim}\sum_{\sigma\in G}\mathcal{M}_{HL} f(\sigma(\mathbf{x}))\,,
\end{equation}
where $\mathcal{M}_{HL}$ denotes the Hardy-Littlewood maximal function on the space of homogeneous type $(\mathbb{R}^N,{\|}\mathbf{x}\hspace{-.5mm}-\hspace{-.5mm}\mathbf{y}{\|},{dw})$. {Likewise, \eqref{Gauss} yields}
$$
\mathcal{M}_{H}f(\mathbf{x}){\lesssim}\sum_{\sigma\in G}\mathcal{M}_{HL} f(\sigma(\mathbf{x}))\,.
$$

Observe that the Poisson kernel {is} an approximation of the identity in the {following} sense.
\begin{proposition}\label{PoissonApprox}
Given {any} compact {sub}set $K\subset\mathbb{R}^N$, {any} \,$r>0$ and {any} \,$\varepsilon>0$, there {exists} \,$t_0=t_0(K,r,\varepsilon)>0$ such that{,} for {every} \,$0<t<t_0$ and {for every} \,$\mathbf{x}\in K$,
$$
\int_{{\|}\mathbf{x}-\mathbf{y}{\|}>r}{p}_t(\mathbf{x},\mathbf{y})\, {dw}(\mathbf{y})<\varepsilon\,.
$$
\end{proposition}
\begin{proof}
Let {$K$ be a compact subset of $\mathbb{R}^N$ and let} $r,\varepsilon>0$. Fix $\mathbf{x}_0\in K$ and consider $f\in C^\infty_c{(\mathbb{R}^N)}$  such that $0\leq f\leq 1$, {$f=1$ on $B(\mathbf{x}_0,r\slash 4)$ and $\operatorname{supp}f\subset B(\mathbf{x}_0, r\slash 2)$}.
 By the inversion formula{,}
\begin{equation*}
f(\mathbf{x})-P_tf(\mathbf{x})
=c_k^{-1}\int_{\mathbb{R}^N}(1\!-\hspace{-.5mm}e^{-\hspace{.25mm}t\hspace{.25mm}{\|}\xi{\|}})\,{E}(i\xi ,\mathbf{x} )\,\mathcal{F}f(\xi)\,{dw}(\xi)\,{,}
\end{equation*}
{hence}
\begin{equation}\label{uniform}
|f(\mathbf{x})-P_tf(\mathbf{x})|
\leq c_k^{-1}\int_{\mathbb{R}^N}{\bigl(}1\!-\hspace{-.5mm}e^{-\hspace{.25mm}t\hspace{.25mm}{\|}\xi{\|}}{\bigr)}\,|\mathcal{F}f(\xi)|\,{dw}(\xi)\,.
\end{equation}
{As} $\mathcal{F}\hspace{-.25mm}f\!\in\hspace{-.5mm}\mathcal{S}(\mathbb{R}^N)$,  \eqref{uniform} implies that there is \hspace{.5mm}$t_0\!=\hspace{-.5mm}t_0(\mathbf{x}_0,\hspace{-.25mm}r,\hspace{-.25mm}\varepsilon)\hspace{-.5mm}>\hspace{-.5mm}0$ \hspace{.5mm}such that
$$
\sup_{\mathbf{x}\in\mathbb{R}^N}|f(\mathbf{x})-P_tf(\mathbf{x})|{<}\varepsilon\quad{\forall}\;0\hspace{-.5mm}<\hspace{-.5mm}t\hspace{-.5mm}<\hspace{-.5mm}{t_0}\hspace{.5mm}.
$$
In particular{, for every \hspace{.25mm}$0\hspace{-.5mm}<\hspace{-.5mm}t\hspace{-.5mm}<\hspace{-.5mm}t_0$\hspace{.25mm} and} for {every} \hspace{.25mm}$\mathbf{x}\hspace{-.5mm}\in\!B(\mathbf{x}_0,r\slash 4)${,} we have
\begin{equation*}\begin{split}
0&\leq\int_{{\|}\mathbf{x}-\mathbf{y}{\|>}r}\!{p}_t(\mathbf{x},\mathbf{y})\,{dw}(\mathbf{y})=1-\int_{{\|}\mathbf{x}-\mathbf{y}{\|\leq}r}\!{p}_t(\mathbf{x},\mathbf{y})\,{dw}(\mathbf{y})\\
&{\leq}f(\mathbf{x})-\int_{{\|}\mathbf{x}-\mathbf{y}{\|\le}r}\!{p}_t(\mathbf{x},\mathbf{y}) f(\mathbf{y})\,{dw}(\mathbf{y})\leq|f(\mathbf{x})-P_tf(\mathbf{x})|{<}\varepsilon\,.
\end{split}\end{equation*}
{A straightforward compactness argument allows us to conclude.}
\end{proof}
The following {results} follow from \eqref{Poisson_low_up}, \eqref{PHL}, and Proposition \ref{PoissonApprox}.
\begin{corollary}\label{PoissonCont}
{Let $f$ be a bounded continuous function on $\mathbb{R}^N$. Then its Poisson integral \,$u(t,\mathbf{x})=P_tf(\mathbf{x})$ is also bounded and continuous on \,$[0,\infty)\times\mathbb{R}^N$.}
\end{corollary}
\begin{corollary}\label{PoissonConv}
Let $f\in L^p({dw})$ {with} $1\leq p\leq \infty$. Then for  almost every \,$\mathbf{x}\!\in\hspace{-.5mm}\mathbb{R}^N$
$$
\lim_{t\to0}\,\sup_{{\|}\mathbf{y}-\mathbf{x}{\|}<t}\,\bigl|P_tf(\mathbf{y})-f(\mathbf{x})\bigr|=0.
$$
\end{corollary}

\begin{remark}\normalfont The assertion of Proposition \ref{PoissonApprox} remains valid with the same proof if
$p_t(\mathbf x,\mathbf y)$  is replaced by $\Phi_t(\mathbf x,\mathbf y)=\tau_{\mathbf x} \Phi_t(-\mathbf y)$,
where $\Phi\in\mathcal S(\mathbb R^N)$ is radial, nonnegative, and $\int \Phi (\mathbf x)\, dw(\mathbf x)=1$.
\end{remark}

\section{Conjugate harmonic functions - subharmonicity}\label{SectionSub}

For $\sigma\in G$ let $f^\sigma(\mathbf{x})=f(\sigma(\mathbf{x}))$. It is easy to check that
\begin{equation}\label{eqsigma}T_{\xi} f^\sigma(\mathbf{x})= (T_{\sigma\xi} f)^\sigma (\mathbf{x}), \ \  \sigma\in G,  \ \mathbf{x}, \xi\in\mathbb{R}^N,
\end{equation}

$$({\Delta} f^\sigma) (\mathbf{x}) = ({\Delta} f)^\sigma (\mathbf{x}).$$
Let $\{\sigma_{ij}\}_{i,j=1}^N$ denote the matrix of $\sigma\in G$ written  in the canonical basis $e_1,{\dots},e_N$ of $\mathbb{R}^N$. Clearly, $\{\sigma_{ij}\}\in O(N)$.

\begin{lemma}
Assume that $\mathbf{u}(x_0,\mathbf{x})=(u_0(x_0,\mathbf{x}), u_1(x_0,\mathbf{x}),{\dots},u_N(x_0,\mathbf{x}))$ satisfies the Cauchy-Riemann equations \eqref{C-R}. For $\sigma\in G$ set
\begin{equation}\label{relation}
u_{\sigma, 0}(x_0,\mathbf{x})= u_0(x_0,\sigma{(}\mathbf{x}{)}), \ \ \ u_{\sigma, j}(x_0,\mathbf{x})= \sum_{i=1}^N \sigma_{ij} u_i(x_0,\sigma (\mathbf{x})), \ j=1,2,{\dots},N.
\end{equation}
Then $\mathbf{u}_\sigma(x_0,\mathbf{x})=( u_{\sigma, 0}(x_0,\mathbf{x}),  u_{\sigma, 1}(x_0,\mathbf{x}),{\dots}, u_{\sigma, N}(x_0,\mathbf{x}))$ satisfies the Cauchy-Riemann equations.
Moreover,
\begin{equation}\label{u_sigma}| \mathbf{u}_\sigma (x_0,\mathbf{x})|=|\mathbf{u}(x_0, \sigma (\mathbf{x}))|.
\end{equation}
\end{lemma}
\begin{proof} Let $1\leq k,j\leq N$. Then
\begin{equation}\begin{split}\label{eq0}
T_{k}u_{\sigma, j}(x_0,\mathbf{x})
&= \sum_{i=1}^N \sigma_{ij}T_k(u_i(x_0,\sigma \cdot))(\mathbf{x})
=\sum_{i=1}^N \sigma_{ij}\sum_{\ell=1}^N \sigma_{\ell k}(T_\ell u_i)(x_0,\sigma( \mathbf{x})),\end{split}\end{equation}
 and, similarly,

\begin{equation}\begin{split}\label{eq1}
T_{j}u_{\sigma, k}(x_0,\mathbf{x})
&=\sum_{i=1}^N \sigma_{ik}\sum_{\ell=1}^N \sigma_{\ell j}(T_\ell u_i)(x_0,\sigma (\mathbf{x})).\\
\end{split}\end{equation}
Recall that $T_{\ell}u_i=T_{i}u_\ell$. Hence,  \eqref{eq1} becomes
\begin{equation}\begin{split}\label{eq2}
T_{j}u_{\sigma, k}(x_0,\mathbf{x})
&=\sum_{i=1}^N \sigma_{ik}\sum_{\ell=1}^N \sigma_{\ell j}(T_i u_\ell )(x_0,\sigma (\mathbf{x})).\\
\end{split}\end{equation}
Now we see that \eqref{eq0} and \eqref{eq2} are equal.
The proof that $T_ku_{\sigma, 0}=T_0u_{\sigma, k}$ is straightforward. The second  equality of \eqref{C-R} follows directly from \eqref{eq2}  and the fact that $\sigma^{-1}=\sigma^*$.

Since $\{\sigma_{ij}\}\in O(N)$,
\begin{equation}\begin{split}
|u_{\sigma, 0}(x_0,\mathbf{x})|^2+ \sum_{j=1}^N|u_{\sigma, j}(x_0,\mathbf{x})|^2
&= |u_{ 0}(x_0,\sigma (\mathbf{x}))|^2+ \sum_{j=1}^N\left|\sum_{i=1}^N \sigma_{ij} u_{ i}(x_0,\sigma (\mathbf{x}))\right|^2\\
&=  |u_{ 0}(x_0,\sigma (\mathbf{x}))|^2+ \sum_{i=1}^N | u_{ i}(x_0,\sigma (\mathbf{x}))|^2,\\
\end{split}\end{equation}
which proves \eqref{u_sigma}.
\end{proof}
Let
\begin{equation}\label{functionF}
F(t,\mathbf{x})= \{\mathbf{u}_\sigma (t,\mathbf{x})\}_{\sigma\in  G}.
\end{equation}
We shall always assume that $\mathbf{u}$ and $\mathbf{u}_\sigma$ are related by \eqref{relation}.
Then, by \eqref{u_sigma},
$$ |F(x_0,\mathbf{x})|^2=\sum_{\sigma\in G} \sum_{\ell=0}^N |u_{\sigma, \ell} (x_0, \mathbf{x})|^2=\sum_{\sigma\in G}|\mathbf{u}_{\sigma}(x_0,\mathbf{x})|^2 = \sum_{\sigma\in G}|\mathbf{u} (x_0,\sigma (\mathbf{x}))|^2. $$
Observe that
$|F(x_0,\mathbf{x})|=|F(x_0,\sigma{(}\mathbf{x}{)})|$ for every $\sigma\in  G$.

Consequently, for every $\alpha\in R$,

\begin{equation}\begin{split}
\sum_{\sigma\in G} \sum_{\ell=0}^N \Big(u_{\sigma, \ell}(x_0,\mathbf{x}) -u_{\sigma,\ell}(x_0,\sigma_\alpha{(}\mathbf{x}{)})\Big)\cdot u_{\sigma,\ell}(x_0,\mathbf{x})\\
=\frac{1}{2}\sum_{\sigma\in G} \sum_{\ell=0}^N \Big|u_{\sigma, \ell}(x_0,\mathbf{x}) -u_{\sigma,\ell}(x_0,\sigma_\alpha{(}\mathbf{x}{)})\Big|^2.\\
\end{split}
\end{equation}

We shall need the following auxiliary lemma.
\begin{lemma}\label{auxiliary}
For every $\varepsilon >0$ there is $\delta>0$ such that for every  matrix $A=\{a_{ij}\}_{i,j=0}^{N}$ with real entries $a_{ij}$ one has
$$ \| A\|^2\leq  \varepsilon \Big((\text{\rm tr} A)^2+\sum_{i<j} (a_{ij}-a_{ji})^2 \Big)+ (1-\delta ) \| A\|_{\rm HS}^2,$$
 where $\| A\|_{\rm HS}$ denotes the Hilbert-Schmidt norm of $A$.
  \end{lemma}
\begin{proof}
The lemma was proved in \cite{Dz}. For the convenience of the reader we present a short proof.
 The inequality is known for trace zero symmetric $A$ (see Stein and Weiss \cite[Lemma 2.2]{SW}). By homogeneity we may assume that $\| A\|_{\rm HS}=1$.
 Assume that the inequality does not hold. Then  there is $\varepsilon >0$ such that  for every $ n>0$ there is  $A_n=\{a_{ij}^{\{n\}}\}_{i,j=0}^N, \ \| A_n\|_{\rm HS}=1$ such that

$$ \| A_n\|^2 > \varepsilon \Big((\text{tr} A_k)^2+\sum_{i<j} (a^{\{n\}}_{ij}-a^{\{n\}}_{ji})^2 \Big)+ \left(1-\frac{1}{n} \right) \|A_n\|^2_{\rm HS}.$$
Thus there is a subsequence  $n_{s}$ such that $A_{n_s}\to A$, $\| A\|_{\rm HS}=1$ and
$$ \| A\|^2\geq \varepsilon \Big((\text{tr} A)^2+\sum_{i<j} (a_{ij}-a_{ji})^2 \Big)+\|A\|^2_{\rm HS}.$$
 But then $A=A^*$ and $\text{tr} A=0$,  and so, $\| A\|^2\geq \| A\|_{\rm HS}^2.$ This contradicts  the already known inequality.
\end{proof}

 We now state and prove the main theorem of Section \ref{SectionSub}, which is the analog in the Dunkl setting of a Euclidean subharmonicity property (see \cite[Chapter VII, Section 3.1]{St1}) and which was proved in the product case in \cite[Proposition 4.1]{Dz}. Recall \eqref{operatorL} that $\mathcal L=T_0^2+{\Delta}$.

\begin{theorem}\label{subharmonic}
 There is an exponent $0<q<1$ which depends on $k$  such that if $\mathbf{u}=(u_0,u_1,{\dots},u_N)\in C^2$ satisfies the Cauchy-Riemann equations \eqref{C-R}, then  the function  $|F|^q$ is $\mathcal L$-subharmonic,
 that is, $\mathcal L(|F|^q)(t,\mathbf{x})\geq 0$ on the set
 where $|F|>0$.
\end{theorem}
\begin{proof}  Observe that $|F|^q$ is $C^2$ on the set where $|F|>0$.
 Let $\cdot$ denote the inner product in $\mathbb{R}^{(N+1)\cdot |G|}$.  For $j=0,1,{\dots},N$,  we have
 \begin{equation*}\begin{aligned}
        \partial_{e_j} |F|^q &=q|F|^{q-2}\Big((\partial_{e_j} F)\cdot F\Big)\\
        \partial_{e_j}^2 |F|^q&=q(q-2)|F|^{q-4}\Big((\partial_{e_j}F)\cdot F\Big)^2 +q|F|^{q-2}\Big((\partial_{e_j}^2F)\cdot
        F+|\partial_{e_j}F|^2\Big).
        \end{aligned}
   \end{equation*}
Recall that $|F(x_0,\mathbf{x})|=|F(x_0,\sigma{(}\mathbf{x}{)})|$. Hence,
\begin{equation}\label{eq366}
 \begin{aligned}
\mathcal L|F|^q&=  q(q-2)|F|^{q-4}\Big\{ \Big(
\sum_{j=0}^N\Big((\partial_{e_j}F)\cdot F\Big)^2\Big\}\\
& \ + q|F|^{q-2}\Big\{ \Big(\sum_{j=0}^N \partial_{e_j}^2 F+2\sum_{\alpha\in R^+} \frac{k(\alpha)}{\langle \alpha , \mathbf{x}\rangle} \partial_{\alpha} F \Big)\cdot F+\sum_{j=0}^N|\partial_{e_j} F|^2 \Big\}.\\
\end{aligned}
\end{equation}
 Since $T_jT_\ell =T_\ell T_j $, we conclude from (\ref{C-R}) applied to $\mathbf{u}_\sigma$ that for $\ell =0,1,{\dots},N$,  we have
 $$  \sum_{j=0}^N \partial_{e_j}^2 u_{\sigma,\ell}(x_0,\mathbf{x}) +2\sum_{\alpha\in R^+} \frac{k(\alpha)}{\langle \alpha , \mathbf{x}\rangle} \partial_\alpha u_{\sigma,\ell}(x_0,\mathbf{x})
 = \sum_{\alpha\in R^+} k(\alpha)\|\alpha\|^2\frac{u_{\sigma, \ell}(x_0,\mathbf{x}) -u_{\sigma,\ell}(x_0,\sigma_\alpha{(}\mathbf{x}{)})}{\langle \alpha, \mathbf{x}\rangle^2}.$$
Thus,
\begin{equation}\label{eq367}\begin{split}
&\Big(\sum_{j=0}^N \partial_{e_j}^2 F+2\sum_{\alpha\in R^+} \frac{k(\alpha)}{\langle \alpha , \mathbf{x}\rangle} \partial_{\alpha} F \Big)\cdot F\\
&=\sum_{\sigma\in  G} \sum_{\ell =0}^N
\Big(\sum_{j=0}^N \partial_{e_j}^2u_{\sigma, \ell} (x_0,\mathbf{x})
+ 2\sum_{\alpha\in R^+} \frac{k(\alpha)}{\langle \alpha , \mathbf{x}\rangle} \partial_\alpha u_{\sigma,\ell}(x_0,\mathbf{x})
\Big)u_{\sigma,\ell}(x_0,\mathbf{x})\\
&= \sum_{\sigma\in G} \sum_{\ell=0}^N \sum_{\alpha\in R^+} k(\alpha)\|\alpha\|^2\frac{u_{\sigma, \ell}(x_0,\mathbf{x}) -u_{\sigma,\ell}(x_0,\sigma_\alpha{(}\mathbf{x}{)})}{\langle \alpha, \mathbf{x}\rangle^2}u_{\sigma,\ell}(x_0,\mathbf{x})\\
&= \sum_{\alpha\in R^+} \frac{k(\alpha)\|\alpha\|^2}{\langle \alpha, \mathbf{x}\rangle^2}\sum_{\sigma\in G} \sum_{\ell=0}^N \Big(u_{\sigma, \ell}(x_0,\mathbf{x}) -u_{\sigma,\ell}(x_0,\sigma_\alpha{(}\mathbf{x}{)})\Big)u_{\sigma,\ell}(x_0,\mathbf{x})\\
&=\frac{1}{2} \sum_{\alpha\in R^+} \frac{k(\alpha)\|\alpha\|^2}{\langle \alpha, \mathbf{x}\rangle^2}\sum_{\sigma\in G} \sum_{\ell=0}^N \Big(u_{\sigma, \ell}(x_0,\mathbf{x}) -u_{\sigma,\ell}(x_0,\sigma_\alpha{(}\mathbf{x}{)})\Big)^2\\
\end{split}
\end{equation}
Thanks to (\ref{eq366}) and (\ref{eq367}), it suffices to prove that there is $0<q<1$ such that
\begin{equation}\label{eq3.1}
 \begin{aligned}
  &(2-q) \sum_{j=0}^N\Big((\partial_{e_j}F(x_0,\mathbf{x}))\cdot F(x_0,\mathbf{x})\Big)^2 \\
  &\leq \frac{1}{2}|F(x_0,\mathbf{x})|^{2}\sum_{\sigma\in G} \sum_{\ell=0}^N  \sum_{\alpha\in R^+} \frac{k(\alpha)\|\alpha\|^2}{\langle \alpha, \mathbf{x}\rangle^2}
  \Big(u_{\sigma, \ell}(x_0,\mathbf{x}) -u_{\sigma,\ell}(x_0,\sigma_\alpha{(}\mathbf{x}{)})\Big)^2\\
 &\ \ +|F(x_0,\mathbf{x})|^{2}\Big(\sum_{j=0}^N|\partial_{e_j} F(x_0,\mathbf{x})|^2\Big).
 \end{aligned}
\end{equation}
Set
$$B_\sigma=\left[\begin{array}{cccc}
\partial_{e_0}  u_{\sigma, 0} &\partial_{e_0}  u_{\sigma , 1}& {\dots} &\partial_{e_0}   u_{\sigma , N}\\
\partial_{e_1}  u_{\sigma, 0} &\partial_{e_1}  u_{\sigma , 1}& {\dots} &\partial_{e_1}   u_{\sigma , N}\\
 \  & \  & {\dots} & \ \\
 \partial_{e_N}  u_{\sigma, 0} &\partial_{e_N}  u_{\sigma , 1}& {\dots} &\partial_{e_N}   u_{\sigma , N}\\
\end{array}\right].$$
Let $\mathbf B=\{ B_\sigma\}_{\sigma\in G}$ be matrix with $N+1$ rows and $(N+1)\cdot | G|$ columns. It represents a linear operator (denoted by $\mathbf B$) from $\mathbb{R}^{(N+1)\cdot | G|}$ into $\mathbb{R}^{1+N}$. Let $\| \mathbf B\|$ be its norm.

Observe that for $0<q<1$ we have
$$ (2-q) \sum_{j=0}^N\Big((\partial_{e_j}F)\cdot F\Big)^2
\leq (2-q)|F|^2  \| \mathbf B\|^2 ,$$
$$ | F|^{2} \sum_{j=0}^N|\partial_{e_j} F|^2 = |F|^2\| \mathbf B\|_{\text{\rm HS}}^2.$$
Clearly,
$$\| \mathbf B\|^2\leq \sum_{\sigma\in G} \| B_\sigma\|^2, \ \ \ \| \mathbf B\|_{\text{\rm HS}}^2=\sum_{\sigma\in G}\| B_\sigma\|_{\text{\rm HS}}^2.$$
Therefore the inequality  (\ref{eq3.1}) will be proven if we show that
\begin{equation}
\label{eq3.2}\begin{split}
(2-q)\sum_{\sigma\in  G}\|B_{\sigma}\|^2
&\leq \sum_{\sigma\in  G}  \| B_\sigma\|_{\text{\rm HS}}^2\\
&\ \ +  \frac{1}{2}\sum_{\sigma\in G} \sum_{\ell=0}^N  \sum_{\alpha\in R^+} \frac{k(\alpha)\|\alpha\|^2}{\langle \alpha, \mathbf{x}\rangle^2}
  \Big(u_{\sigma, \ell}(x_0,\mathbf{x}) -u_{\sigma,\ell}(x_0,\sigma_\alpha{(}\mathbf{x}{)})\Big)^2.\\
\end{split}\end{equation}
Recall that
$$\gamma= \sum_{j=1}^N \sum_{\alpha\in R^+} \frac{k(\alpha) \langle \alpha ,e_j\rangle^2}{\|\alpha\|^2}=\sum_{j=0}^N \sum_{\alpha\in R^+} \frac{k(\alpha) \langle \alpha ,e_j\rangle^2}{\|\alpha\|^2}$$
(see \eqref{gamma}).
{By a}pplying {first} the Cauchy-Riemann equations (\ref{C-R}) and {next} the {Cauchy-}Schwarz inequality, we obtain
\begin{equation}\label{eq368}\begin{split}
( \text{\rm tr} B_\sigma )^2
& = \left(-\sum_{j=1}^N\sum_{\alpha\in R^+} k(\alpha ) \langle \alpha, e_j\rangle \frac{ u_{\sigma, j}(x_0,\mathbf{x})-u_{\sigma, j}(x_0, \sigma_\alpha{(}\mathbf{x}{)})}{\langle \alpha, \mathbf{x}\rangle}\right)^2\\
&\leq \Big(\sum_{j=1}^N \sum_{\alpha\in R^+} \frac{k(\alpha) \langle \alpha ,e_j\rangle^2}{\|\alpha\|^2}\Big) \Big(\sum_{j=1}^N \sum_{\alpha\in R^+}\|\alpha\|^2k(\alpha)  \frac{\big(u_{\sigma, j}(x_0,\mathbf{x}) -u_{\sigma , j}(x_0,\sigma_\alpha{(}\mathbf{x}{)})\big)^2}{\langle \alpha , \mathbf{x}\rangle^2}\Big)\\
&\leq \gamma\sum_{j=0}^N \sum_{\alpha\in R^+}\|\alpha\|^2k(\alpha)  \frac{\big(u_{\sigma, j}(x_0,\mathbf{x}) -u_{\sigma , j}(x_0,\sigma_\alpha{(}\mathbf{x}{)})\big)^2}{\langle \alpha , \mathbf{x}\rangle^2}.
\end{split}\end{equation}
Utilising  again the Cauchy-Riemann equations \eqref{C-R}, we get
\begin{equation}\label{eq369}\begin{split}
 & \sum_{0\leq i<j\leq N} \big(\partial_{e_i}u_{\sigma , j}(x_0, \mathbf{x}) -\partial_{e_j}  u_{\sigma ,i}(x_0,\mathbf{x})\big)^2 \\
 &=\sum_{j=1}^N\Big(\sum_{\alpha\in R^+} k(\alpha ) \langle \alpha, e_j\rangle \frac{ u_{\sigma, 0}(x_0,\mathbf{x})-u_{\sigma , 0}(x_0,\sigma_\alpha{(}\mathbf{x}{)})}{\langle \alpha , \mathbf{x}\rangle}\Big)^2 \\
 &\ \  +\sum_{1\leq i<j\leq N} \Big(\sum_{\alpha\in R^+} - k(\alpha ) \langle \alpha, e_i\rangle
 \frac{ u_{\sigma, j}(x_0,\mathbf{x})-u_{\sigma , j}(x_0,\sigma_\alpha{(}\mathbf{x}{)})}{\langle \alpha , \mathbf{x}\rangle}\\
 &\hskip3cm+  k(\alpha ) \langle \alpha, e_j\rangle
 \frac{ u_{\sigma, i}(x_0,\mathbf{x})-u_{\sigma , i}(x_0,\sigma_\alpha{(}\mathbf{x}{)})}{\langle \alpha , \mathbf{x}\rangle} \Big)^2\\
 &\leq 2 \Big(\sum_{j=0}^N \sum_{\alpha\in R^+} \frac{k(\alpha) \langle \alpha ,e_j\rangle^2}{\|\alpha\|^2}\Big) \Big(\sum_{j=0}^N \sum_{\alpha\in R^+}\|\alpha\|^2k(\alpha)  \frac{\big(u_{\sigma, j}(x_0,\mathbf{x}) -u_{\sigma , j}(x_0,\sigma_\alpha{(}\mathbf{x}{)})\big)^2}{\langle \alpha , \mathbf{x}\rangle^2}\Big).
\end{split} \end{equation}
Using  the auxiliary Lemma \ref{auxiliary} together with (\ref{eq368}) and (\ref{eq369}) we have that for every $\varepsilon >0$  there is $0<\delta <1$ such that
\begin{equation}\begin{split}\label{eq3.44}
 \sum_{\sigma\in G} \| B_\sigma\|^2
 &\leq (1-\delta)\sum_{\sigma\in G}\| B_\sigma\|_{\text{\rm HS}}^2 \\
 &\  \ +  3 \varepsilon \gamma \sum_{\sigma \in G}\sum_{j=0}^N \sum_{\alpha\in R^+}\|\alpha\|^2k(\alpha)  \frac{\big(u_{\sigma, j}(x_0,\mathbf{x}) -u_{\sigma , j}(x_0,\sigma_\alpha{(}\mathbf{x}{)})\big)^2}{\langle \alpha , \mathbf{x}\rangle^2}.
\end{split}\end{equation}
Taking $\varepsilon>0$ such that $ 3 \varepsilon \gamma \leq \frac{1}{4}$ and utilizing (\ref{eq3.44}) we deduce that (\ref{eq3.2}) holds for $q$ such that  $(1-\delta)\leq (2-q)^{-1}$.
\end{proof}

\section{Harmonic functions in the Dunkl setting.}\label{HarmFunt}

In this section we characterize certain $\mathcal L$-harmonic functions in the half-space $\mathbb{R}_+^{{1+N}}$  by adapting the classical proofs (see, e.g., \cite{FS}, \cite{St1} and \cite{SW}). Let us first construct an auxiliary  barrier function.
\smallskip

{\bf Barrier function.}
For fixed $\delta>0$ let $v_1,{\dots},v_s\in\mathbb{R}^N$ be a set of vectors of the unit sphere in $ S^{N-1}=\{ \mathbf{x}\in\mathbb{R}^N: \|\mathbf{x}\|=1\}$ which forms a  $\delta$--net on $S^{N-1}$. Let $M,\varepsilon >0$. Define
\begin{equation}\label{V_m}
 \mathcal V_m(x_0,\mathbf{x})=2M\varepsilon x_0+\varepsilon {E}\Big(\frac{\varepsilon\pi}{4}\mathbf{x},v_m\Big)\cos \Big(\frac{\varepsilon\pi}{4}x_0\Big)
\end{equation}
 (cf. \cite[Chapter VII, Section 1.2]{St1} in the classical setting). The function $\mathcal V_m$ is $\mathcal L$-harmonic and strictly positive on $[0,\varepsilon^{-1}]\times\mathbb{R}^N$. Set
$$\mathcal V(x_0,\mathbf{x})=\sum_{m=1}^s \mathcal V_m(x_0,\mathbf{x}).$$
By Corollary \ref{EstimateNearDiagonal},
\begin{equation}
\lim_{\|\mathbf{x}\|\to\infty}\mathcal V(x_0,\mathbf{x})=\infty \ \ \
\text{\rm uniformly in }
x_0\in [0,\varepsilon^{-1}].
\end{equation}

{\bf Maximum principle and the mean value property.}
As we have already remarked in Section \ref{preliminaries}, the operator $\mathcal L$ is the Dunkl-Laplace operator associated with the root system $R$ as a subset of
$\mathbb{R}^{1+N}=\mathbb{R}\times\mathbb{R}^N$. We shall denote the element of $\mathbb{R}^{1+N}$ by
$\boldsymbol x=(x_0,\mathbf{x})$. The associated measure will be denoted by $\boldsymbol {w}$. Clearly,
$d\boldsymbol {w}(\boldsymbol x)={w}(\mathbf{x})\,d\mathbf{x}\, dx_0$. Moreover, ${E}(\boldsymbol x,\boldsymbol y)=e^{x_0 y_0} {E}(\mathbf{x},\mathbf{y})$. We shall slightly abuse notation and use the same letter $\sigma$ for the action of the group $G$ in $\mathbb{R}^{1+N}$, so $\sigma(\boldsymbol x)=\sigma (x_0,\mathbf{x})=(x_0,\sigma (\mathbf{x}))$.

The following weak maximum principle for $\mathcal L$-subharmonic functions  was actually proved in  Theorem 4.2 of R\"osler \cite{Roesler2}.
\begin{theorem}\label{max_priciple} Let $\Omega\subset\mathbb{R}^{1+N}$ be open, bounded, and $\overline{ \Omega}\subset (0,\infty)\times\mathbb{R}^N$.
Assume that $\Omega$ is $ G$-invariant, that is, $(x_0,\sigma(\mathbf{x})) \in \Omega$ for $(x_0,\mathbf{x})\in \Omega$  and all $\sigma \in  G$.
Let  $f\in C^2(\Omega)\cap C( \overline{\Omega})$
 be real-valued and $\mathcal L$-subharmonic. Then
$$ \max_{\overline{ \Omega}} f=\max_{\partial \Omega}f. $$
\end{theorem}

Let $ f^{\{r\}}(\boldsymbol x)=\chi_{B(0, r)}(\boldsymbol x)$ be the characteristic function of the ball in $\mathbb{R}^{1+N}$.
Set
$$ f(r,\boldsymbol x,\boldsymbol y)= \tau_{\boldsymbol x} f^{\{r\}}(-\boldsymbol y).$$
Clearly, $0\leq f(r,\boldsymbol x,\boldsymbol y) \leq 1$.
The following mean value theorem was proved in  \cite[Theorem 3.2]{GR}.

\begin{theorem}\label{mean-value}
  Let $\Omega\subset\mathbb{R}^{1+N}$ be an open and $G$-invariant  set and let $u$ be a $C^2$ function in $\Omega$. Then $u$ is $\mathcal L$-harmonic if and only if $u$ has the following  mean value property: for all $\boldsymbol x\in\Omega$ and $\rho>0$ such that
  $B(\boldsymbol x , \rho )\subset \Omega$ we have
  $$ u(\boldsymbol x)=\frac{1}{\boldsymbol {w}(B(0,r))}\int_{\Omega} f(r,\boldsymbol x,\boldsymbol y) u(\boldsymbol y)d\boldsymbol {w}(\boldsymbol y)\ \ \ \text{for} \ 0<r<\rho\slash 3.$$
\end{theorem}

{\bf Characterizations of $\mathcal L$-harmonic functions in the upper half-space.}

\begin{theorem}\label{harmonic1}
  Suppose that $u$ is a $C^2$ function on $\mathbb{R}_+^{{1+N}}$. Then $u$ is a Poisson integral of a bounded function on $\mathbb{R}^N$ if and only if $u$ is $\mathcal L$-harmonic and bounded.
\end{theorem}

\begin{proof}
  The proof is identical to that of Stein \cite{St1}. Clearly,
  the Poisson integral of a bounded function is bounded and $\mathcal L$-harmonic.
To prove the converse assume that $u$ is $\mathcal L$-harmonic and bounded, so $|u|\leq M$. Set $f_n(\mathbf{x})=u(\frac{1}{n} , \mathbf{x})$ and $u_n(x_0,\mathbf{x})=P_{x_0}f_n(\mathbf{x})$.
  Then $U_n(x_0,\mathbf{x})=u(x_0+\frac{1}{n},\mathbf{x})-u_n(x_0,\mathbf{x})$ is $\mathcal L$-harmonic,  $|U_n|\leq 2M$, continuous on $[0,\infty)\times\mathbb{R}^N$, and $U_n(0,\mathbf{x})=0$. We shall prove that $U_n\equiv 0$. Fix $(y_0, \mathbf{y})\in\mathbb{R}_+^{{1+N}}$. Set
  $$ U(x_0,\mathbf{x})= U_n(x_0,\mathbf{x})+\mathcal V(x_0,\mathbf{x})$$
  and consider the function $U$ on the closure of the  set $\Omega=(0,\varepsilon ^{-1})\times {B(0,R)}$, with $\varepsilon >0$ small and $R$ large enough. Then $U$ is $\mathcal L$-harmonic in $\Omega$, continuous on $\bar\Omega$,
  and positive on the boundary of the $\partial \Omega$. Thus, by the maximum principle, $U$ is positive in $\bar \Omega$, so
  $$ U_n(y_0,\mathbf{y})>-2M\varepsilon y_0 -\sum_{m=1}^s \varepsilon E\Big(\frac{\varepsilon\pi}{4}\mathbf{y},v_m\Big)\cos \Big(\frac{\varepsilon\pi}{4}y_0\Big).$$
Letting $\varepsilon\to 0$ we obtain $U_n(y_0,\mathbf{y})\geq 0$.
The same argument applied to $-u$ gives $-U_n(y_0,\mathbf{y})\geq 0$, so $U_n\equiv 0$, which can be written as
\begin{equation}\label{P1}
  u\Big(x_0+\frac{1}{n},\mathbf{x}\Big)=P_{x_0}f_n(\mathbf{x})=\int p_{x_0}(\mathbf{x},\mathbf{y})f_n(\mathbf{y})\, {dw}(\mathbf{y}).
\end{equation}
Clearly $|f_n|\leq M$, so by the *-weak compactness, there is  a subsequence $n_j$ and $f\in L^{\infty}(\mathbb{R}^N)$ such that for $\varphi \in L^1({dw})$, we have
$$\lim_{j\to\infty} \int \varphi(\mathbf{y})f_{n_j}(\mathbf{y})\, {dw}(\mathbf{y})=\int \varphi (\mathbf{y})f(\mathbf{y})\, {dw}(\mathbf{y}).$$
So,
\begin{equation*}
\begin{split}
u(x_0,\mathbf{x})&=\lim_{j\to\infty} u\Big(x_0+\frac{1}{n_j},\mathbf{x}\Big)=\lim_{j\to\infty} \int p_{x_0}(\mathbf{x},\mathbf{y})f_{n_j}(\mathbf{y})\, {dw}(\mathbf{y})\\
&= \int p_{x_0}(\mathbf{x},\mathbf{y})f(\mathbf{y})\, {dw}(\mathbf{y}).
\end{split}
\end{equation*}
\end{proof}

\begin{corollary}
  If $u$ is $\mathcal L$-harmonic and bounded in $\mathbb{R}_+^{{1+N}}$ then $u$ has a nontangential limit at almost every point of the boundary.
\end{corollary}

\begin{theorem}\label{LpBounds}
  Suppose that $u$ is a $C^2$-function on $\mathbb{R}_+^{{1+N}}$.
  If $1<p<\infty$ then $u$ is a Poisson integral of an $L^p({dw})$ function if and only if $u$ is $\mathcal L$-harmonic and
  \begin{equation}\label{HpCondition}
    \sup_{x_0>0} \|u(x_0,\cdot)\|_{L^p({dw})}<\infty.
  \end{equation}
  If $p=1$ then $u$ is a Poisson integral of a bounded measure $\omega$ if and only if $u$ is $\mathcal L$-harmonic and
  \begin{equation}\label{H1Condition}
    \sup_{x_0>0} \|u(x_0,\cdot)\|_{L^1({dw})}<\infty.
  \end{equation}
 Moreover, if $u^*\in L^1({dw})$ (see~\eqref{star}), then $d\omega(\mathbf{x})= f(\mathbf{x}){dw}(\mathbf{x})$, where $f\in L^1({dw})$.
\end{theorem}

\begin{proof}
Assume that either \eqref{HpCondition} or \eqref{H1Condition} holds. Then, by Theorem \ref{mean-value},  for every $\varepsilon>0$
  \begin{equation}\label{L-infty}
    \sup_{x_0>0}\sup_{\mathbf{x}\in\mathbb{R}^N} |u(x_0+\varepsilon, \mathbf{x})|\leq C_\varepsilon <\infty.
  \end{equation}
  Set $f_n(\mathbf{x})=u(\frac{1}{n},\mathbf{x})$. From Theorem \ref{harmonic1} we conclude that $u(\frac{1}{n}+x_0, \mathbf{x})=P_{x_0}f_n(\mathbf{x})$.
  Moreover, there is a subsequence $n_j$ such that  $f_{n_j}$ converges  weakly-* to $f\in L^p({dw})$ (if $1<p<\infty$) or to a measure $\omega$ (if $p=1$). In both cases $u$ is the Poisson integral either  of $f$ or $\omega$. If additionally $u^*\in L^1({dw})$, then the measure $\omega$ is absolutely continuous with respect to ${dw}$.
  \end{proof}

{\bf Proof of a part of Theorem \ref{main1}.} We are now in a position to prove a part of Theorem \ref{main1}, which is stated in the following proposition. The converse is proven at the very end of Section \ref{Atomic} (see Proposition \ref{converse}).

\begin{proposition}\label{PropMain1Part1}
Assume that $\mathbf{u}\in\mathcal{H}^1_k$.
Then
\begin{equation}\label{eqMax}\| \mathbf{u}^*\|_{L^1({dw})}\leq C\| \mathbf{u}\|_{\mathcal{H}^1_k}.
\end{equation}
\end{proposition}

\begin{proof}
Fix $\varepsilon>0$. Set $u_{j,\varepsilon}(x_0,\mathbf{x})= u_j(\varepsilon +x_0, \mathbf{x})$, $f_{j,\varepsilon}(\mathbf{x})=u_j(\varepsilon,\mathbf{x})$. Then, by Theorem \ref{mean-value}, the $\mathcal L$-harmonic function $u_{j,\varepsilon} (x_0,\mathbf{x})$ is  bounded  and continuous  on the closed set $[0,\infty)\times\mathbb{R}^N$. In particular  $f_{j,\varepsilon}\in L^\infty\cap L^1({dw})\cap C^2$. By Theorem \ref{harmonic1},
$$u_{j, \varepsilon} (x_0,\mathbf{x})=P_{x_0}f_{j,\varepsilon}(\mathbf{x}).$$
It is not difficult to conclude using \eqref{DtDxDyPoisson} (with $m=0$) that
$\lim_{\| (x_0,\mathbf{x})\|\to\infty} |u_{j,\varepsilon }(x_0,\mathbf{x})|=0$.
Thus also $\lim_{\| \mathbf{x}\|\to\infty }f_{j,\varepsilon}(\mathbf{x})=0$.
Set $\mathbf{u}_\varepsilon= (u_{0,\varepsilon}, u_{1,\varepsilon},{\dots}, u_{N,\varepsilon})$. Clearly,
$\mathbf{u}_\varepsilon \in\mathcal{H}^1_k$.
Let $F_\varepsilon(x_0,\mathbf{x})=F (\varepsilon+x_0,\mathbf{x})$, where $F(x_0,\mathbf{x})$ is defined by \eqref{functionF}. Set $\boldsymbol f_{\varepsilon} (\mathbf{x})=|F(\varepsilon, \mathbf{x})|$. Let $0<q<1$ be as in Theorem \ref{subharmonic} and $p=q^{-1}>1$.
Observe that the function $|F_\varepsilon (x_0,\mathbf{x})|^q- P_{x_0} (\boldsymbol f_{\varepsilon}^q ) (\mathbf{x})$ vanishes for $x_0=0$ and
$$ \lim_{\|(x_0,\mathbf{x})\|\to\infty } \Big(|F_\varepsilon (x_0,\mathbf{x})|^q- P_{x_0} (\boldsymbol f_{\varepsilon}^q) (\mathbf{x})\Big)=0.$$
So, by Theorem \ref{subharmonic} and the maximum principle (see Theorem \ref{max_priciple}),
 \begin{equation}\label{eq55}
 |\mathbf{u}(\varepsilon+x_0,\mathbf{x})|^q\leq |F_\varepsilon (x_0,\mathbf{x})|^q\leq  P_{x_0}(\boldsymbol f_\varepsilon ^q)(\mathbf{x}).
 \end{equation}
Set $\mathbf{u}^*_{\varepsilon}(\mathbf{x})=\sup_{\|\mathbf{x}-\mathbf{y}\|<x_0} |\mathbf{u}(\varepsilon+x_0,\mathbf{y})|$. Then, by \eqref{eq55} and \eqref{PHL},
$$ \| \mathbf{u}_\varepsilon^*\|_{L^1({dw})}\leq C_p\| \boldsymbol f_\varepsilon ^q\|_{L^p({dw})}^p=C_p\| \boldsymbol f_\varepsilon \|_{L^1({dw})}\leq C_p\| \mathbf{u}\|_{\mathcal{H}^1_k}.$$
Since $\mathbf{u}^*_\varepsilon (\mathbf{x})\to \mathbf{u}^*(\mathbf{x})$ as $\varepsilon \to 0$  and the convergence is monotone, we use the Lebesgue monotone convergence  theorem and get \eqref{eqMax}.

\end{proof}

From  Theorem \ref{LpBounds} and Proposition \ref{PropMain1Part1} we obtain the following corollary.
\begin{corollary}\label{Limits}
If $\mathbf{u}\in\mathcal{H}^1_k$, then there are $f_j\in L^1({dw})$, $j=0,1,{\dots},N$,  such that
$|f_j(\mathbf{x})|\leq \mathbf{u}^*(\mathbf{x})$ and
$u_j(x_0,\mathbf{x})=P_{x_0}f_j(\mathbf{x})$. Moreover, the limit $\lim_{x_0\to 0} u_j(x_0,\mathbf{x})=f_j(\mathbf{x})$ exists in $L^1({dw})$.
\end{corollary}

\section{Riesz transform characterization of $H^1_{{\Delta}}$}\label{SectionRiesz}

{\bf Riesz transforms.}
The Riesz transforms in the Dunkl setting are defined by
$$\mathcal{F} (R_jf)(\xi) = - i\frac{\xi_j}{\|\xi\|} (\mathcal{F} f)(\xi), \ \ j=1,2,{\dots},N.$$
  They are bounded operators on $L^2({dw})$. Clearly,
  $$R_jf=-T_{e_j} (-{\Delta})^{-1\slash 2} f =- \lim_{\varepsilon\to 0, \, M\to\infty} c  \int_{\varepsilon}^M T_{e_j} e^{t{\Delta}}f\frac{dt}{\sqrt{t\,}},$$
  and the convergence is in $L^2({dw})$ for $f\in L^2({dw})$.
  It follows from \cite{AS} that $R_j$ are bounded operators on $L^p({dw})$ for $1<p<\infty$.

  Our task is to define $R_jf$ for $f \in L^1({dw})$. To this end we set
  $$\mathcal T_k=\{ \varphi \in L^2({dw}):
  (\mathcal{F}\varphi) (\xi)(1+\|\xi\|)^n\in L^2({dw}), \ n=0,1,2, {\dots}\}.$$
  It is not difficult to check that if $\varphi \in\mathcal T_k$, then
   $\varphi\in C_0(\mathbb{R}^N)$ and
  $R_j\varphi\in C_0(\mathbb{R}^N)\cap L^2({dw})$. Moreover, for fixed $\mathbf{y}\in\mathbb{R}^N$ the function $p_t(\mathbf{x},\mathbf{y})$ belongs to $\mathcal T_k$.
  Now   $R_jf $ for $f\in L^1({dw})$ is defined in a weak sense as a functional on $\mathcal T_k$, by
  $$ \langle R_jf,\varphi\rangle = -\int_{\mathbb{R}^N} f(\mathbf{x})R_j\varphi(\mathbf{x})\, {dw}(\mathbf{x}).$$

{\bf Proof of Theorem \ref{main4}.} Assume that $f\in L^1({dw})$ is such that $R_jf$ belong to $L^1({dw})$ for $j=1,2,{\dots},N$. Set $f_0(\mathbf{x})=f(\mathbf{x})$, $f_j(\mathbf{x})=R_jf(\mathbf{x})$, $u_0(x_0,\mathbf{x})=P_{x_0}f(\mathbf{x})$, $u_j(x_0,\mathbf{x})=P_{x_0}f_j(\mathbf{x})$. Then $\mathbf{u}=(u_0,u_1,{\dots},u_n)$ satisfies \eqref{C-R}.  Moreover,
$$ \sup_{x_0>0}\int_{\mathbb{R}^N}| u_j(x_0,\mathbf{x})|\, {dw}(\mathbf{x})\leq \| f_j\|_{L^1({dw})}  \ \ \text{\rm for  } j=0,1,{\dots},N.$$
Thus $\mathbf{u}\in\mathcal{H}^1_k$ and
$$\| f\|_{H^1_{{\Delta}}}=\| \mathbf{u}\|_{\mathcal{H}^1_k}\leq \| f\|_{L^1({dw})}+\sum_{j=1}^N \| R_jf\|_{L^1({dw})}.$$

We turn to prove the converse. Assume that $f_0\in H^1_{{\Delta}}$. By the definition of $H^1_{{\Delta}}$ there is a system
$\mathbf{u}=(u_0,u_1,{\dots}, u_N)\in\mathcal{H}^1_k$ such that $f_0(\mathbf{x})=\lim_{x_0\to 0}u_0(x_0,\mathbf{x})$ (convergence in $L^1({dw}))$. Set $f_j(\mathbf{x})=\lim_{x_0\to 0}u_j(x_0,\mathbf{x})$, where limits exist in $L^1({dw})$ (see Corollary \ref{Limits}). We have $u_j(x_0,\mathbf{x})=P_{x_0} f_j(\mathbf{x})$. It suffices to prove that $R_jf_0=f_j$. To this end, for $\varepsilon >0$, let $f_{j,\varepsilon} (\mathbf{x})=u_j(\varepsilon ,\mathbf{x})$, $u_{j,\varepsilon}(x_0,\mathbf{x})=u_j(x_0+\varepsilon,\mathbf{x})$.  Then $f_{j,\varepsilon} \in L^1({dw})\cap C_0(\mathbb{R}^N)$. In particular $f_{j,\varepsilon} \in L^2({dw})$. Set $g_j=R_jf_{0,\varepsilon}$, $v_j(x_0,\mathbf{x})=P_{x_0}g_j(\mathbf{x})$. Then $\mathbf v=(u_{0,\varepsilon}, v_1,{\dots}, v_N)$ satisfies the Cauchy-Riemann equations \eqref{C-R}. Therefore, $T_j u_{0,\varepsilon}(x_0,\mathbf{x})=T_0u_{j,\varepsilon}(x_0,\mathbf{x})=T_0v_j(x_0,\mathbf{x})$.
Hence, $u_{j,\varepsilon}(x_0,\mathbf{x})-v_j(x_0,\mathbf{x})=c_j(\mathbf{x})$.  But $\lim_{x_0\to\infty } u_{j,\varepsilon} (x_0,\mathbf{x})=0= \lim_{x_0\to\infty } v_{j}(x_0,\mathbf{x})$ for every $\mathbf{x}\in\mathbb{R}^N$. Consequently, $u_{j,\varepsilon} (x_0,\mathbf{x})=v_j(x_0,\mathbf{x})$.
Thus, $f_{j,\varepsilon }=R_jf_{0,\varepsilon}$. Since $\lim_{\varepsilon \to 0} f_{j, \varepsilon} = f_j$ in $L^1({dw})$ and $R_jf_{0,\varepsilon}\to Rf_0$ in the sense of distributions, we have $f_j=R_jf_0$.

\section{Inclusion $H^1_{(1,q,M)}\subset H^1_{{\Delta}}$}\label{Sect9}
In this section we show that the atomic space $H^1_{(1,q,M)}$ with $M>\mathbf{N}$  is contained in the Hardy space $H^1_{{\Delta}}$ and there exists $C=C_{k,q,M}$ such that
\begin{equation}\label{inAtom}
\| f\|_{H^1_{{\Delta}}}\leq C\| f\|_{H^1_{(1,q,M)}}.
\end{equation}
 Let $f\in H^1_{(1,q,M)}$. According to Theorem \ref{main4}, it is enough to show that $R_jf\in L^1({dw})$ and $\| R_j f\|_{L^1({dw})}\leq C\| f\|_{H^1_{(1,q,M)}}$. By the definition of the atomic space  there is a sequence $a_j$ of $(1,q,M)$ atoms and $\lambda_i\in\mathbb C$ such that $f=\sum_{i}\lambda_i a_i$ and $\sum_{i}|\lambda_i|\leq 2 \| f\|_{H^1_{(1,q,M)}}$. Observe that the series converges  in $L^1({dw})$, hence $R_jf=\sum_{i} \lambda_j R_ja_j$ in the sense of distributions. Therefore it suffices to prove that there is a constant $C>0$ such $\| R_ja\|_{L^1({dw})}\leq C$ for every $a$ being  a $(1,q,M)$-atom. Our proof follows ideas of \cite{HMMLY}. Let $b\in\mathcal{D}({\Delta}^M)$ and $B(\mathbf{y_0}, r)$ be as in the definition of $(1,q,M)$ atom. Since $R_j$ is bounded on $L^q({dw})$, by the H\"older inequality,
 we have
$$ \| R_ja\|_{L^1(\mathcal{O}(B(\mathbf{y_0}, 4r)))}\leq C.$$
In order to estimate $R_ja$ on the set $\mathcal{O}(B(\mathbf{y_0}, 4r))^c$ we write
\begin{equation*}
\begin{split}
R_ja
& =c''_k\int_0^\infty T_{j,\mathbf{x}} e^{t{\Delta}} a\frac{dt}{\sqrt{t\,}}\\
&=c''_k\int_0^{r^2} T_{j,\mathbf{x}} e^{t{\Delta}} a\frac{dt}{\sqrt{t\,}}+
c''_k\int_{r^2}^\infty T_{j,\mathbf{x}} e^{t{\Delta}} ({\Delta})^Mb\frac{dt}{\sqrt{t\,}}\\
&=c''_k\int_0^{r^2} T_{j,\mathbf{x}} e^{t{\Delta}} a\frac{dt}{\sqrt{t\,}}+
c''_k\int_{r^2}^\infty T_{j,\mathbf{x}}{\partial_t^M}e^{t{\Delta}}b\frac{dt}{\sqrt{t\,}}\\
&= R_{j,0}a+R_{j,\infty} a.
 \end{split}
 \end{equation*}
 Further, using \eqref{TxiDtHeat} with $m=0$  together with \eqref{growth}, we get
 \begin{equation}\begin{split}\label{a0}
 |R_{j,0} a(\mathbf{x})|
 &\leq C \int_0^{r^2} \int_{\mathbb{R}^N} t^{-1} {w}(B(\mathbf{y}, \sqrt{t\,}))^{-1}
 e^{-cd(\mathbf{x},\mathbf{y})^2\slash t}|a(\mathbf{y})|\, {dw}(\mathbf{y})dt\\
 &\leq C\frac{r^{\mathbf{N}+1}}
{d(\mathbf{x},\mathbf{y})^{\mathbf{N}+1}{w}(B(\mathbf{y}_0, r))}.
 \end{split}\end{equation}
To estimate $R_{j,\infty } a$ we recall that $\| b\|_{L^1({dw})}\leq r^{2M}$. Using \eqref{TxiDtHeat} with $m=M$, we obtain
 \begin{equation}\begin{split}\label{ainfty}
 |R_{j,\infty} a(\mathbf{x})|
 &\leq C \int_{r^2}^\infty \int_{\mathbb{R}^N} t^{-M-1} {w}(B(\mathbf{y}, \sqrt{t\,}))^{-1}
 e^{-cd(\mathbf{x},\mathbf{y})^2\slash t}|b(\mathbf{y})|\, {dw}(\mathbf{y})dt\\
 &\leq C\frac{r^{2M}}
{d(\mathbf{x},\mathbf{y})^{2M}{w}(B(\mathbf{y}_0, r))}.
 \end{split}\end{equation}
Obviously,  \eqref{a0} and \eqref{ainfty} combined with  \eqref{growth} imply  $ \| R_ja\|_{L^1(\mathcal{O}(B(\mathbf{y_0}, 4r))^c)}\leq C$.

\section{Maximal functions}\label{Max=Max}

Let $\Phi(\mathbf{x})$ be a radial continuous function such that $|\Phi(\mathbf{x}) |\leq C (1+\|\mathbf{x}\|)^{-\mathbf \kappa -\beta}$ with $\kappa >\mathbf{N}$. Let $\Phi_t(\mathbf{x})=t^{-N}\Phi(t^{-1}\mathbf{x})$ and $\Phi_t(\mathbf{x},\mathbf{y})=\tau_{\mathbf{x}}\Phi_t(-\mathbf{y})$. Then, by Corollary \ref{translation2},
$$ |\Phi_t(\mathbf{x},\mathbf{y})|\leq C V(\mathbf{x},\mathbf{y}, t)^{-1} \Big(1+\frac{d(\mathbf{x},\mathbf{y})}{t}\Big)^{-\beta}.$$
Set $\mathcal{M}_{\Phi,a}f(\mathbf{x})=\sup_{\|\mathbf{x}-\mathbf{y}\|<a t} |\Phi_tf(\mathbf{y})|$, where
$$\Phi_tf(\mathbf{x})=\Phi_t*f(\mathbf{x})=\int_{\mathbb{R}^N} \Phi_t(\mathbf{x},\mathbf{y})\, f(\mathbf{y})\, {dw}(\mathbf{y}).$$
If $a=1$, then we simply write $\mathcal{M}_{\Phi}$.
We say that $f\in H^1_{{\rm max},\Phi}$ if $\mathcal{M}_{\Phi} f\in L^1({dw})$. Then we set
$\| f\|_{H^1_{{\rm max},\Phi}}=\|\mathcal{M}_{\Phi} f\|_{L^1({dw})}$.

{\bf The space $\mathcal N$. }\label{SpaceN}  The space $H^1_{{\rm max},\Phi}$ is  related with the tent space $\mathcal N$.
\begin{definition}\normalfont
For $a>0$, $\lambda>\mathbf{N}$,  and a function $u(t,\mathbf{x})$ denote
$$u_a^*(\mathbf{x})=\sup_{\|\mathbf{x}-\mathbf{y}\|<at} |u(t,\mathbf{y})|, \ \ u^{**}_\lambda(\mathbf{x})=\sup_{\mathbf{y}\in\mathbb{R}^N,\, t>0} |u(t,\mathbf{y})|\Big(\frac{t}{\|\mathbf{y}-\mathbf{x}\|+t}\Big)^\lambda.$$
The tent space $\mathcal N$ is defined by
 $$\mathcal N_a=\{ u(t,\mathbf{x}): \| u\|_{\mathcal N_a}=\| u_a^*\|_{L^1({dw})}<\infty\}.$$
If $a=1$, then we write  $\mathcal N$, $\| u\|_{\mathcal N}$, and $u^*$ (cf.~\eqref{star}).
\end{definition}

\begin{lemma}\label{Nab}
 There are constants $C,C_\lambda,c_\lambda>0$ such that
 \begin{equation}\label{Nab_eq}
 \| u\|_{\mathcal N_a}\leq C\left(\frac{a+b}{b}\right)^{\mathbf{N}} \| u\|_{\mathcal N_b},
 \end{equation}
 \begin{equation}\label{maxlamba}
c_\lambda \| u\|_{\mathcal N}\leq \| u^{**}_\lambda \|_{L^1({dw})}\leq C_\lambda \| u\|_{\mathcal N}.
 \end{equation}
\end{lemma}

\begin{proof}
The proofs are the same as those in
\cite[Chapter II]{St2} and \cite[page 114]{FollanStein} .
\end{proof}

If $\Omega \subset\mathbb{R}^N$ is an open set, then the tent over $\Omega$ is given by
$$ \widehat\Omega=\Big( (0,\infty) \times \mathbb{R}^N\Big) \setminus \bigcup_{\mathbf{x} \in\Omega^c} \Gamma (\mathbf{x}), \ \ \text{\rm where } \Gamma (\mathbf{x})=\{ (t, \mathbf{y}): \| \mathbf{x}-\mathbf{y}\|<4t\}.$$

The space $\mathcal N$ admits the following atomic decomposition (see \cite{St2}).

\begin{definition}\label{N-atom}\normalfont
A function $A(t,\mathbf{x})$ is an atom for $\mathcal N$ if there is a ball $B$ such that

$\bullet$ $\text{\rm supp}\, A\subset \hat B$,

$\bullet$ $\| A\|_{L^\infty}\leq {w}(B)^{-1}$.
\end{definition}
Clearly, $\| A\|_{\mathcal N}\leq 1$ for every atom $A$  for $\mathcal N$.
Moreover, every $u\in\mathcal N$ can be written as
$ u=\sum_j \lambda_j A_j$, where $A_j$ are atoms for $\mathcal N$, $\lambda_j\in\mathbb C$,  and $\sum_j |\lambda_j|\leq C \| u\|_{\mathcal N}$.

\begin{proposition}\label{Poiss<Poiss} Let $u(t,\mathbf{x})=P_t f(\mathbf{x})$, $v(t,\mathbf{x})=t^n \frac{d^n}{dt^n} P_tf(\mathbf{x})$. Then for $f\in L^1({dw})$ we have
$$\|v\|_{\mathcal N} \leq C_n \| u\|_{\mathcal N},  $$
where $P_t  =e^{-t\sqrt{-{\Delta}}}$ is the Poisson semigroup.
 \end{proposition}
\begin{proof}
Assume that $\| u\|_{\mathcal N}<\infty$. Clearly,
$v(t,\mathbf{x})=2^nQ_{t\slash 2} P_{t\slash 2}f(\mathbf{x})$, where $Q_t=t^n\frac{d^n}{dt^n} P_t$. Set $u^{\{1\}}(t,\mathbf{x})=u(\frac{t}{2}, \mathbf{x})$. Then
$$\|u^{\{1\}}\|_{\mathcal N} \leq C \| u\|_{\mathcal N}.$$
By the atomic decomposition we write
$ u^{\{1\}}=\sum_j c_j A_j$, where $A_j$ are atoms for $\mathcal N$, $c_j\in\mathbb C$, and $\sum |c_j|\lesssim \| u\|_{\mathcal N}$, (see Definition \ref{N-atom}). Thus, by Lemma \ref{Nab}, we have
$$ v(t,\mathbf{x})=2^n \sum_j c_j Q_{t\slash 2} A_j (t,\mathbf{x}),$$
$$  Q_{t\slash 2} A_j (t,\mathbf{x})=\int Q_{t\slash 2}(\mathbf{x},\mathbf{y})A_j(t, \mathbf{y}){dw}(\mathbf{y}).$$
From Proposition \ref{Poiss_new} and the definition of $\mathcal N$ atoms we conclude that  $\| Q_{t\slash 2} A_j (t,\mathbf{x})\|_{\mathcal N}\leq C$.
\end{proof}

{\bf Calder\'on reproducing formula.} {Fix a positive integer $m$ sufficiently large.  Let $\tilde \Theta\in C^m(\mathbb{R})$ be an even function such that
$\| \tilde \Theta\|_{\mathcal{S}^m}<\infty$ (see \eqref{seminorm}). Set $\Theta(\mathbf{x})=\tilde\Theta (\|\mathbf{x}\|)$. Assume that $\int_{\mathbb{R}^N}\Theta(\mathbf{x})\, {dw}(\mathbf{x})=0$. The Plancherel theorem for the Dunkl transform implies
\begin{equation}\label{L2toL2}
\| \Theta_t* f(\mathbf{x})\|_{L^2(\mathbb{R}_+^{{1+N}},\, {dw}(\mathbf{x})\frac{dt}{t})}\leq C \| f\|_{L^2({dw})}.
\end{equation}
By duality,
\begin{equation}\label{L2backL2}
\| \pi_{\Theta}F(\mathbf{x})\|_{L^2({dw}(\mathbf{x}))}\leq C \| F(t,\mathbf{x})\|_{L^2(\mathbb{R}_+^{{1+N}},\, {dw}(\mathbf{x})\frac{dt}{t})},
\end{equation}
where
$$\pi_\Theta F(\mathbf{x})=\int_0^\infty (\Theta_t * F(t,\cdot ))(\mathbf{x})\frac{dt}{t}=\int_0^\infty\int_{\mathbb{R}^N} \Theta_t(\mathbf{x},\mathbf{y})F(t,\mathbf{y})\, {dw}(\mathbf{y})\frac{dt}{t}.$$ }

Let $\Phi (\mathbf{x})\geq 0$ be a radial $C^\infty$ real-valued function on $\mathbb{R}^N$ supported by $B(0,1\slash 4)$, $\Phi (\mathbf{x}) \equiv 1$ on $B(0,1\slash 8)$. Let $\kappa$ be a positive integer, $\kappa >\mathbf{N}\slash 2$.  Set
$$\Psi(\mathbf{x})={\Delta}^{2\kappa} (\Phi * \Phi)(\mathbf{x})=({\Delta}^{\kappa} \Phi )* ({\Delta}^\kappa \Phi) (\mathbf{x}).$$
Then $\Psi$ is radial and  real-valued,
$$\text{supp}\, \Psi \subset B(0,1\slash 2),$$
$$\int \Psi (\mathbf{x}){dw}(\mathbf{x})=0,$$
$$\mathcal{F}\Psi (\xi)=c_k\| \xi\|^{4\kappa}  (\mathcal{F}\Phi)^{2} (\xi)=c_k \| \xi\|^{4\kappa}  |\mathcal{F}\Phi(\xi)|^2. $$
Clearly,
 \begin{equation}\label{supprt_Psi}
 \Phi_t(\mathbf{x},\mathbf{y})=\Psi_t(\mathbf{x},\mathbf{y})=0 \ \ \ \text{\rm if }  d(\mathbf{x},\mathbf{y})>t\slash 2
 \end{equation} and
$$\int \Psi_t(\mathbf{x}, \mathbf{y})\, {dw}(\mathbf{y})=\int \Psi_t(\mathbf{x},\mathbf{y})\, {dw}(\mathbf{x})=0.$$
Moreover, for  $n=0,1,2,{\dots}$, and $f\in L^2({dw})$ we have the Calder\'on reproducing formulae:
$$ f=  c'_n \int_0^\infty \Psi_t t^n (\sqrt{-{\Delta}})^ne^{-t\sqrt{-{\Delta}}} f\frac{dt}{t}=c'\int_0^\infty t^2\Psi_t{\Delta}e^{t^2{\Delta}} f\frac{dt}{t} $$
and the integrals  converge in the $L^2({dw})$-norm.

 Fix a positive integer $m$ (large enough). Let $\Phi^{\{j\}}(\mathbf{x})=\tilde\Phi^{\{j\}}(\|\mathbf{x}\|)$, $j=1,2$, where $\tilde\Phi^{\{j\}}$ are even $C^m$-functions such that $\| \tilde\Phi^{\{j\}}\|_{\mathcal{S}^m}<\infty$  and
 \begin{equation}\label{integral_one} \int_{\mathbb{R}^N} \Phi^{\{j\}}(\mathbf{x})\, {dw}(\mathbf{x})=1, \ \ j=1,2.
 \end{equation}
  Taking instead of $\Phi^{\{j\}}$ their  dilations $\Phi^{\{j\}}_s(\mathbf{x})=s^{-\mathbf{N}} \Phi^{\{j\}}(\mathbf{x}\slash s)$ if necessary, we may assume that
 \begin{equation}
 \label{Calderon}f= c_j'' \int_0^\infty  \Psi_t\Phi^{\{j\}}_tf\frac{dt}{t} ,\ \ f\in L^2({dw}), \ j=1,2,
 \end{equation}
 where the integrals converge in the $L^2$-norm. Moreover, by Lemma \ref{Nab}, there is a constant $C_s>0$ such that if $u^{\{j\}}(t,\mathbf{x})=\Phi^{\{j\}}_tf(\mathbf{x})$ and $v^{\{j\}}(t,\mathbf{x})=\Phi^{\{j\}}_{ts}f(\mathbf{x})=u(st,\mathbf{x})$, then
 $$ C_s^{-1}\|v^{\{j\}}\|_{\mathcal N} \leq  \|u^{\{j\}}\|_{\mathcal N} \leq  C_s\|v^{\{j\}}\|_{\mathcal N}.$$

We are in a position to state the main results of this section.

\begin{proposition}\label{max=max} For $\Phi^{\{1\}}$ and $\Phi^{\{2\}}$ as above and every $f\in L^2({dw})$ we have
\begin{align*}
\| \Phi^{\{1\}}_tf\|_{\mathcal N_\alpha}=\|M_{\Phi^{\{1\}},\alpha }f\|_{L^1({dw})} &\leq C_{\Phi^{\{1\}},\Phi^{\{2\}}, \alpha, \alpha'}\|M_{\Phi^{\{2\}},\alpha'} f\|_{L^1({dw})}
\\&= C_{\Phi^{\{1\}},\Phi^{\{2\}}, \alpha, \alpha'}  \| \Phi^{\{2\}}_tf\|_{\mathcal N_{\alpha'}}.
\end{align*}
\end{proposition}
\begin{proof}
 Let $\Psi^{\{1\}}=\Phi^{\{1\}}-\Phi^{\{2\}}$. Then $\Psi^{\{1\}}$ is radial and thanks to \eqref{integral_one}, we have  $\mathcal{F}\Psi^{\{1\}}(\xi)=O(\|\xi\|^2)$ for $\|\xi\|<1$.  It suffices to prove that
$$\| \Psi^{\{1\}}_tf\|_{\mathcal N}\leq C\| \Phi^{\{2\}}_tf\|_{\mathcal N}.$$
Using the Calder\'on reproducing formula \eqref{Calderon}, we obtain
$$ \Psi^{\{1\}}_tf=c_2'\int_0^\infty \Psi^{\{1\}}_t\Psi_s\Phi^{\{2\}}_sf\frac{ds}{s}$$
According to Proposition \ref{Psi_st}, for any $\eta, \ell >0$ the integral kernel $K_{t,s}(\mathbf{y},\mathbf z)$ of the operator $\Psi^{\{1\}}_t\Psi_s$ satisfies
 $$|K_{t,s}(\mathbf{y},\mathbf z)| \leq C_{\eta, \ell} \min \Big(\big(\frac{t}{s}\big)^{2}, \big(\frac{s}{t}\big)^{\ell}\Big) \frac{1}{V(\mathbf{y},\mathbf z, s+t )}\Big(1+\frac{d(\mathbf{y},\mathbf z)}{s+t}\Big)^{-\mathbf{N}-\eta}. $$
We take $\mathbf{N}<\lambda<\eta<\ell$. Then for $\|\mathbf{x}-\mathbf{y}\|<t$ we have
\begin{equation}\label{ts}
\int | K_{t,s}(\mathbf{y},\mathbf z)|\Big(1+\frac{d(\mathbf{x},\mathbf z)}{s}\Big)^\lambda \, {dw}(\mathbf z)
\leq C'\min \Big(\big(\frac{s}{t}\big)^{\ell-\lambda} , \big(\frac{t}{s}\big)^2\Big).
\end{equation}
Therefore, using \eqref{ts} we obtain,
\begin{equation}\begin{split}
\sup_{\|\mathbf{x}-\mathbf{y}\|<t}|\Psi^{\{1\}}_tf(\mathbf{y})|
&=c_2'\sup_{\|\mathbf{x}-\mathbf{y}\|<t} \left|\int_0^\infty \int K_{t,s} (\mathbf{y},\mathbf z)\Phi^{\{2\}}_sf(\mathbf z)\, {dw}(\mathbf z)\frac{ds}{s}
\right|\\
&\leq c_2' \sup_{\mathbf z,s} |\Phi^{\{2\}}_sf(\mathbf z)|\Big(1+\frac{d(\mathbf{x},\mathbf z)}{s}\Big)^{-\lambda}\\
&\times
\sup_{\|\mathbf{x}-\mathbf{y}\|<t}\int_0^\infty \int | K_{t,s}(\mathbf{y},\mathbf z)|\Big(1+\frac{d(\mathbf{x},\mathbf z)}{s}\Big)^\lambda \, {dw}(\mathbf z)\frac{ds}{s}\\
&\leq C \sup_{\mathbf z,s} |\Phi^{\{2\}}_sf(\mathbf z)|\Big(1+\frac{d(\mathbf{x},\mathbf z)}{s}\Big)^{-\lambda}.
\end{split}\end{equation}
The proof is complete, by applying~\eqref{maxlamba}.
\end{proof}

\begin{remark}\label{r33}
  \normalfont
  It follows from the proof of Proposition \ref{max=max} that if $\Theta\in\mathcal{S}(\mathbb{R}^N)$ is radial and $\int_{\mathbb{R}^N}\Theta(\mathbf{x})\, {dw}(\mathbf{x})=0$, and $\Phi^{\{2\}}$ is as above, then for $f\in L^2({dw})$ we have
  $$\| \Theta_tf\|_{\mathcal N} \leq C \| \Phi^{\{2\}}_tf\|_{\mathcal N}$$
\end{remark}
\begin{proposition}\label{PropMax} For a function $\Phi^{\{1\}}$ as described above and $\alpha, \alpha'>0$  there is a constant $C_{\Phi^{\{1\}},\alpha, \alpha'}>0$ such that
$$\|M_{\Phi^{\{1\}},\alpha }f\|_{L^1({dw})} \leq C_{\Phi^{\{1\}}, \alpha, \alpha'}\|M_{P,\alpha'} f\|_{L^1({dw})}, \ \ \text{for } f\in L^1({dw})\cap L^2({dw}), $$
where ${P_t=e^{-t\sqrt{\Delta}}}$ is the Poisson semigroup.
 \end{proposition}
\begin{proof}
For a positive integer $n$ (large) set  $\phi (\xi)=e^{-\|\xi\|}\Big(\sum_{j=0}^{n+1} \frac{\|\xi\|^j}{j!}\Big)$. Then
$$\phi (\xi) -1=O(\|\xi\|^{n+1}) \ \ \text{for } \ \|\xi\|<1.$$
 So $\phi$ is a $C^n(\mathbb{R}^N)$ function such that $|\partial^\alpha \phi (\xi)|\leq C_\alpha \exp({-\|\xi\|\slash 2})$, $|\alpha |\leq n$. Put $\Phi^{\{2\}} =c_k^{-1}\mathcal{F}^{-1}\phi$.
 Applying Proposition \ref{max=max}, we have
 $$ \| \Phi_t^{\{1\}} f\|_{\mathcal N}\lesssim \| \Phi^{\{2\}}_tf\|_{\mathcal N}.$$
 Notice that $\frac{d^j}{dt^j}P_tf(\mathbf{x})=\mathcal{F}^{-1}(\|t\xi\|^je^{-\hspace{.25mm}t\hspace{.25mm}{\|}\xi{\|}}\mathcal{F}f(\xi))(\mathbf{x})$. Hence, from Proposition \ref{Poiss<Poiss} we conclude,
 $$\| \Phi^{\{2\}}_t f\|_{\mathcal N} \leq C \sum_{j=0}^{n+1}\Big\| t^j\frac{d^j}{dt^j}P_tf\Big\|_{\mathcal N}\leq C' \| P_tf\|_{\mathcal N}. $$
\end{proof}
\noindent

\section{Atomic decompositions; inclusions $H^1_{{\rm max},P}\subset H^1_{(1,\infty,M)}$ and  $H^1_{{\rm max},{H}}\subset H^1_{(1,\infty,M)}$}\label{Atomic}

Note that Proposition \ref{PropMax} implies that $H^1_{{\rm max},P}\cap L^2({dw}) \subset H^1_{{\rm max},{H}}$ and
\begin{equation}\label{hP}
\|\mathcal{M}_{H} f\|_{L^1({dw})} \leq C\|\mathcal{M}_P f\|_{L^1({dw})} \ \ \text{for} \ f\in L^2({dw}).
\end{equation}
\begin{lemma}\label{PoissonHeat} $H^1_{{\rm max},{H}}\subset H^1_{{\rm max},P}$ and there is a constant $C>0$ such that
\begin{equation}\label{Ph}
\|\mathcal{M}_P f\|_{L^1({dw})} \leq C\|\mathcal{M}_{H} f\|_{L^1({dw})} \ \ \text{for} \ f\in L^1({dw}).
\end{equation}
\end{lemma}
\begin{proof}The proof is standard. Let $f\in L^1({dw})$. Set $u(t,\mathbf{x})=e^{t^2{\Delta}}f(\mathbf{x})$.  By the subordination formula \eqref{subordination} for fixed $t>0$ we have
\begin{equation*}\begin{split}
\sup_{\|\mathbf{x}'-\mathbf{x}\|<t}|P_tf(\mathbf{x}')|
&\leq \frac{1}{2\sqrt{\pi}} \int_0^\infty \sup_{\|\mathbf{x}'-\mathbf{x}\|<t}|u(ts,\mathbf{x}')|e^{-\frac{1}{4s^2}}\frac{ds}{s^2}\\
&=\frac{1}{2\sqrt{\pi}} \int_0^\infty \sup_{\|\mathbf{x}'-\mathbf{x}\|<t}|u(ts,\mathbf{x}')|\Big(\frac{ts}{\|\mathbf{x}-\mathbf{x}'\|+ts}\Big)^\lambda
\Big(\frac{\|\mathbf{x}-\mathbf{x}'\|+ts}{ts}\Big)^\lambda
e^{-\frac{1}{4s^2}}\frac{ds}{s^2}\\
&\leq \frac{1}{2\sqrt{\pi}} \int_0^\infty u^{**}_\lambda (\mathbf{x})
\Big(\frac{1+s}{s}\Big)^\lambda
e^{-\frac{1}{4s^2}}\frac{ds}{s^2}\\
&\leq Cu^{**}_\lambda(\mathbf{x})
\end{split}
\end{equation*}
Now the lemma follows from \eqref{maxlamba}.
\end{proof}
Let us note that \eqref{hP} and \eqref{Ph} imply that $H^1_{{\rm max},{H}}\cap L^2({dw}) = H^1_{{\rm max},P}\cap L^2({dw})$ and
$$ \|\mathcal{M}_{P} f\|_{L^1({dw})}\sim \|\mathcal{M}_{H} f\|_{L^1({dw})} \ \ \text{for } f\in L^2({dw}).$$

In the next theorem we show that all elements in  $H^1_{{\rm max},{H}}\cap L^2({dw})=H^1_{{\rm max},P}\cap L^2({dw})$ admit atomic decompositions into $(1,\infty, M)$-atoms. The $L^2(dw)$ condition is removed afterwards in Theorem \ref{AtomPoisson}.

\begin{theorem}\label{AtomicL2}
 For every positive integer $M$ there is a constant $C_M>0$ such that every element
 $f\in H^1_{{\rm max},{H}}\cap L^2({dw})=H^1_{{\rm max},P}\cap L^2({dw})$  can be written  as
$$ f=\sum \lambda_j a_j$$
where  $a_j$ are $(1,\infty, M)$-atoms, $\sum |\lambda_j|\leq C_{M} \|\mathcal{M}_P f\|_{L^1({dw})}$. Moreover,  the convergence is in $L^2({dw})$.
\end{theorem}

\begin{proof}
The proof is a straightforward adaptation of \cite{SY} with the difference that tents are now constructed with respect to the orbit distance \hspace{.25mm}$d(\mathbf{x},\mathbf{y})$\hspace{.25mm}. We include details for the convenience of readers unfamiliar with \cite{SY}. More experienced readers may skip the proof and jump to Lemma \ref{L2Lemma}. Without loss of generality, we may assume that $M$ is an even integer $>\hspace{-.25mm}2\hspace{.5mm}\mathbf{N}$\hspace{.25mm}.
\smallskip

{\bf Step 1. Reproducing formulae.}\label{step1}
Let $\Phi$, $\Psi$ be as in the Calder\'on reproducing formula with $\kappa=M\slash 2$ (see Section \ref{Max=Max}). Denote $$\varphi(\xi)=\mathcal{F}(\Phi)(\xi)=\tilde \varphi(\| \xi\|),$$
$$\psi(\xi)=\mathcal{F}(\Psi)(\xi)=c_k\| \xi\|^{2M}|\varphi(\xi)|^2 = \tilde \psi(\| \xi\|)=c_k\| \xi\|^{2M}|\tilde \varphi(\| \xi\|)|^2.$$
Then there is a constant $c$ such that
$$f=\lim_{\varepsilon \to 0}c\int_\varepsilon^{\varepsilon^{-1}} \Psi_tt^2{\Delta}e^{t^2{\Delta}}f\frac{dt}{t}$$
with convergence in $L^2({dw})$. We have
$$\mathcal{F}f(\xi)= \lim_{\varepsilon \to 0}c_k c\int_\varepsilon^{\varepsilon^{-1}} t^2\| \xi\|^2 \tilde \psi (t\| \xi\|)e^{-t^2\| \xi\|^2}\mathcal{F}f(\xi)\frac{dt}{t}.$$
For $\xi\ne0$ set
$$ \eta(\xi)=c_k c\int_1^\infty t^2 \|\xi\| ^2\tilde \psi(t\|\xi\|)e^{-t^2\|\xi\|^2}\frac{dt}{t}=c_kc\int_{\|\xi\|}^\infty t^2\tilde \psi(t)e^{-t^2}\frac{dt}{t}. $$
Put $\eta(0)=1$. Then $\eta$ is a Schwartz class radial real-valued function.
Set $\Xi(\mathbf{x})=c_k^{-1}\mathcal{F}^{-1} \eta(\mathbf{x})$.
Then $\Xi\in\mathcal{S}(\mathbb{R}^N)$,  $\int \Xi(\mathbf{x})\, {dw}(\mathbf{x})=1$, and
\begin{equation}\label{differ}
c\int_a^b \Psi_tt^2{\Delta}e^{t^2{\Delta}}f\frac{dt}{t}=\Xi_af - \Xi_bf.
\end{equation}

{\bf Step 2. Space of orbits.}\label{step2} Let $X=\mathbb{R}^N\slash G$ be the space of orbits equipped with the metric $d(\mathcal{O}(\mathbf{x}),\mathcal{O}(\mathbf{y}))=d(\mathbf{x},\mathbf{y})$ and the measure $\boldsymbol m(A)={w}\Big(\bigcup_{\mathcal{O}(\mathbf{x})\in A}\mathcal{O}(\mathbf{x}) \Big)$.
So $(X,d,\boldsymbol m)$ is the space of homogeneous type in the sense of Coifman--Weiss. The space $X$ can be identified with a positive Weyl chamber. Any open set in $X$ of  finite measure admits the following easily proved  Whitney type covering lemma.

\begin{lemma}
\label{Whitney}
Suppose that ${\boldsymbol\Omega}\subset X$ is an open set with finite measure. Then there is a sequence of balls $B_X(\mathcal{O}(\mathbf{x}_{\{n\}}), r_{\{n\}})$ such that $r_{\{n\}}=d(\mathcal{O}(\mathbf{x}_{\{n\}}),{\boldsymbol\Omega}^c)$,
$$\bigcup_{n \in\mathbb{N}} B_X(\mathcal{O} (\mathbf{x}_{\{n\}}),r_{\{n\}}\slash 2)={\boldsymbol
\Omega},$$
 the balls $B_X(\mathcal{O}(\mathbf{x}_{\{n\}}),r_{\{n\}}\slash 10)$ are disjoint.
\end{lemma}

{\bf Step 3. Decomposition of $\mathbb{R}^{N+1}_+$.}\label{step3} Assume that $f\in H^1_{{\rm max},{H}} \cap L^2({dw})$. Let
$$F(t,\mathbf{x})=\Big(|t^2{\Delta} e^{t^2{\Delta}}f(\mathbf{x})|+|\Xi_tf(\mathbf{x})|\Big),$$
$$ \boldsymbol F(t,\mathbf{x})=\sup_{\sigma\in G} F(t,\sigma (\mathbf{x})),$$
 and
$$\mathcal{M}f(\mathbf{x})=\sup_{d(\mathbf{x},\mathbf{y})<5t} F(t,\mathbf{y})=\sup_{\| \mathbf{x}-\mathbf{y}\|<5t} \boldsymbol F(t,\mathbf{y}).$$
Then, by Proposition \ref{max=max} and Remark \ref{r33}, we have $\|\mathcal{M}f\|_{L^1({dw})}\leq C\| f\|_{H^1_{{\rm max},{H}}}$.
Observe that $\mathcal{M}f(\sigma (\mathbf{x}))=\mathcal{M}f( \mathbf{x})$.
Therefore  $\mathcal{M}f( \mathbf{x})$ can be identified with the function
$\boldsymbol {\mathcal{M}}f(\mathcal{O}(\mathbf  x))$
on $X$, moreover
$\|\mathcal{M}f( \mathbf{x})\|_{L^1({dw})}=\| \boldsymbol {\mathcal{M}}f(\mathcal{O}( \mathbf{x}))\|_{L^1(\boldsymbol m)}$.
For an open set $\boldsymbol{\Omega}\subset X$ let
$$ \hat{\boldsymbol{\Omega}}=\{ (t,\mathcal{O}(\mathbf{x})): B_X(\mathcal{O}(\mathbf{x}), 4t)\subset \boldsymbol{\Omega}\}$$
be the tent over $\boldsymbol{\Omega}$.  For $j \in\mathbb{Z}$ define
$$\boldsymbol{\Omega}_j=\{\mathcal{O}({\mathbf{x}})\in X: \boldsymbol {\mathcal{M}}f(\mathcal{O}(\mathbf  x))>2^j\}, \ \ \Omega_j=\{ \mathbf{x}\in\mathbb{R}^N:\mathcal{M}f(\mathbf{x})>2^j\}.$$

 Then $\boldsymbol{\Omega}_j$ is open in $X$,
$\Omega_j=\bigcup_{\mathcal{O}(\mathbf{x})\in \boldsymbol{\Omega_j}}{\mathcal{O}}(\mathbf{x})$, $\boldsymbol m(\boldsymbol {\Omega_j})={w}(\Omega_j)$,
$$\sum_{j} 2^j{w}(\Omega_j)\sim \|\mathcal{M}f\|_{L^1({dw})} \sim \| f\|_{H^1_{{\rm max},{H}}}.$$
 Clearly, $\hat \Omega_j = \{ (t,\mathbf{x})\in\mathbb{R}^{N+1}_+: (t,\mathcal{O}(\mathbf{x}))\in\hat{\boldsymbol{\Omega}}_j\}$. Set $\mathbf T_j=\hat\Omega_j\setminus \hat\Omega_{j+1}$. Then,
\begin{equation}
\begin{split}
{\rm supp}\, F(t,\mathbf{x}) &\subset \bigcup_{j\in\mathbb Z} \hat \Omega_j
=\bigcup_{j\in\mathbb Z} (\hat \Omega_j\setminus \hat\Omega_{j+1})= \bigcup_{j\in\mathbb Z} \mathbf  T_j\\
\end{split}\end{equation}
Let $B_X(\mathcal{O}(\mathbf{x}_{\{n,\,j\}}), r_{\{n,\,j\}}\slash 2))$, $\mathbf{x}_{\{n,\,j\}}\in\mathbb{R}^N$, $n=1,2,{\dots},$  be a Whitney covering of $\boldsymbol \Omega_j$. Set
$$Q_{\{n,\,j\}}=\{\mathbf{x}\in\mathbb{R}^N:\mathcal{O}(\mathbf{x})\in B_X(\mathcal{O}(\mathbf{x}_{\{n,\,j\}}), r_{\{n,\,j\}}\slash 2))\}=\mathcal{O}(B(\mathbf{x}_{\{n,\,j\}}, r_{\{n,\,j\}}\slash 2)).$$

Obviously, ${w}(B(\mathbf{x}_{\{n,\,j\}}, r_{\{n,\,j\}}\slash 2))\leq {w}(Q_{\{n,\,j\}})\leq |G|{w}(B(\mathbf{x}_{\{n,\,j\}}, r_{\{n,\,j\}}\slash 2))$.   We define a cone over a $G$-invariant set $E$ as
$$\mathcal R(E)=\{(t,\mathbf{y}): d(\mathbf{y}, E)<2t\}.$$
For $n=1,2,{\dots}$, let
$$ \mathbf T_{\{n,\,j\}}=\mathbf T_j\cap \Big(\mathcal R(Q_{\{n,\,j\}})
\setminus \bigcup_{i=0}^{n-1}\mathcal R(Q_{\{i,\,j\}})\Big), \ \ \mathcal R(Q_{\{0,\,j\}})=\emptyset.$$
Clearly, $\hat\Omega_j\subset \bigcup_{n\in\mathbb N}\mathcal R(Q_{\{n,\,j\}})$, $\mathbf T_{\{n,\,j\}}\cap \mathbf T_{\{n',\,j'\}}=\emptyset$ if $(j,n)\ne (j',n')$. Thus we have \begin{equation}\begin{split}
{\rm supp}\, F(t,\mathbf{x}) &\subset \bigcup_{j\in\mathbb Z} \bigcup_{n\in\mathbb N} \mathbf T_{\{n,\,j\}}.
\end{split}\end{equation}

{\bf Step 4. Decomposition of $f$ and $L^2({dw})$-convergence.}\label{step4} Write
\begin{equation}\begin{split}
f&= \sum_{j\in\mathbb Z, \, n\in\mathbb N}
c\int_0^\infty\Psi_t\Big(\chi_{\mathbf T_{\{n,\,j\}}}t^2{\Delta} e^{t^2{\Delta}}f\Big)\frac{dt}{t}=\sum_{j\in\mathbb Z, \, n\in\mathbb N} \lambda_{\{n,\,j\}}a_{\{n,\,j\}},
\end{split}\end{equation}
where $\lambda_{\{n,\,j\}}=2^j{w}(Q_{\{n,\,j\}})$,
 \begin{equation*}\begin{split}
 a_{\{n,\,j\}}
 & =(\lambda_{\{n,\,j\}})^{-1} c\int_0^\infty\Psi_t\Big(\chi_{\mathbf T_{\{n,\,j\}}}t^2(-{\Delta}) e^{t^2{\Delta}}f\Big)\frac{dt}{t}\\
 &=
(\lambda_{\{n,\,j\}})^{-1} c\int_0^\infty t^{2M}(-{\Delta})^M\Phi_t\Phi_t\Big(\chi_{\mathbf T_{\{n,\,j\}}}t^2(-{\Delta})e^{t^2{\Delta}}f\Big)\frac{dt}{t}
\end{split}\end{equation*}
and, thanks to \eqref{L2backL2}, the convergence is in $L^2({dw})$, because $\mathbf T_{\{n,\,j\}}$ are pairwise disjoint.

{\bf Step 5. What remains to prove.}\label{step5} Our task is to prove that the functions $a_{\{n,\,j\}}$ are proportional to $(1,\infty, M)$-atoms. If this is done then
$$ \sum_{j\in\mathbb{Z},\,n\in\mathbb{N}} |\lambda_{\{n,\,j\}}| = \sum_{j\in\mathbb{Z},\,n\in\mathbb{N}} 2^j {w}(Q_{\{n,\,j\}})\lesssim \sum_{j\in\mathbb Z} 2^j{w}(\Omega_j)\sim \| f\|_{H^1_{{\rm max},\, H}},$$
which proves the atomic decomposition.

{\bf Step 6. Functions $b_{\{n,\,j\}}$. Support of ${\Delta}^mb_{\{n,\,j\}}$ for $m=0,1,\ldots,M$.}\label{step6}  Observe that
\begin{equation}\label{t_support}\begin{split} a_{\{n,\,j\}}&=(\lambda_{\{n,\,j\}})^{-1} c\int_0^{r_{\{n,\,j\}}}\Psi_t\Big(\chi_{\mathbf T_{\{n,\,j\}}}t^2(-{\Delta})e^{t^2{\Delta}}f\Big)\frac{dt}{t}\\
&=(\lambda_{\{n,\,j\}})^{-1} c\int_0^{r_{\{n,\,j\}}} t^{2M}(-{\Delta})^M\Phi_t\Phi_t\Big(\chi_{\mathbf T_{\{n,\,j\}}}t^2(-{\Delta})e^{t^2{\Delta}}f\Big)\frac{dt}{t}.
\end{split}\end{equation}
Indeed, if $t>r_{\{n,\,j\}}$ and $(t,\mathbf{y})\in\mathcal R(Q_{\{n,\,j\}})$ then
\begin{equation}\label{distances}
d(\mathbf{y}, (\Omega_j)^c)\leq d(\mathbf{y}, Q_{\{n,\,j\}})+\frac{1}{2} r_{\{n,\,j\}}
+d(\mathbf{x}_{\{n,\,j\}},(\Omega_j)^c)\leq 2t+\frac{1}{2}t+t=\frac{7}{2}t.
\end{equation}
 Hence $(t,\mathbf{y})\notin \mathbf T_{\{n,\,j\}}$.

As a consequence of \eqref{supprt_Psi}, \eqref{t_support}, and \eqref{distances}, we have
 \begin{equation}\label{support_a}
 {\rm supp}\, a_{\{n,\,j\}}\subset \Big\{\mathbf{x}\in\mathbb{R}^N: d(\mathbf{x},\mathbf{x}_{\{n,\,j\}})\leq \frac{7}{2}r_{\{n,\,j\}}\Big\}=\mathcal O\Big(B\Big(\mathbf x_{\{n,\,j\}},\frac{7}{2}r_{\{n,\,j\}}\Big)\Big).
 \end{equation}
Let
$$ b_{\{n,\,j\}}=(\lambda_{\{n,\,j\}})^{-1} c\int_0^{r_{\{n,\,j\}}} t^{2M} \Phi_t\Phi_t\Big(\chi_{\mathbf T_{\{n,\,j\}}}t^2(-{\Delta}) e^{t^2{\Delta}}f\Big)\frac{dt}{t}. $$
Then $b_{\{n,\,j\}}\in\mathcal{D}({\Delta}^M)$,
$$ (-{\Delta})^mb_{\{n,\,j\}}=(\lambda_{\{n,\,j\}} )^{-1} c\int_0^{r_{\{n,\,j\}}} t^{2M}(-{\Delta})^m\Phi_t\Phi_t\Big(\chi_{\mathbf T_{\{n,\,j\}}}t^2(-{\Delta})e^{t^2{\Delta}}f\Big)\frac{dt}{t}$$
for  $m=1,2,{\dots}M$, and, by the same arguments,
\begin{equation}\label{support_atoms1} \text{\rm supp}\, {\Delta}^mb_{\{n,\,j\}}\subset
\mathcal O\Big(B\Big(\mathbf x_{\{n,\,j\}},\frac{7}{2}r_{\{n,\,j\}}\Big)\Big).
\end{equation}
Note also that ${\Delta}^mb_{\{n,\,j\}}(\mathbf{x})\ne 0$ implies that there is $(t,\mathbf{y})\in \hat {\Omega}_j$ such that $d(\mathbf{x},\mathbf{y})\leq t$.
Then
$\mathcal{O}(\mathbf{x})\in B_X(\mathcal{O}(\mathbf{y}), t)\subset B_X(\mathcal{O}(\mathbf{y}), 4t)\subset {\boldsymbol{\Omega}}_j$. Hence,
\begin{equation}\label{support_atoms2}
\text{\rm supp}\, {\Delta}^mb_{\{n,\,j\}}\subset \Omega_j.
\end{equation}

{\bf Step 7. Size of ${\Delta}^mb_{\{n,\,j\}}$ for $m=0,1,{\dots},M-1$.}\label{step7}
Suppose that $(t,\mathbf{y})$ is such that $\chi_{\mathbf T_{\{n,\,j\}}}(t,\mathbf{y})=1$. Then $(t,\mathbf{y})\in (\hat{\Omega}_{j+1})^c$, so $|t^2{\Delta} e^{t^2{\Delta}}f(\mathbf{y})|\leq 2^{j+1}$. Consequently,
\begin{equation*}
\begin{split}
&|{\Delta}^mb_{\{n,\,j\}}(\mathbf{x})|
=\frac{c}{\lambda_{\{n,\,j\}}}\left| \int_0^{r_{\{n,\,j\}}} t^{2M-2m} (t^2(-{\Delta}))^m \Phi_t\Phi_t(\chi_{\mathbf T_{\{n,\,j\}}}t^2(-{\Delta})e^{t^2{\Delta}}f)(\mathbf{x})\frac{dt}{t}\right|\\
&= (\lambda_{\{n,\,j\}})^{-1} c\left| \int_0^{r_{\{n,\,j\}}} \int t^{2M-2m}  K_t^m(\mathbf{x},\mathbf{y})(\chi_{\mathbf T_{\{n,\,j\}}}(t,\mathbf{y})t^2(-{\Delta})e^{t^2{\Delta}}f(\mathbf{y})){dw}(\mathbf{y})\frac{dt}{t}\right|,\\
\end{split}
\end{equation*}
where $K_t^m(\mathbf{x},\mathbf{y})$ is the integral kernel of the operator $(-t^2{\Delta})^m \Phi_t\Phi_t$. Recall that  $$|K_t^m(\mathbf{x},\mathbf{y})|\leq C{w}(B(\mathbf{x},t))^{-1}$$
and
$$K_t^m(\mathbf{x},\mathbf{y})=0 \ \ \text{\rm for } d(\mathbf{x},\mathbf{y})>t\slash 2$$
 (see \eqref{supprt_Psi} and Corollary \ref{translation2}).
Thus,
\begin{equation}\label{atombnj}\begin{split}
|{\Delta}^mb_{\{n,\,j\}}(\mathbf{x})|
&\leq C(\lambda_{\{n,\,j\}})^{-1} 2^{j+1}   \int_0^{r_{\{n,\,j\}}} \int t^{2M-2m}  |K_t^m(\mathbf{x},\mathbf{y})|{dw}(\mathbf{y})\frac{dt}{t}\\
&\leq C(\lambda_{\{n,\,j\}})^{-1} 2^{j+1} \int_0^{r_{\{n,\,j\}}} t^{2M-2m}\frac{dt}{t}\\
&= C(\lambda_{\{n,\,j\}})^{-1} 2^{j} (r_{\{n,\,j\}})^{2M-2m}\\
&=C {w}(Q_{\{n,\,j\}})^{-1} (r_{\{n,\,j\}})^{2M-2m}.
\end{split}\end{equation}

{\bf Step 8.  Key lemma.}\label{step8} It remains to estimate
$$ a_{\{n,\,j\}}(\mathbf{x})=(\lambda_{\{n,\,j\}})^{-1} c\int_0^{\infty} \int \Psi_t(\mathbf{x},\mathbf{y}) \chi_{\mathbf T_{\{n,\,j\}}}(t,\mathbf{y})(t^2(-{\Delta})e^{t^2{\Delta}}f)(\mathbf{y})\, {dw}(\mathbf{y})\frac{dt}{t}.$$
Let $E_{\{n,\,j\}}=\bigcup_{i=1}^n Q_{\{i,\,j\}}$. Then
 \begin{equation}\begin{split}\chi_{\mathbf T_{\{n,\,j\}}}(t,\mathbf{y})&=\chi_{\hat{\Omega}_j}(t,\mathbf{y})\chi_{(\hat{\Omega}_{j+1})^c}(t,\mathbf{y})\chi_{\mathcal R(E_{\{n,\,j\}})}(t,\mathbf{y})\chi_{(\mathcal R(E_{\{n-1,\,j\}}))^c}(t,\mathbf{y})\\
&=\chi_1(t,\mathbf{y})\chi_2(t,\mathbf{y})\chi_3(t,\mathbf{y})\chi_4(t,\mathbf{y}).
\end{split}\end{equation}
 The following lemma (see \cite[Lemma 4.2]{SY}) plays a crucial role in the remaining part of the proof of Theorem \ref{AtomicL2}.

\begin{lemma}\label{key_lemma}
For every $\mathbf{x}\in \Omega_j$  and every function $\chi_s$, $s=1,2,3,4$, there are numbers $0<\delta_s\leq \omega_s$ such that $\omega_s\leq 3\delta_s$ and

  \noindent
  either $\Psi_{t}(\mathbf{x},\mathbf{y})\chi_s(t,\mathbf{y})=0$ for every $0<t<\delta_s$ or
  $\Psi_{t}(\mathbf{x},\mathbf{y})\chi_s(t,\mathbf{y})=\Psi_t(\mathbf{x},\mathbf{y})$ for every $0<t<\delta_s$

  and

  \noindent
  either $\Psi_{t}(\mathbf{x},\mathbf{y})\chi_s(t,\mathbf{y})=0$ for every $t>\omega_s$ or
  $\Psi_{t}(\mathbf{x},\mathbf{y})\chi_s(t,\mathbf{y})=\Psi_t(\mathbf{x},\mathbf{y})$ for every $t>\omega_s$.
 \end{lemma}

\begin{proof}
For the reader's convenience, we include a short proof along the lines of \cite{SY}.
Fix $t>0$ and define $\chi_{1}'(\mathbf{y})=\chi_{[4t,\infty)}(d(\mathbf{y},\Omega_j^{c}))$, $\chi_{2}'(\mathbf{y})=\chi_{(-\infty,4t)}(d(\mathbf{y},\Omega_{j+1}^{c}))$, $\chi_{3}'(\mathbf{y})=\chi_{(-\infty,2t)}(d(\mathbf{y},E_{\{n,\,j\}}))$, $\chi_{4}'(\mathbf{y})=\chi_{[2t,\infty)}(d(\mathbf{y},E_{\{n-1,\,j\}}))$. Clearly,  $\chi_s'(\mathbf{y})=\chi_s(t,\mathbf{y})$ for $s=1,2,3,4$. If $d(\mathbf{x},\mathbf{y})\geq t$, then $\Psi_t(\mathbf{x},\mathbf{y})=\Psi_t(\mathbf{x},\mathbf{y})\chi_s (t,\mathbf{y})=0$. Therefore, to finish the proof, we  assume that  $d(\mathbf{x},\mathbf{y})<t$.  Then
$$   -t+d(A,\mathbf{x}) <d(A,\mathbf{y})< t+d(A,\mathbf{x}) \ \  \text{for }  A=\Omega_j^{c},\Omega_{j+1}^{c},E_{\{n,\,j\}},E_{\{n-1,\,j\}}.$$  We are in a position to define  consecutively   $\delta_s$ and $\omega_s$.
\begin{enumerate}[(1)]
\item{If $d(\mathbf{x},\Omega_{j}^c) < 3t$ or $d(\mathbf{x},\Omega_{j}^c) > 5t$, then $\chi'_1(\mathbf{y}) = 0$ and $\chi'_1(\mathbf{y}) = 1$ respectively, so we put $\delta_1=\frac{1}{5}d(\mathbf{x},\Omega_{j}^c)$ and $\omega_1=\frac{1}{3}d(\mathbf{x},\Omega_{j}^c)$.}

\item{If $d(\mathbf{x},\Omega_{j+1}^c) < 3t$ or $d(\mathbf{x},\Omega_{j+1}^c) > 5t$, then $\chi'_2(\mathbf{y}) = 1$ and $\chi'_2(\mathbf{y}) = 0$ respectively. Hence we set $\delta_2=\frac{1}{5}d(\mathbf{x},\Omega_{j+1}^c)$ and $\omega_2=\frac{1}{3}d(\mathbf{x},\Omega_{j+1}^c)$ if $d(\mathbf{x},\Omega_{j+1}^c) \neq 0$, $\delta_2=\omega_2=\delta_1$ otherwise.}

\item{If $d(\mathbf{x},E_{\{n,\,j\}}) < t$ or $d(\mathbf{x},E_{\{n,\,j\}}) > 3t$, then $\chi'_3(\mathbf{y}) = 1$ and $\chi'_3(\mathbf{y}) = 0$ respectively. Thus we put $\delta_3=\frac{1}{3}d(\mathbf{x},E_{\{n,\,j\}})$ and $\omega_3=d(\mathbf{x},E_{\{n,\,j\}})$ if $d(\mathbf{x},E_{\{n,\,j\}}) \neq 0$, $\delta_3=\omega_3=\delta_1$ otherwise.}

\item{If $d(\mathbf{x},E_{\{n-1,\,j\}}) < t$ or $d(\mathbf{x},E_{\{n-1,\,j\}}) > 3t$, then $\chi'_4(\mathbf{y}) = 0$ and $\chi'_4(\mathbf{y}) = 1$ respectively, so we put $\delta_4=\frac{1}{3}d(\mathbf{x},E_{\{n-1,\,j\}})$ and $\omega_4=d(\mathbf{x},E_{\{n-1,\,j\}})$ if $d(\mathbf{x},E_{\{n-1,\,j\}}) \neq 0$, $\delta_4=\omega_4=\delta_1$ otherwise.}
\end{enumerate}\end{proof}

We finish Step 8 by the remark (see Case 1 of the proof of the lemma) that if $t>\omega_1>0$ then
 $$\Psi_t(\mathbf{x},\mathbf{y}) \chi_{\mathbf T_{\{n,\,j\}}}(t,\mathbf{y})=0.$$

 {\bf Step 9. Estimates for $a_{\{n,\,j\}}$.}\label{step9}
 We shall prove that
 \begin{equation}\label{atomanj}
  |a_{\{n,\,j\}}(\mathbf{x})|\leq C{w}(Q_{\{n,\,j\}})^{-1}.
  \end{equation}
 Fix $\mathbf{x}\in\Omega_j$. Recall that $\text{\rm supp}\, a_{\{n,\,j\}}\subset \Omega_j$.  Let $J=\bigcup_{s=1}^4 [\delta_s,\omega_s]$,    $I=(0,\infty)\setminus J$, where $\delta_s,\ \omega_s$ are from Lemma~\ref{key_lemma}. Obviously, $I=(a_1,b_1)\cup{\dots}\cup(a_m,b_m)$, where $m\leq 5$,  $a_1=0$, $b_m=\infty$, and  $(a_l,b_l)$ are connected disjoint  components of $I$.
Clearly,

\begin{equation*}\begin{split}
\Big|a_{\{n,\,j\}}(\mathbf{x})\Big|
&\leq \sum_{s=1}^4(\lambda_{\{n,\,j\}})^{-1} c\int_{\delta_s}^{\omega_s} \int \Big|\Psi_t(\mathbf{x},\mathbf{y}) \chi_{\mathbf T_{\{n,\,j\}}}(t,\mathbf{y})(t^2(-{\Delta})e^{t^2{\Delta}}f)(\mathbf{y})\Big|\, {dw}(\mathbf{y})\frac{dt}{t}\\
&+\sum_{s=1}^m (\lambda_{\{n,\,j\}})^{-1} c\left|\int_{a_s}^{b_s} \int \Psi_t(\mathbf{x},\mathbf{y}) \chi_{\mathbf T_{\{n,\,j\}}}(t,\mathbf{y})(t^2(-{\Delta})e^{t^2\Delta}f)(\mathbf{y})\, {dw}(\mathbf{y})\frac{dt}{t}\right|.\\
\end{split}
\end{equation*}

Consider the integral over $[\delta_s,\omega_s]$. Take
$t\in [\delta_s,\omega_s]$ and $\mathbf{y}$ such that the integrant
$\Big|\Psi_t(\mathbf{x},\mathbf{y}) \chi_{\mathbf T_{\{n,\,j\}}}(t,\mathbf{y})(t^2(-{\Delta})e^{t^2{\Delta}}f)(\mathbf{y})\Big|\ne 0$. Then $(t,\mathbf{y})\notin \hat{\Omega}_{j+1}$.
Thus, there is $\mathbf{x}'$ such that $d(\mathbf{y},\mathbf{x}')<4t$ and $\mathbf{x}'\notin \Omega_{j+1}$, which means that $\mathcal{M}f(\mathbf{x}')\leq 2^{j+1}$. Consequently, $|t^2(-{\Delta})e^{t^2{\Delta}}f(\mathbf{y})|\leq 2^{j+1}$. Hence,
\begin{equation}\label{atom1}\begin{split}
(\lambda_{\{n,\,j\}})^{-1}
& c\int_{\delta_s}^{\omega_s} \int \Big|\Psi_t(\mathbf{x},\mathbf{y}) \chi_{\mathbf T_{\{n,\,j\}}}(t,\mathbf{y})(t^2(-{\Delta})e^{t^2{\Delta}}f)(\mathbf{y})\Big|\, {dw}(\mathbf{y})\frac{dt}{t}\\
&\leq (\lambda_{\{n,\,j\}})^{-1} 2^{j+1} c\int_{\delta_s}^{\omega_s} \int \Big| \Psi_t(\mathbf{x},\mathbf{y})\Big| \, {dw}(\mathbf{y}) \frac{dt}{t}\\
&\leq C' (\lambda_{\{n,\,j\}})^{-1} 2^{j+1} c\int_{\delta_s}^{\omega_s}  \frac{dt}{t}\\
&\leq C{w}(Q_{\{n,\,j\}})^{-1},
\end{split}
\end{equation}
because $0<\omega_s\leq 3\delta_s$.

We turn to estimate the integrals over $[a_s,b_s]$.
Assume that
$$(\lambda_{\{n,\,j\}})^{-1} c\left|\int_{a_s}^{b_s} \int \Psi_t(\mathbf{x},\mathbf{y}) \chi_{\mathbf T_{\{n,\,j\}}}(t,\mathbf{y})(t^2(-{\Delta})e^{t^2{\Delta}}f)(\mathbf{y})\, {dw}(\mathbf{y})\frac{dt}{t}\right|>0$$
By Lemma \ref{key_lemma} for fixed $\mathbf{x}\in\Omega_j$ and $s\in\{1,2,{\dots},m\}$
either $\chi_{\mathbf T_{\{n,\,j\}}}(t,\mathbf{y})\equiv 0$ for all $t\in [a_s,b_s]$ and
$d(\mathbf{x},\mathbf{y})<t$ or $\chi_{\mathbf T_{\{n,\,j\}}}(t,\mathbf{y})\equiv 1$ for all  $t\in [a_s,b_s]$ and $d(\mathbf{x},\mathbf{y})<t$. So the  letter holds. This gives that for every $t\in [a_s,b_s]$ and $\mathbf{y}$ such that $d(\mathbf{x},\mathbf{y})<t$ we have that $(t,\mathbf{y})\notin \hat{\Omega}_{j+1}$. So there is $\mathbf{x}'$ (which depends on $(t,\mathbf{y})$) such that $d(\mathbf{y},\mathbf{x}')<4t$ and $\mathcal{M} f(\mathbf{x}')<2^{j+1}$. Note that $d(\mathbf{x},\mathbf{x}')<d(\mathbf{x},\mathbf{y})+d(\mathbf{y},\mathbf{x}')<5t$. Consequently, for every $t\in [a_s,b_s]$ we have
$$ 2^{j+1}\geq\mathcal{M}f(\mathbf{x}')\geq \sup_{d(\mathbf{x}',\mathbf z)<5t}|\Xi_tf(\mathbf z)|\geq |\Xi_tf(\mathbf{x})|.$$
Finally, in our case
\begin{equation}\label{atom2}\begin{split}
 (\lambda_{\{n,\,j\}})^{-1} c & \left|\int_{a_s}^{b_s} \int \Psi_t(\mathbf{x},\mathbf{y}) \chi_{\mathbf T_{\{n,\,j\}}}(t,\mathbf{y})(t^2(-{\Delta})e^{t^2{\Delta}}f)(\mathbf{y})\, {dw}(\mathbf{y})\frac{dt}{t}\right|\\
&=(\lambda_{\{n,\,j\}})^{-1} c\left|\int_{a_s}^{b_s} \int \Psi_t(\mathbf{x},\mathbf{y}) (t^2(-{\Delta})e^{t^2{\Delta}}f)(\mathbf{y})\, {dw}(\mathbf{y})\frac{dt}{t}\right|\\
&=(\lambda_{\{n,\,j\}})^{-1} c \left| \Xi_{a_s}f(\mathbf{x})-\Xi_{b_s}f(\mathbf{x})\right|\\
&\leq C{w}(Q_{\{n,\,j\}})^{-1},
\end{split}\end{equation}
where in the last equality we have used \eqref{differ}. The estimates  \eqref{atom1} and \eqref{atom2} give \eqref{atomanj}. Recall that $w(Q_{\{n,j\}})\sim w(B(\mathbf x_{\{n,j\}},7r_{\{n,j\}}\slash 2))$. Hence, from \eqref{atomanj}, \eqref{atombnj}, \eqref{support_a}, and \eqref{support_atoms1} we deduce Step 5. The proof of Theorem \ref{AtomicL2} is complete.
\end{proof}

{The next lemma will help us to remove the extra assumption that $f\in L^2({dw})$.

\begin{lemma}\label{L2Lemma} Assume that $f\in H^1_{{\rm max},P}$. Then  $P_tf\in L^2({dw})$ for every $t>0$ and
\begin{equation}\label{converge}\lim_{t\to 0} \|P_tf-f\|_{H^1_{{\rm max},P}}=0.
\end{equation}
\end{lemma}
\begin{proof}
 Proposition \ref{Poiss_new} implies that $P_tf\in L^2({dw})$. To prove \eqref{converge} we follow{,} e.g.{,} \cite[proof of (6.5)]{DZ2016}.

 First observe that there is a constant $C>0$ such that for every $A>0$ and $t>0$ we have
 \begin{align}\label{Poisson-diff}
\Big\|\sup_{\|\mathbf{x}-\mathbf{y}\|<s,s>At}|P_{t+s}f(\mathbf{y})-P_{s}f(\mathbf{y})|\Big\|_{L^1({dw}(\mathbf{x}))} \leq C A^{-1}\|f\|_{L^1({dw})}.
\end{align}
To see \eqref{Poisson-diff} fix $\mathbf z \in\mathbb{R}^N$. For $s>At$, thanks to \eqref{Poisson_low_up}, we have
\begin{align*}
&|p_{s+t}(\mathbf{y},\mathbf z)-p_{s}(\mathbf{y},\mathbf z)|=\left|\int_{0}^{t}\partial_{u}p_{s+u}(\mathbf{y},\mathbf z)\,du\right|\\
& \leq C \int_{0}^{t}\frac{1}{u+s+d(\mathbf{y},\mathbf z)}{w}(B(\mathbf z,s+u+d(\mathbf{y},\mathbf z)))^{-1}\,du \\
&\leq  C\int_{0}^{t}\frac{1}{s+d(\mathbf{y},\mathbf z)}{w}(B(\mathbf z,s+d(\mathbf{y},\mathbf z)))^{-1}\,du \\
&\leq  \frac{C}{A}\frac{s}{s+d(\mathbf{y},\mathbf z)}{w}(B(\mathbf z,s+d(\mathbf{y},\mathbf z)))^{-1}.
\end{align*}
Since $s+d(\mathbf{x},\mathbf z)\leq s+d(\mathbf{x},\mathbf{y})+d(\mathbf{y},\mathbf z)\leq s+\|\mathbf{x}-\mathbf{y}\|+d(\mathbf{y},\mathbf z)\leq 2(s+d(\mathbf{y},\mathbf z))$, we obtain
\begin{equation}\begin{split}\sup_{\|\mathbf{x}-\mathbf{y}\|<s}|p_{s+t}(\mathbf{y},\mathbf z)-p_{s}(\mathbf{y},\mathbf z)|
&\leq  \frac{C}{A}\frac{s}{s+d(\mathbf{x},\mathbf z)}{w}(B(\mathbf z,s+d(\mathbf{x},\mathbf z)))^{-1},\end{split}\end{equation}
which implies \eqref{Poisson-diff}.

In order to finish the proof of  \eqref{converge} assume that $f\in H^1_{{\rm max},P}$. Then

\begin{align*}
\left\|P_{t}f-f\right\|_{H^1_{{\rm max},P}}
&\leq \Big\|\sup_{\|\mathbf{x}-\mathbf{y}\|<s,s>At}|P_{t+s}f(\mathbf{y})-P_{s}f(\mathbf{y})|\Big\|_{L^1({dw}(\mathbf{x}))}\\
&\ +\Big\|\sup_{\|\mathbf{x}-\mathbf{y}\|<s,s \leq At}|P_{t+s}f(\mathbf{y})-P_{s}f(\mathbf{y})|\Big\|_{L^1({dw}(\mathbf{x}))} \\
&\leq C A^{-1}\|f\|_{L^1({dw})}
+\Big\|\sup_{\|\mathbf{x}-\mathbf{y}\|<s,s \leq At}|P_{s+t}f(\mathbf{y})-f(\mathbf{x})|\Big\|_{L^1({dw}(\mathbf{x}))}\\
&+\Big\|\sup_{\|\mathbf{x}-\mathbf{y}\|<s,s \leq At}|
P_{s}f(\mathbf{y})-f(\mathbf{x})|\Big\|_{L^1({dw}(\mathbf{x}))}\\
&\leq CA^{-1}\|f\|_{L^1({dw})}+2\Big\|\sup_{\|\mathbf{x}-\mathbf{y}\|<s,s \leq (A+1)t}|P_{s}f(\mathbf{y})-f(\mathbf{x})|\Big\|_{L^1({dw}(\mathbf{x}))}.
\end{align*}
Fix $\varepsilon>0$ and take $A=C\varepsilon^{-1}$. Corollary \ref{PoissonConv}  implies
  $$\lim_{t \to 0}\sup_{\|\mathbf{x}-\mathbf{y}\|<s,\, s \leq (A+1)t}|P_{s}f(\mathbf{y})-f(\mathbf{x})|=0 \ \
 \text{ \rm for almost every} \ \mathbf x\in\mathbb R^N.$$
 Since
$
\sup_{\|\mathbf{x}-\mathbf{y}\|<s,s \leq (A+1)t}|P_{s}f(\mathbf{y})-f(\mathbf{y})| \leq 2\mathcal{M}_Pf(\mathbf{x})\in L^1({dw}(\mathbf{x})),
$
the proof is complete by applying the Lebesgue dominated   convergence  theorem.
\end{proof}

Having Lemma \ref{L2Lemma} we are in a position to complete the proof of the atomic decomposition of $H^1_{{\rm max},P}$ functions. This is stated in the theorem below.

\begin{theorem}\label{AtomPoisson}
  There is a constant $C>0$ such that every function  $f\in H^1_{{\rm max},P}$ can be written as
  $$ f=\sum \lambda_j a_j,$$
where  $a_j$ are $(1,\infty, M)$-atoms, $\sum |\lambda_j|\leq C \|\mathcal{M}_P f\|_{L^1({dw})}$.
\end{theorem}
\begin{proof}
Take a sequence $t_n\to 0$, $n=0,1,{\dots},$  such that $\|P_{t_0}f\|_{H^1_{{\rm max},P}}\leq 2\| f\|_{H^1_{{\rm max},P}}$,
$\|P_{t_{n+1}}f-P_{t_n}f\|_{H^1_{{\rm max},P}}\leq 2^{-n}\| f\|_{H^1{{\rm max}, P}}$. Then $f=P_{t_0}f+\sum_{n=1}^\infty (P_{t_{n}}f-P_{t_{n-1}}f)=g_0+\sum_{n=1}^\infty g_n$, with convergence in $L^1({dw})$. The functions $g_n\in L^2({dw})\cap H^1_{{\rm max},P}$, so, by Theorem \ref{AtomicL2}
 they admit atomic decompositions into $(1,\infty,M)$-atoms with the required control of their atomic norms.
\end{proof}
The following theorem is a direct consequence of Lemma \ref{PoissonHeat} and Theorem \ref{AtomPoisson}.
\begin{theorem}
  There is a constant $C>0$ such that every element $f\in H^1_{{\rm max},{H}}$ can be written as
  $$ f=\sum \lambda_j a_j$$
where  $a_j$ are $(1,\infty, M)$-atoms, $\sum |\lambda_j|\leq C \|\mathcal{M}_{H} f\|_{L^1({dw})}$.
\end{theorem}

We are in a position to complete the proof of Theorem \ref{main1}, by proving the following proposition, which is the converse to Proposition \ref{PropMain1Part1}.
\begin{proposition}\label{converse}
  Assume that $u_0$ is $\mathcal L$-harmonic and satisfies $u_0^*\in L^1({dw})$. Then there is a system $\mathbf{u}=(u_0,u_1,{\dots},u_N)\in\mathcal{H}^1_{k}$ such that
  $ \| \mathbf{u}\|_{\mathcal{H}^1_{k}}\leq C \| u^*_0\|_{L^1({dw})}.$
\end{proposition}
\begin{proof}
By Theorem \ref{LpBounds} we have $u_0(t,\mathbf{x})=P_tf_0(\mathbf{x})$, where $f_0\in L^1({dw})$. So $f_0\in H^1_{\max,\, P}$ and $\| f_0\|_{H^1_{\max, \, P}}=\| u^*_0\|_{L^1({dw})}$. Using  Theorem \ref{AtomPoisson} and then \eqref{inAtom} we obtain that $f_0\in H^1_{{\Delta}}$ and $\| f_0\|_{H^1_{{\Delta}}}\leq C \| u^*_0\|_{L^1({dw})}$.
\end{proof}

\section{Inclusion $H^1_{(1,q,M)}\subset H^1_{{\rm max},H}$}\label{SectionMax}

In this section we shall prove that for every integer $M\geq 1$ and every $1<q
\leq \infty$, we have $H^1_{(1,q,M)}\subset H^1_{{\rm max},H}$ and
$$ \| f\|_{H^1_{{\rm max},H}}\leq C_{M,q}\| f\|_{H^1_{(1,q,M)}}.$$
It suffices to establish that  there is a constant $C_{M,q}>0$ such that
$$ \| a\|_{H^1_{{\rm max},{H}}}\leq C_{M,q}$$
for every $a$ being $(1,q,M)$-atom.
Since every $(1,q,M)$-atom is automatically  $(1,q,1)$-atom, it is enough to consider $M=1$ only.

Assume that $a$ is a $(1,q,1)$-atom associated with a set $\mathcal B=\bigcup_{\sigma\in G} B(\sigma (\mathbf{y}_0), r)$. Then there is a function $b\in\mathcal{D}({\Delta})$ such that $a={\Delta}b$, $\text{\rm supp}\, {\Delta}^jb\subset\mathcal B$, $\| {\Delta}^jb\|_{L^q({dw})} \leq r^{2-2j} {w}(\mathcal B)^{\frac{1}{q}-1}$, $j=0,1$.
Set $u(t,\mathbf{x})=e^{t^2{\Delta}}a(\mathbf{x})$. Observe that
$$\|u^*\|_{L^q({dw})}\leq C_q\| a\|_{L^q({dw})}\leq {w}(\mathcal B)^{\frac{1}{q}-1}$$
 (see \eqref{star} for the definition of $u^*$).  Thus, by the doubling property of the measure ${dw}(\mathbf{x})\, d\mathbf{x}$ and the H\"older inequality,
$$ \int_{d(\mathbf{x},\mathbf{y}_0)\leq 8r} u^*(\mathbf{x})\, {dw}(\mathbf{x})\leq C'_q.$$

We turn to  estimate $u^*(\mathbf{x})$ on $d(\mathbf{x},\mathbf{y}_0)>8r$. Clearly,
\begin{equation}\begin{split}
u^*(\mathbf{x})
&\leq   \sup_{0<t<d(\mathbf{x},\mathbf{y}_0)\slash 4, \, d(\mathbf{x}',\mathbf{x})<t} |e^{t^2{\Delta}}{\Delta} b(\mathbf{x}')|+
 \sup_{t>d(\mathbf{x},\mathbf{y}_0)\slash 4,\,  d(\mathbf{x}',\mathbf{x})<t} |e^{t^2{\Delta}} {\Delta} b(\mathbf{x}')|\\
&= J_1(\mathbf{x})+J_2(\mathbf{x}).\\
\end{split}\end{equation}
Recall that $\| b\|_{L^1({dw})}\leq r^2$ and note that
$$e^{t^2{\Delta}}{\Delta} ={\Delta} e^{t^2{\Delta}}=\frac{d}{ds}e^{s{\Delta}}{\big|_{s=t^2}}.$$
To deal with $J_1$ we note that if $d(\mathbf{x}',\mathbf{x})<t\leq d(\mathbf{x},\mathbf{x}_0)\slash 4$, $d(\mathbf{x},\mathbf{y}_0)>4r$, and $d(\mathbf{y},\mathbf{y}_0)<r$, then $d(\mathbf{x}',\mathbf{y})\sim d(\mathbf{x},\mathbf{y}_0)$. So, using  \eqref{Gauss}, we have
$$\Big|\frac{d}{ds}h_s(\mathbf{x}',\mathbf{y})\Big|_{\big|{s=t^2}}\leq \frac{C}{t^2{w}(B(\mathbf{y}_0,d(\mathbf{y}_0,\mathbf{x})))}e^{-c'd(\mathbf{y}_0,\mathbf{x})^2\slash t^2}. $$
Hence,
$$ J_1(\mathbf{x})\lesssim {w}(B(\mathbf{y}_0,d(\mathbf{x},\mathbf{y}_0)))^{-1} \frac{r^2}{d(\mathbf{x},\mathbf{y}_0)^2}.$$
 In order to estimate $J_2$, we observe from \eqref{Gauss} that for $t>d(\mathbf{x},\mathbf{y})$ and $d(\mathbf{y},\mathbf{y}_0)<r<t$ we have
  $$\Big|\frac{d}{ds}h_s(\mathbf{x}',\mathbf{y})\Big|_{\big|{s=t^2}}\leq \frac{C}{t^2{w}(B(\mathbf{y}_0,d(\mathbf{y}_0,\mathbf{x})))} $$
Consequently,
$$ J_2(\mathbf{x})\lesssim {w}(B(\mathbf{y}_0, d(\mathbf{x},\mathbf{y}_0)))^{-1} \frac{r^2}{d(\mathbf{x},\mathbf{y}_0)^2}.$$
Now
\begin{equation*}
\begin{split}
 \int_{d(\mathbf{x},\mathbf{y}_0)>8r}u^*(\mathbf{x})\, {dw}(\mathbf{x})
&\lesssim \sum_{j=3}^\infty \int_{2^jr<d(\mathbf{x},\mathbf{y}_0)\leq 2^{j+1}r} \frac{r^2}{ {w}(B(\mathbf{y}_0, d(\mathbf{x},\mathbf{y}_0)))d(\mathbf{x},\mathbf{y}_0)^2}\, {dw}(\mathbf{x})\\
&\lesssim \sum_{j=3}^\infty 2^{-2j}=C.
\end{split}
\end{equation*}

\section{Square function characterization}\label{SectionSquare}

In this section we  prove Theorem \ref{main5}. More precisely  we show that the atomic Hardy space $H^1_{(1,2,M)}$ coincides with the Hardy space defined by the square function \eqref{square} with $Q_t=t\hspace{.25mm}\sqrt{-\Delta\,}\smash{e^{-\hspace{.25mm}t\hspace{.25mm}\sqrt{-\Delta}}}$. This is achieved by mimicking arguments in \cite{HMMLY}. The proof for $Q_t\hspace{-.25mm}=t^2\hspace{.25mm}(-\Delta)\hspace{.5mm}e^{\hspace{.5mm}t^2\Delta}$ is similar.
\medskip

{\bf Tent spaces $T_2^p$ on spaces of homogeneous type.}
The square function characterization of the Hardy space $H^1_{(1,2,M)}$ can be related with the so called tent space $T_2^1$. The tent spaces on Euclidean spaces were introduced  in \cite{CMS} and then extended on spaces of homogeneous type (see, e.g. \cite{Rus}). For more details we refer the reader to \cite{St2}.

For a measurable function $F(t, \mathbf{x})$ on $(0,\infty)\times\mathbb{R}^N$ let
$$\mathcal AF(\mathbf{x}) :=\Big( \int_0^\infty\int_{\|\mathbf{y}-\mathbf{x}\|<t} |F(t,\mathbf{y})|^2\frac{{dw}(\mathbf{y})}{{w}(B(\mathbf{x},t))}\frac{dt}{t}\Big)^{1\slash 2}.$$

\begin{definition}\normalfont
For $1\leq p<\infty$ the tent space $T_2^p$ is defined to be
$$T_2^p=\{F: \| F\|_{T_2^p}:=\|\mathcal AF\|_{L^p({dw})}<\infty\}. $$
\end{definition}
Clearly, by the doubling property,
\begin{equation}\label{T22}
\| F\|_{T_2^2}^2=\|\mathcal AF\|_{L^2({dw})}^2\sim \int_0^\infty \int_{\mathbb{R}^N} |F(t,\mathbf{y})|^2\frac{{dw}(\mathbf{y})dt}{t}.
\end{equation}

\begin{remark}\label{RemarkT22}
\normalfont
By~\eqref{L2backL2} and~\eqref{T22} the operator $\pi_{\Psi}$ maps continuously the space $T_2^2$ into $L^2({dw})$.

Furthermore, by \eqref{L2toL2},  if $F(t, \mathbf{x})=Q_tf(\mathbf{x})$ for  $f\in L^2({dw})$, then
$$ \| F\|_{T_2^2}=\| Sf\|_{L^2({dw})}\sim \| f\|_{L^2({dw})}.$$
and
 $f=c \pi_{\Psi}( F)$.
\end{remark}

The tent space $T_2^1$ on the space of homogenous type   admits the following  atomic decomposition (see{,} e.g.{,}~\cite{Rus}).

\begin{definition}\normalfont
A measurable function $A(t,\mathbf{x})$ is a {\it $T^1_2$-atom} if there is a ball $B\subset\mathbb{R}^N$ such that

$\bullet$  $ \text{supp}\, A\subset \widehat B$

$\bullet$ $\iint_{(0,\infty) \times \mathbb{R}^N} |A(t,\mathbf{x})|^2\, {dw}(\mathbf{x})\, \frac{dt}{t}\leq {w}(B)^{-1}.$
\end{definition}
A function $F$ belongs to $T_2^1$ if and only if there are   sequences $A_j$ of $T_2^1$-atoms and $\lambda_j\in\mathbb C$ such that
$$  \sum_j \lambda_j A_j=F,\ \ \ \sum_j |\lambda_j|\sim \| F\|_{T_2^1},$$
where the convergence is in $T_2^1$ norm and a.e.

The H\"older inequality immediately gives that there is a  constant $C>0$ such that for every function  $A(t,\mathbf{x})$ being a $T_2^1$-atom one has
$$ \| A\|_{T_2^1}\leq C.$$

Observe that for $f\in L^1({dw})$ the function $F(t, \mathbf{x})=Q_tf(\mathbf{x})$ is well defined. Moreover,
$\mathcal AF(\mathbf{x})=Sf(\mathbf{x})$ and $\| Sf\|_{L^1({dw})}=\| F\|_{T^1_2}$.

\

\begin{remark}\label{remarkAtom}
\normalfont
According to the proof of atomic decomposition of $T_2^1$ presented in \cite{Rus}, the function $\lambda_jA_j$ can be taken of the form $\lambda_j A_j(\mathbf{x},t) =\chi_{S_j}(\mathbf{x}, t)$, where $S_j$ are disjoint, $\mathbb{R}_+^{N+1}=\bigcup S_j$, and $S_j$ is contained in a tent $\widehat B_j$.

So, if $F\in T^1_2\cap T^2_2$, then $F$ can be decomposed into atoms such that
$ F(t,\mathbf{x})=\sum_{j} \lambda_j A_j(\mathbf{x},t)$ and the convergence is in $T_2^1$, $T_2^2$, and pointwise.
\end{remark}

\begin{lemma}\label{Ps}
The map $(P_s F)(t,\mathbf{x})=\int p_s(\mathbf{x},\mathbf{y}) F(t, \mathbf{y})\, {dw}(\mathbf{y})$ is bounded on $T_2^1$.
\\
Moreover, there is a constant $C>0$ independent of $s>0$ such that $\| P_sF\|_{T_2^1}\leq C\| F\|_{T_2^1}$.
\end{lemma}
\begin{proof}
Let $F(t,\mathbf{x})=\sum_j \lambda_j A_j(t,\mathbf{x})$ be an atomic decomposition of $F\in T_2^1$ as described above. Since $p_s(\mathbf{x},\mathbf{y})\geq 0$, it suffices to prove that there is a constant $C>0$ such that
$$ \Big\|P_s|A|\Big\|_{T_2^1}\leq C$$
for every atom $A$ of $T_2^1$. To this end let $B=B(\mathbf{x}_0, r)$ be a ball associated with $A$.
Obviously, $P_s|A|(t,\mathbf{x}')=0$ for $t>r$.

{\bf Case 1:} $s>r$. Then, by \eqref{Poisson_low_up} and the H\"older inequality,
\begin{equation*}\begin{split}
P_s|A|(t,\mathbf{x}')\leq \frac{Cs}{s+d(\mathbf{x}_0,\mathbf{x}')}\frac{{w}(B(\mathbf{x}_0, r))^{1\slash 2}}{ {w}(B(\mathbf{x}_0, s+d(\mathbf{x}_0,\mathbf{x}')))}
\Big(\int |A(t,\mathbf{y})|^2{dw}(\mathbf{y})\Big)^{1\slash 2}.
\end{split}
\end{equation*}
If $\|\mathbf{x}-\mathbf{x}'\|< t\leq r$, then $s+d(\mathbf{x}_0,\mathbf{x}')\sim s+d(\mathbf{x}_0,\mathbf{x})$, because, by our assumption,  $s>r$.
Hence,
\begin{equation*}
\begin{split}
\Big\| P_s|A|\Big\|_{T_2^1}
&\leq C\int \frac{s}{s+d(\mathbf{x}_0,\mathbf{x})} \frac{{w}(B(\mathbf{x}_0,r))^{1\slash 2}}{{w}(B(\mathbf{x}_0, s+d(\mathbf{x}_0,\mathbf{x})))} \\
&\ \ \ \ \  \times
\Big(\int_0^r\int_{\|\mathbf{x}-\mathbf{x}'\|<t} \int |A(t,\mathbf{y})|^2{dw}(\mathbf{y})\frac{{dw}(\mathbf{x}')dt}{{w}(B(\mathbf{x}, t))t}\Big)^{1\slash 2}{dw}(\mathbf{x})\\
&\leq  C\int \frac{s}{s+d(\mathbf{x}_0,\mathbf{x})} \frac{{dw}(\mathbf{x})}{{w}(B(\mathbf{x}_0, s+d(\mathbf{x}_0,\mathbf{x})))}\leq C,
\end{split}
\end{equation*}
where to get the second to last inequality we first integrated with respect to ${dw}(\mathbf{x}')$ and then used the definition of $T_2^1$-atom.

{\bf Case 2:} $s\leq r$. Recall that $P_s$ is a contraction on $L^2({dw})$.  Hence,
\begin{equation}\begin{split}
\|\mathcal AP_s|A|\|_{L^1(\mathcal{O}(B(\mathbf{x}_0, 4r)),\,{dw})}
&\leq C{w}(B(\mathbf{x}_0,r))^{1\slash 2} \|\mathcal AP_s|A|\|_{L^2({dw})}\\
&\leq C{w}(B(\mathbf{x}_0,r))^{1\slash 2} \|P_s|A|\|_{T_2^2}\\
&\leq C{w}(B(\mathbf{x}_0,r))^{1\slash 2} \||A|\|_{T_2^2}\leq C.\\
\end{split}
\end{equation}
If $d(\mathbf{x},\mathbf{x}_0)>4r$, $\|\mathbf{x}'-\mathbf{x}\|<t<r$, and $\|\mathbf{x}_0-\mathbf{y}\|<r$,
then $s+d(\mathbf{x}',\mathbf{y})\sim s+d(\mathbf{x},\mathbf{x}_0)$. Now we proceed as in Case 1 to get the required bound on $\mathcal{O}(B(\mathbf{x}_0, 4r))^c$.
\end{proof}

\begin{lemma}\label{Approx}
The family $P_s$ forms approximate of identity in $T_2^1$, that is,
$$ \lim_{s\to 0} \| P_sF-F\|_{T_2^1}=0.$$
\end{lemma}
\begin{proof} According to Lemma \ref{Ps}, it suffices to establish that for every $A$ being a $T_2^1$-atom we have
\begin{equation}\label{limPs}
\lim_{s\to 0}\| P_sA-A\|_{T_2^1}=\lim_{s\to 0}\|\mathcal A(P_sA-A)\|_{L^1({dw})}=0.
\end{equation}
Let $A$ be such an atom and let $B=B(\mathbf{x}_0,r)$ be its associated ball. To prove \eqref{limPs} it suffices to consider $0<s<r$.

If $d(\mathbf{x},\mathbf{x}_0)>4r$, $\|\mathbf{y}-\mathbf{x}_0\|<r$, and $\|\mathbf{x}-\mathbf{x}'\|<t<r$, then $s+d(\mathbf{x}',\mathbf{y})\sim d(\mathbf{x},\mathbf{x}_0)$, so
\begin{equation*}\begin{split}
|P_s A(t,\mathbf{x}')|\leq \frac{Cs}{s+d(\mathbf{x}_0,\mathbf{x})}\frac{{w}(B(\mathbf{x}_0, r))^{1\slash 2}}{ {w}(B(\mathbf{x}_0, s+d(\mathbf{x}_0,\mathbf{x})))}
\Big(\int |A(t,\mathbf{y})|^2{dw}(\mathbf{y})\Big)^{1\slash 2}.
\end{split}
\end{equation*}
Since $\text{supp}\, A\cap \{(t,\mathbf{x}'): \|\mathbf{x}'-\mathbf{x}\|<t<r\}=\emptyset$, we have
$$ |\mathcal A(P_s A-A)(\mathbf{x})|=|\mathcal A(P_s A)(\mathbf{x})|\leq  \frac{Cs}{s+d(\mathbf{x}_0,\mathbf{x})}\frac{1}{ {w}(B(\mathbf{x}_0, s+d(\mathbf{x}_0,\mathbf{x})))}.$$
Hence,
$$ \lim_{s\to 0} \int_{d(\mathbf{x},\mathbf{x}_0)>4r} |\mathcal A(P_sA-A)(\mathbf{x})|\, {dw}(\mathbf{x})=0.$$
We now turn to estimate $\|\mathcal A(P_sA-A)\|_{L^1(\mathcal{O}(B(\mathbf{x}_0, 4r)),{dw})}$. Observe that
$$ |(P_sA-A)(t,\mathbf{x}')|\leq 2\mathcal{M}_PA(t,\mathbf{x}') \ \ \text{\rm and} \ \|\mathcal{M}_PA(t,\mathbf{x}')\|_{L^2({dw}(\mathbf{x}'))}\leq C\| A(t,\mathbf{x}')\|_{L^2({dw}(\mathbf{x}'))}.$$
Moreover, $\lim_{s\to 0} \| P_s A(t,\mathbf{x}')-A(t,\mathbf{x}')\|_{L^2({dw}(\mathbf{x}'))}=0$ for almost every $t>0$.
Therefore, applying the H\"older inequality and \eqref{T22}, we have
\begin{equation*}\begin{split}
\limsup_{s\to 0} &\|\mathcal A(P_sA-A)\|_{L^1(\mathcal{O}(B(\mathbf{x}_0,4r)))}\\
&\leq \limsup_{s\to 0} C{w}(B)^{1\slash 2}  \|\mathcal A(P_sA-A)\|_{L^2(\mathcal{O}(B(\mathbf{x}_0,4r)))}\\
&\leq \limsup_{s\to 0} C{w}(B)^{1\slash 2}\Big(\int_0^r\int |P_sA(t,\mathbf{x})-A(t,\mathbf{x})|^2\frac{{dw}(\mathbf{x})\, dt}{t}\Big)^{1\slash 2}=0,
\end{split}\end{equation*}
where in the last equality we have used the Lebesgue dominated convergence theorem.\end{proof}

\begin{lemma}
For every positive integer $M$ there is a constant $C_M>0$ such that for every $a(\mathbf{x})$ being a $(1,2,M)$-atom if $F(t,\mathbf{x})= Q_ta(\mathbf{x})$,  then
$$\| F(t,\mathbf{x})\|_{T_2^1} \leq C_M.$$
\end{lemma}

\begin{proof}
Let $a$ be a $(1,2,M)$--atom, $M\geq 1$,  associated with a ball $B=B(\mathbf{x}_0,r)$. By definition  $a=\Delta_{k}^{M}b$ with ${\Delta}^\ell b$ (for $\ell =0,1,{\dots}, M$) satisfying relevant support and size conditions (see Definition \ref{def-atom}). By the H\"older inequality,
\begin{align*}
\|Sa\|_{L^1(\mathcal{O}(8B))}&  \lesssim \|Sa\|_{L^2(\mathcal{O}(8B))}{w}(\mathcal{O}(8B))^{1/2}  \lesssim 1.
\end{align*}
If $d(\mathbf{x},\mathbf{x}_0)>8r$ then choose $n \geq 3$ such that $2^{n}r \leq d(\mathbf{x},\mathbf{x}_0) < 2^{n+1}r$ and split the integral as below
\begin{align*}
Sa(\mathbf{x})^2&= \int \int_{t>\|\mathbf{x}-\mathbf{y}\|}|Q_ta(\mathbf{y})|^{2}w(B(\mathbf{y},t))^{-1}\,{dw}(\mathbf{y})\, \frac{dt}{t}\\&=\int_{0}^{2^nr/4} \int_{t>\|\mathbf{x}-\mathbf{y}\|}+\int_{2^nr/4}^{\infty} \int_{t>\|\mathbf{x}-\mathbf{y}\|}=I_1+I_2.
\end{align*}
 Define $a_1=\Delta_{k}^{M-1}b$. Then by the definition of the atom $\|a_1\|_{L^1({w})} \leq r^{2}$. Note that
\begin{align*}
Q_t(a)=Q_t({\Delta} a_1)=({\Delta} Q_t)(a_1)=t(\partial_tQ_t)^3(a_1).
\end{align*}
\textbf{Estimation for $I_1$.}
If $z \in\mathcal{O}(B)$ and $\|\mathbf{x}-\mathbf{y}\|<t\leq 2^nr/4$, then $2^n r \lesssim d(\mathbf z,\mathbf{y})$. Therefore, thanks to \eqref{Poisson_low_up} and \eqref{DtDxDyPoisson} with $m=3$, we have
\begin{align*}
|Q_ta(\mathbf{y})|^{2} &= \left|\int t(\partial^3_t)(p_t(\mathbf{y},\mathbf z))a_1(\mathbf z)\,{dw}(\mathbf z)\right|^{2} \\&\lesssim \left(\int  {d(\mathbf z,\mathbf{y})}^{-2}\frac{t}{t+d(\mathbf z,\mathbf{y})}V(\mathbf z,\mathbf{y},t+d(\mathbf z,\mathbf{y}))^{-1}|a_1(\mathbf z)|\,{dw}(\mathbf z)\right)^{2} \\&\lesssim
(2^nr)^{-4}\frac{t^2}{(2^nr)^2}{w}(B(\mathbf{x}_0,2^nr))^{-2}\|a_1\|_{L^1({dw})}^2.
\end{align*}
Consequently,
\begin{align*}
I_1 &\lesssim \left(\int_{0}^{2^nr}t\,dt\right) {w}(B(\mathbf{x}_0,2^nr))^{-2}\|a_1\|_{L^1({dw})}^2 (2^nr)^{-4} (2^nr)^{-2}  \lesssim 2^{-4n} {w}(B(\mathbf{x}_0,2^nr))^{-2}.
\end{align*}
\\
\textbf{Estimation for $I_2$.} In this case $t \geq 2^nr/4$, so thanks to~\eqref{DtDxDyPoisson} with $m=3$ we have
\begin{align*}
|Q_ta(\mathbf{y})|^{2} &= \left(\int t(\partial^3_t)(p_t(\mathbf{y},\mathbf z))a_1(\mathbf z)\,{dw}(\mathbf z)\right)^{2} \\&\lesssim \left(\int  {t}^{-2}\frac{t}{t+d(\mathbf z,\mathbf{y})}V(\mathbf z,\mathbf{y},t+d(\mathbf z,\mathbf{y}))^{-1}|a_1(\mathbf z)|\,{dw}(\mathbf z)\right)^{2} \\&\lesssim
t^{-4}{w}(B(\mathbf{x}_0,2^nr))^{-2}\|a_1\|_{L^1({dw})}^2.
\end{align*}
Consequently,
\begin{align*}
I_2 &\lesssim \left(\int_{2^nr/4}^{\infty}t^{-5}\,dt\right) {w}(B(\mathbf{x}_0,2^nr))^{-2}\|a_1\|_{L^1({dw})}^2  \lesssim  2^{-4n} {w}(B(\mathbf{x}_0,2^nr))^{-2}.
\end{align*}
Finally,
\begin{align*}
 \| Sa\|_{L^1(\mathcal{O}(8B)^c)}&\lesssim \sum_{n\geq 3} \int_{2^nr<d(\mathbf{x},\mathbf{x}_0)\leq 2^{n+1}r} 2^{-2n}{w}(B(\mathbf{x}_0, 2^nr))^{-1}{dw}(\mathbf{x})\lesssim 1.
\end{align*}
\end{proof}

\begin{proposition}
Let $M$ be a positive integer.
Assume that for $f\in L^1({dw})$ the function $F(t,\mathbf{x})=Q_tf(\mathbf{x})$ belongs to $T_2^1$. Then
there are $\lambda_j\in\mathbb C$ and   $a_j$ being $(1,2,M)$-atoms such that
$$ f=\sum_{j}\lambda_j a_j,$$
and
$$ \sum_{j}|\lambda_j|\leq C\|F\|_{T_2^1}.$$
The constant $C$ depends on $M$ but it is independent of $f$.
\end{proposition}
\begin{proof}

We start our proof under the  additional assumption $f\in L^2({dw})$.  Then $F(t, \mathbf{x})=Q_tf(\mathbf{x})\in T_2^1\cap T_2^2$.
The proof in this case  is the same as that of \cite[Theorem 4.1]{HMMLY}. The only difference is to control support of functions ${\Delta}^sb_j$. For the convenience of the reader we provide its sketch.

Let $F=\sum_{j} \lambda_j A_j$ be a $T_2^1$ atomic decomposition of the function $Q_tf(\mathbf{x})$
 as it is described in Remark \ref{remarkAtom}.
In particular, $\sum_j |\lambda_j|\leq C \| Sf\|_{L^1({dw})}$.
 Let $\Psi$ be chosen such that
$ \int_0^\infty \Psi_tQ_t\frac{dt}{t}$ forms a Calder\'on reproducing formula, with $\Psi={\Delta}^{M+1}\Psi^{\{1\}}$, where $\Psi^{\{1\}}$ is a radial $C^\infty$ function supported by $B(0,1\slash 4)$.
By Remark \ref{RemarkT22} we have
\begin{equation}\label{sumA}
f=\pi_\Psi F=\sum_j \lambda_j\pi_\Psi A_j
\end{equation}
and the series converges in $L^2({dw})$.
Let $B_j=B(\mathbf{y}_j,r_j)$ be a ball associated with $A_j$. Then $\text{supp}\, A_j\subset \widehat B_j$.

Set $a_j=\pi_{\Psi} (A_j)={\Delta}^M b_j$, where
$$ b_j=\int_0^\infty t^{2M} (t^2{\Delta}\Psi^{\{1\}}_tA)\frac{dt}{t}.$$
Clearly, $\text{supp}\, b_j\subset\mathcal{O}(B(\mathbf{y}_j, 2r_j))$. The same argument as in the proof of Lemma 4.11. in  \cite{HMMLY}  shows that for every $s=0,1,2,{\dots},M$, the function
$$ b_{j,s}={\Delta}^s b_j=\int_0^\infty t^{2M} ({\Delta}^s t^2L \Psi^{\{1\}}_tA)\frac{dt}{t}$$
is supported by  $\mathcal{O}(B(\mathbf{y}_j, 2r_j))$ and its $L^2(w)$-norm is bounded by $r^{2M-2s}{w}(B_j)^{-1\slash 2}$. Thus $a_j$ are proportional to $(1,2,M)$-atoms. In particular
$\| a_j\|_{L^1({dw})}\leq C$ and, consequently, the series  \eqref{sumA} converges in $L^1({dw})$.

To remove the additional assumption $f\in L^2({dw})$ we    use Lemma \ref{Approx}
together with the fact that $P_sf\in L^2({dw})$ for $f\in L^1({dw})$,  and apply the same arguments as those in the proof of Theorem \ref{AtomPoisson}.
\end{proof}

\end{document}